\documentclass{article}

\usepackage{microtype}
\usepackage{graphicx}
\usepackage{subfigure}
\usepackage{booktabs} 

\usepackage{hyperref}

\usepackage[accepted]{icml2024}

\usepackage{amsmath}
\usepackage{amssymb}
\usepackage{mathtools}
\usepackage{amsthm}
\usepackage{bbm}

\usepackage[capitalize,noabbrev]{cleveref}

\theoremstyle{plain}
\newtheorem{theorem}{Theorem}[section]
\newtheorem{proposition}[theorem]{Proposition}
\newtheorem{lemma}[theorem]{Lemma}
\newtheorem{corollary}[theorem]{Corollary}
\theoremstyle{definition}
\newtheorem{definition}[theorem]{Definition}
\newtheorem{assumption}[theorem]{Assumption}
\theoremstyle{remark}
\newtheorem{remark}[theorem]{Remark}

\newcommand{\bX}{\boldsymbol{X}}
\newcommand{\bx}{\boldsymbol{x}}
\newcommand{\bY}{\boldsymbol{Y}}

\newcommand{\bB}{\boldsymbol{B}}
\newcommand{\bF}{\boldsymbol{F}}
\newcommand{\bz}{\boldsymbol{z}}
\newcommand{\bZ}{\boldsymbol{Z}}

\usepackage[textsize=tiny]{todonotes}

\icmltitlerunning{Minimax Optimality of Score-based Diffusion Models}

\begin{document}

\twocolumn[
\icmltitle{Minimax Optimality of Score-based Diffusion Models: \\Beyond the Density Lower Bound Assumptions}



\icmlsetsymbol{equal}{*}

\begin{icmlauthorlist}
\icmlauthor{Kaihong Zhang}{equal,uiuc}
\icmlauthor{Caitlyn H. Yin}{equal,uiuc}
\icmlauthor{Feng Liang}{uiuc}
\icmlauthor{Jingbo Liu}{uiuc}
\end{icmlauthorlist}

\icmlaffiliation{uiuc}{Department of Statistics, University of Illinois Urbana-Champaign, Champaign, IL, USA}

\icmlcorrespondingauthor{Kaihong Zhang}{kaihong5@illinois.edu}
\icmlcorrespondingauthor{Caitlyn H. Yin}{heqiyin2@illinois.edu}
\icmlcorrespondingauthor{Feng Liang}{liangf@illinois.edu}
\icmlcorrespondingauthor{Jingbo Liu}{jingbol@illinois.edu}

\icmlkeywords{Diffusion model, Score estimation, Nonparametric statistics, Minimax rate, Kernel density estimator, Sampling}

\vskip 0.3in
]



\printAffiliationsAndNotice{\icmlEqualContribution} 

\begin{abstract}
We study the asymptotic error of score-based diffusion model sampling in large-sample scenarios from a non-parametric statistics perspective. We show that a kernel-based score estimator achieves an optimal mean square error of  $\widetilde{O}\left(n^{-1} t^{-\frac{d+2}{2}}(t^{\frac{d}{2}} \vee 1)\right)$ for the score function of $p_0*\mathcal{N}(0,t\boldsymbol{I}_d)$, where $n$ and $d$ represent the sample size and the dimension, $t$ is bounded above and below by polynomials of $n$, and $p_0$ is an arbitrary sub-Gaussian distribution. As a consequence, this yields an $\widetilde{O}\left(n^{-1/2} t^{-\frac{d}{4}}\right)$ upper bound for the total variation error of the distribution of the sample generated by the diffusion model under a mere sub-Gaussian assumption. If in addition, $p_0$ belongs to the nonparametric family of the $\beta$-Sobolev space with $\beta\le 2$, by adopting an early stopping strategy, we obtain that the diffusion model is nearly (up to log factors) minimax optimal. This removes the crucial lower bound assumption on $p_0$ in previous proofs of the minimax optimality of the diffusion model for nonparametric families.
\end{abstract}

\section{Introduction}
\label{sec_introduction}
Diffusion models have emerged as a powerful tool of generative models, demonstrating exceptional performance in a wide range of applications. Pioneering work by \citet{ho2020denoising} and \citet{dhariwal2021diffusion} contributed to the development of image generation. Recent advances have demonstrated the effectiveness of diffusion models in a variety of domains, as exemplified by state-of-the-art results in image and text generation \citep{ramesh2022hierarchical, nichol2021glide}, text-to-speech synthesis  \citep{Jeong2021DiffTTSAD, popov2021grad, huang2022prodiff}, and molecular structure modeling \citep{xu2023geometric, hua2023mudiff}. Their reach has extended to scientific fields such as neuroscience \citep{pinaya2022brain} and materials science \citep{manica2023accelerating}.

Score-based generative modeling, a specific family of diffusion modeling, uses learned score functions (i.e., gradients of the log probability density functions) to transform white noise into the target data distribution via solving a stochastic differential equation. More specifically, the forward process converts samples drawn from a data distribution, denoted as $p_0$ (such as natural images), into complete noise, while the reverse process effectively reverts complete noise back into samples from $p_0$. Implementing the reverse process requires approximating the score function, a task typically accomplished through training neural networks using a score matching objective, as seen in prior work such as \citet{hyvarinen2005estimation}, \citet{vincent2011connection} and \citet{song2019generative}. While the score-based generative model has demonstrated remarkable performance in numerous applications, there remain gaps in our theoretical understanding.

From a statistical perspective, score-based diffusion model can be framed as the estimation of an unknown distribution from random samples. Any inaccuracies in score estimation introduce errors into the subsequent diffusion process. This naturally leads to the following question:

\textit{Consider viewing the score-based diffusion model as an algorithm designed to generate samples from an unknown distribution, using information from finite training samples generated by the forward process.
Under what conditions does this algorithm attain statistically optimal error rate for a given training sample size?
}


The total error in the sample distribution generated by a diffusion model can be approximately attributed to three components: the error stemming from approximating the true score, the discretization error incurred when discretizing stochastic differential equations, and the error arising from the convergence of the forward process (reflecting the deviation between the initial distribution of the reverse process and the standard Gaussian). Existing literature often employs Pinsker's inequality and Girsanov's theorem to bound the total variation distance between the true and approximate reverse processes, as seen in \citet{chen2022sampling}, \citet{chen2023improved} and \citet{benton2023linear}.

However, existing results are limited in the following aspects:

\textbf{Uniform-in-time assumption on score estimation error.}
In \citet{chen2022sampling} and other works such as \citet{benton2023linear},
it is assumed that the distribution is arbitrary, but there is a uniform (in time) i.e., uniformly over $[0, T]$ upper bound on the score estimation error.
The uniform upper bound assumption fails to capture the observation in practice that the score error typically decreases in time, due to the smoothing effect of the Gaussian kernel (see our numerical experiment in Appendix~\ref{appendix_numerical}). Therefore, the bound in \citet{chen2022sampling} is not sharp enough to yield the minimax optimal rate.
They also did not discuss a natural question in large sample scenarios: the relationship between estimation error and sample size $n$.

\textbf{Strong assumptions on the true data distribution.} 
By making an assumption on the data distribution class, 
it is sometimes possible to resolve the above issue by 
characterizing the decay of the score error in time.
For example,
\citet{oko2023diffusion}
derived an estimation error bound for the Besov class of density function,
showing that the diffusion model is minimax as long as the distribution satisfies a density lower bound on its support. The density lower bound greatly simplifies the proof of the score estimation error bound;
however, it excludes natural distribution classes, such as multi-modal distributions or mixtures with well-separated components.
In fact, under the compact support and the density lower bound assumption in \citet{oko2023diffusion}, 
we can deduce that the true density satisfies the log-Sobolev inequality (LSI) by the 
Holley-Stroock perturbation principle \citep{holley1987logarithmic}.
It is known that under LSI, running the Langevin dynamics is sufficient for achieving statistical efficiency \citep{koehler2022statistical},
while the diffusion model is expected to achieve efficiency for a wider class of distributions, by introducing a class of smoothed distributions. 
In this regard, the theory of \citet{oko2023diffusion} has not uncovered the true strength of the diffusion model. This paper uses a different proof technique than \citet{oko2023diffusion} and does not have this limitation. 

Comparing the above two lines of work, it might seem that a strong distribution assumption is a price we need to pay for deriving sharp bounds for settings where the score error is not uniform in time (such as the nonparametric class).
In this paper, we show that, perhaps surprisingly, 
we can derive a general sampling error bound only using a Sub-Gaussian assumption,
which, when specialized to the nonparametric class (without the density lower bound assumption),
is nearly (up to polylog factors) minimax optimal.

\subsection{Main Contributions}

In this paper, we derive a time-dependent error bound for the score function $s_t(x) = \nabla \log p_t(x)$ (see Theorem~\ref{theorem2}). Utilizing the error bound for the score function, we are able to control the total variation distance between the true data distribution $p_0$ and the distribution of a sample from the diffusion model as described in Algorithm~\ref{algorithm1}. Our main results only require the first or both of the following assumptions on the ground true data distribution $p_0$. We will make these statements rigorous in Section~\ref{sec_Assumptions}.

\begin{itemize}
  \item[\textbf{A1}] The true data distribution $p_0$ is $\sigma_0$-Sub-Gaussian.
  \item[\textbf{A2}] The true data distribution $p_0$ belongs to the Sobolev class of density functions with the order of smoothness $\beta \le 2$.
\end{itemize}
In particular, under \textbf{A1} and \textbf{A2},
we show that the error in the distribution of the sample generated by the diffusion model coincides with the classical minimax optimal convergence rate of density estimation in nonparametric statistics \citep{tsybakov}.
This rate cannot be improved,
since sampling error upper bounds the density estimation error.
We summarize our main result informally as follows:
\begin{theorem}[Informal; see Theorem~\ref{theorem2} and Theorem~\ref{main_theorem2}]
    Suppose that $p_0$ satisfies assumption \textup{\textbf{A1}}, and $C>0$ is arbitrary.
    There exists a score estimator $\hat{s}_t(x)$ ($t>t_0$) such that for the early stopping time $t_0=n^{-C}$, The distribution of sample from Algorithm~\ref{algorithm1} differs from $p_{t_0}$ by at most ${\rm polylog}(n) n^{-1/2}{t_0}^{-\frac{d}{4}}$ in the total variation (TV) distance, where $p_{t_0}$ denotes the distribution of the forward process at time $t_0$.
    Furthermore, if $p_0$ also fulfils \textup{\textbf{A2}}, by choosing $t_0=n^{-\frac{2}{2\beta+d}}$,
    the distribution of the output sample differs from $p_0$ by at most 
    ${\rm polylog}(n) n^{-\frac{\beta}{2 \beta + d}}$ in TV.
\end{theorem}
In previous studies, several restrictive assumptions were made to construct minimax optimal score estimators for diffusion models. 
One of the key assumptions is that $p_0$ has compact support and whose density is bounded from below \citep{oko2023diffusion}.
As alluded to before, these assumptions are strong enough to guarantee LSI and hence fail to unveil the key advantage of the diffusion model over the Langevin dynamics \citep{koehler2022statistical}.
It is not at all obvious that the density lower bound assumption can be removed without impairing the convergence rate;
in fact, in some examples of nonparametric statistics, such as density estimation under Wasserstein error,
the minimax rate can indeed get worse without the density lower bound assumption
\citep{niles2022minimax}. 

In this work, 
we employ a truncated version of the score estimator similar to that of \citet{zhang1997empirical} in the context of empirical Bayes, 
but with more refined analysis to control the error in the regime of polynomially small $t$.
More precisely, we note that $s_t(x)=\nabla p_t(x)/p_t(x)$,
and estimate the numerator and the denominator by carefully constructed kernel density estimators.
We define the truncated score estimator by
\begin{equation}\label{new_score_estimator}
\hat s_t(x) :=\frac{\nabla\hat{p}_t(x)}{\hat p_t(x)} \mathbbm{1}_{\left\{\hat p_t(x) > \rho_n\right\}}
\end{equation}
where $\hat p_t(x)$ is a kernel density estimator of original data distribution $p_t(x)$ [see Appendix~\ref{appendix_construct_kernel} for details about properties of the kernel needed and proof of its existence]. 
We set our score estimator $\hat{s}_t(x)$ to be zero when the kernel density estimator $\hat{p}_t(x)$ is less than $\rho_n$, where $\rho_n := \frac{1}{n t^{d/2}}$ is a parameter decaying at the same rate as MSE of $\hat{p}_t$. 
If $\hat{p}_t(x)$ is larger than $\rho_n$, the MSE for  $\nabla\hat{p}_t(x)$ is proportional to $p_t(x)$
(Proposition~\ref{MSE_main}), so that $p_t(x){\rm MSE}(\hat{s}_t(x))\approx {\rm MSE}(\nabla\hat{p}_t(x))/p_t(x)$ is unaffected by the lower density.
If $\hat{p}_t(x)$ is smaller than $\rho_n$, there will be too few observations near $x$, and we cannot show that ${\rm MSE}[\nabla\hat{p}_t(x)]$ is proportional to $p_t(x)$ via Bernstein's concentration;
however, the sub-Gaussian assumption ensures that the contribution to the mean square score estimation error from this case is \emph{exactly} bounded at the level of the minimax rate.
See Section~\ref{sec_main_results} for details on the convergence of our truncated score estimator.

To summarize, we provide a convergence guarantee for the diffusion model in the sense of total variation (TV) distance 
under mild assumptions on the true distribution $p_0$.
Our main results Corollary~\ref{main_corollary} and Theorem~\ref{main_theorem2}, obtain a convergence rate of $n^{-\frac{\beta}{2\beta+d}} \mathrm{polylog}(n)$ for sampling from $p_0$ using the score-based diffusion model, which matches the minimax optimal rate in the classical nonparametric estimation theory \citep{tsybakov} up to logarithmic factors.
While the recent work of \citet{oko2023diffusion} showed a similar minimax optimality result, they made the restrictive assumptions that the data distribution $p_0$ has bounded support and a density lower bounded, i.e., $p_0(x) > C$ for some constant $C$. 
Here we relax the compact support assumption to sub-Gaussianity and completely remove the density lower bound.

\subsection{Prior Works}
Regarding the theoretical convergence results of score-based generative models, previous works, such as \citet{benton2023linear}, \citet{chen2023improved}, 
\citet{chen2022sampling},
and 
\citet{lee2023convergence},  often assume an oracle score estimator with bounded error,
without explicitly touching on the issue of the statistical error of the score estimator for finite samples.
For example, in \citet{chen2022sampling} it is assumed that their exists some estimator $s_{\theta}$ satisfying a uniform in time bound
\begin{align}
     \mathbb{E}_{q_{t_k}(\mathbf{x})}
    [\|\nabla \log q_{T-t_k}(\mathbf{x})  
     -s_\theta(\mathbf{x}, T-t_k)\|^2] \le \varepsilon_{\text{score}}^2,
     \label{e2}
\end{align}
where $t_k$ are the time discretization points,
and showed that the resulting reverse diffusion process produces a sample with error roughly the order of $\varepsilon_{\text{score}}$.
It is still not clear whether such $s_{\theta}$ exists,
or if the uniform in time assumption \eqref{e2} is valid, in specific settings.
In contrast, our study incorporates the statistical error in score estimation,
by applying a (refined) truncated score estimator similar to the one in \citet{zhang1997empirical}. This provides a unique perspective for analyzing score functions within the low-density regions.

Furthermore, the broader landscape of score-generative models relies on a variety of assumptions regarding the true distribution $p_0$, as discussed in prior works including  \citet{song2020score}, \citet{pidstrigach2022score}, and \citet{de2022convergence}. These studies assumed the true distribution $p_0$ is supported on a low dimensional manifold $\mathcal{M} \subset \mathbb{R}^d$, with a smooth density relative to the manifold. \citet{de2022convergence} specifically highlight that, without further assumptions, the generalization error exhibits exponential dependence on both the diameter of $\mathcal{M}$ and the reciprocal of the desired error margin. Further research in this domain, such as \citet{block2022generative}, explored the convergence of score estimation error and finite-sample bounds for sampling using Langevin diffusion.

Another line of work imposes the restrictive assumption of smoothness conditions such as a log-Sobolev Inequality (LSI) for the true distribution $p_0$, as discussed by \citet{lee2022convergence}, \citet{wibisono2022convergence}. These studies provide insights into polynomial convergence guarantees for score function estimation in TV distance under $L^2$-accurate error. Building upon this, \citet{chen2022sampling} and \citet{lee2023convergence} have extended the analysis by removing the LSI assumption, addressing the complexities of real-world data distributions, although they still require smoothness conditions and bounded support for the true distribution.

Closer to our line of work are the results of \citet{oko2023diffusion}, which provides a theoretical analysis of approximation and generalization abilities of score-based diffusion modeling under Besov function spaces and other specific assumptions, 
such as the true distribution being supported on a bounded domain and having a density lower bound.
They gave an upper bound for the total generalization error of $n^{-\frac{2 s}{d+2 s}} \log ^{18} n$ for the score network estimation, which matches the optimal convergence rate in the nonparametric function class setting up to logarithmic factors. Considering the neural network architecture's covering number, they minimized the empirical score-matching error.

After we finalized our manuscript, we noted a concurrent work by \citet{wibisono2024optimal} shares similarities with our work. They constructed a regularized version of the score estimator
\begin{equation*}
    \hat s_h^{\epsilon}(x) := \frac{\nabla \hat p_h (x)}{\max\left(\hat p_h(x) ,\epsilon\right)},
\end{equation*}
where $\hat p_h (x)$ is the kernel density estimator with the Gaussian kernel. Using this estimator, they achieved an optimal rate of $\tilde{O}(n^{-\frac{2}{d + 4}})$ for estimating the score function of an unknown probability distribution that is sub-Gaussian and has score function that is Lipschitz-continuous.

\subsection{Organization}
The rest of the paper is organized as follows. 
In Section~\ref{sec_background}, we establish the background of score-based generative models in our settings, detailing the necessary notations and assumptions. Section~\ref{main_result_section} presents our main results concerning the estimation error of the score function (over time $t$), and the bound for the TV distance between the data distribution and the distribution of sample from the diffusion model in Algorithm~\ref{algorithm1}. Section~\ref{sec_proof_overview} provides proof overviews for the theorems in Section~\ref{main_result_section}. We provide some discussions of this paper in Section~\ref{sec_discussion}. Finally, we briefly summarize our main findings in Section~\ref{sec_conclusion}.

\section{Background} \label{sec_background}
\subsection{Forward and Backward Processes}
\noindent\textbf{Forward Process.} Given the ground truth data distribution $p_0$, a forward process $(\mathbf \bX_t)_{t\in [0,T]}$ is defined as the solution for the following It\^o SDE: 
\begin{equation}\label{e3}
   \mathrm{d} \bX_t=\bF(\bX_t, t) \mathrm{d} t+g(t) \mathrm{d} \bB_t, \;\;\;\; \bX_0 \sim p_0,
\end{equation}
where $\bF(\cdot, t): \mathbb{R}^d \rightarrow \mathbb{R}^d$ is a vector-valued function called the drift coefficient, $g(t): \mathbb{R} \rightarrow \mathbb{R}$ is a scalar function called the diffusion coefficient, and $(\bB_t)_{t\in[0,T]}$ denotes a standard Brownian motion. 

In this paper, we focus on the case of Brownian motion process,
\begin{equation}\label{BM_process}
   \mathrm{d} \bX_t= \mathrm{d} \bB_t,\;\;\;\; \bX_0 \sim p_0,
\end{equation}
which is a special case of \eqref{e3}. In this example the drift term $\bF(\cdot, t) = \mathbf 0$ and the diffusion coefficient is a constant $g(t) = 1$. This SDE has an explicit solution
\begin{equation}\label{BM_process2}
    \bX_t = \bX_0 + \sqrt{t} \bZ. \;\;\;\; \bZ  \sim\mathcal{N}(\boldsymbol{0},\boldsymbol I_d)\perp \!\!\! \perp \bX_0\, .
\end{equation}
Then we have 
\begin{equation}
    \bX_t | \bX_0 \sim \mathcal{N}(\bX_0, t \,\mathbf{I}_d).
\end{equation}
Note that if we perturb our original data $\bX_0$ with Gaussian noise $\bB_T$ for a large enough time $T$, the marginal distribution of $\bX_T$ will have only a weak dependence on $\bX_0$ and will be approximately Gaussian distributed. We hereafter denote by $p_t(x)$ the probability density function of $\bX_t$ in \eqref{BM_process2}.

\textbf{Reverse Process for Sample Generation.} 
If we reverse the diffusion process \eqref{e3} in time, we can generate new samples from $p_0$. Importantly, denotes $\bY_t = \bX_{T-t}$, where $(\bX_{t})_{t \in [0,T]}$ is a solution to the SDE in \eqref{e3}, then $(\bY_{t})_{t \in [0,T]}$ satisfies:
\begin{align}\label{backward_sde}
    \mathrm{d} \bY_t &= \left[-\bF(\bY_t, T-t) + g^2(t) \nabla \log p_{T-t}(\bY_t)\right] \mathrm{d} t \nonumber \\
    &\quad + g(t) \mathrm{d} \bB_t, \quad \bY_0 = \bX_T \sim p_T,
\end{align}
where $(\bB_t)_{t\in[0,T]}$ is another independent standard Brownian motion. Here, $\nabla \log p_{t}(\bx)$ is called the \textit{score function} for $p_t$. If we run the process \eqref{backward_sde} with $\bY_0 \sim p_T$, then $\bY_T \sim p_0$ and we obtain a sample from the true data distribution. Setting $\bF(\cdot, t) = \mathbf 0$ and $g(t) = 1$, the corresponding reverse SDE for the Brownian motion process \eqref{BM_process} is
\begin{equation}\label{backward_sde_bm}
    \mathrm{d} \bY_t= \nabla \log p_{T-t}(\bY_t) \mathrm{d} t+\mathrm{d} \bB_t, \;\; \bY_0 = \bX_T \sim p_T.
\end{equation}
To compute the reverse SDE, we need to construct an estimator $\hat s_t(x)$ for the score function $\nabla \log p_t(x)$ from the dataset $\{x_i\}_{i=1}^n$. The \textit{score error} (SE) \citep{song2021scorebased} for the estimator $\hat s_t(\bx)$ is defined as:
\begin{equation}\label{score_error}
    \mathrm{SE}(\hat s) := \int_{t \in [0,T]}\mathbb{E}_{x \sim p_t}\left[\left\|\nabla \log p_t(x)-\hat s_t(x)\right\|_2^2\right]dt.
\end{equation}
The expectation of score error over the dataset $\mathbb{E}_{\{x_i\}_{i=1}^n}\left[\mathrm{SE}(\hat{s})\right]$ can be used to evaluate the performance of the proposed score estimator.
Once we have an estimator of the score function, we can plug it into the reverse SDE and solve it to get a sample from the original data distribution $p_0(x)$. 
\subsection{Sampling Method}
In order to sample from the unknown distribution $p_0$, in this paper we utilize the Brownian diffusion process $(\bX_t)_{t\in [0,T]}$ as defined in \eqref{BM_process} and the corresponding backward process $(\bY_t)_{t\in [0,T]}$ defined in \eqref{backward_sde_bm}. If we run the forward process $(\bX_t)_{t\in [0,T]}$ for a large enough time $T$, the distribution of $\bX_T$ can be approximated as Gaussian.
Suppose we obtain an estimator $\hat s_t$ for the score function $\nabla \log p_t(\bx)$ from the dataset $\{\bx_i\}_{i=1}^n$, we can run the backward SDE starting from Gaussian distribution with the unknown score function replaced by the score estimator:
\begin{equation}\label{output_of_algorithm}
    \mathrm{d} \widehat \bY_t= \hat s_{T-t}(\widehat \bY_t) \mathrm{d} t+\mathrm{d} \bB_t, \;\;\;\; \widehat \bY_0 \sim \mathcal{N}(0,T\boldsymbol{I}_d).
\end{equation}
To avoid learning the true score function, $\nabla \log p_0(x)$, which can lead to large score estimation errors due to low-density regions in $p_0$, we introduce an early stopping time $t_0$. Consequently, $\widehat \bY_{T-t_0}$ is a sample generated from the model.

We summarize the diffusion model based on the Brownian motion \eqref{BM_process2} in Algorithm~\ref{algorithm1}.
\begin{algorithm}
\caption{Brownian diffusion model}
\label{algorithm1}
\begin{algorithmic}[1]
    \STATE \textbf{Input:} The score estimator $\hat s_t$, early stopping time $t_0$, and a large enough time $T$.
    \STATE Sample $\bz$ from $\mathcal{N}(\boldsymbol{0}, T\boldsymbol{I}_d)$.
    \STATE Solve the backward SDE 
    \begin{equation*}
        \mathrm{d} \bY_t= \hat s_{T-t}(\bY_t) \mathrm{d} t+\mathrm{d} \bB_t 
    \end{equation*} with $\bY_0 = \bz$
    \STATE \textbf{Output:} $\bY_{T-t_0}$, a sample generated from the model.
\end{algorithmic}
\end{algorithm}
Note that discretizing the process for solving reverse SDE is not included here.
But in practice, we need to use a discrete-time approximation for the process of sampling \eqref{output_of_algorithm}. 
This problem has been discussed in many previous works. For example, \citet{chen2022sampling} in their Theorem 2 addressed the $L$-Lipschitz regularity condition on $\nabla \log q_t$ for all $t$, based on the TV metric. Theorem 1 of \citet{chen2023improved} also regarded the $L$ Lipschitz condition, to the Kullback–Leibler (KL) divergence metric, etc. A recent work \citet{benton2023linear} controlled the discretization error more precisely of the reverse SDE, based on tools from stochastic localization initially developed in \citet{el2022sampling}, \citet{montanari2023sampling}.
Without the Lipschitz continuity condition, \citet{oko2023diffusion} provided further insight regarding this issue, while considering other assumptions, including Besov spaces and density lower bounds.


\subsection{Connections with Ornstein–Uhlenbeck Process}
Many works on diffusion models utilize the Ornstein-Uhlenbeck (OU) process rather than the Brownian motion process described in \eqref{BM_process2}. The OU process is a special case of the It\^o stochastic differential equation in \eqref{e3}, with drift function $\bF(\bX_t,t) = -\bX_t $ and diffusion function $g(t) = \sqrt{2}$:
\begin{equation}
    \mathrm{d} \bX_t=-\bX_t \mathrm{d} t+ \sqrt{2} \mathrm{d} \bB_t, \;\;\;\; \bX_0 \sim p_0.
\end{equation}
The solution $(\bX_t)_{t\in[0,T]}$ admits a closed form:
\begin{equation}\label{e11}
    \bX_t = e^{-t}\bX_0 + \sqrt{1-e^{-2t}} \, \bZ \;\;\;\; \bZ  \sim\mathcal{N}(\boldsymbol{0},\boldsymbol I_d)\perp \!\!\! \perp \bX_0\, .
\end{equation}
As stated previously in \eqref{BM_process2}, the solution to the Brownian motion process $(\bY_t)_{t \in [0,T]}$ satisfies:
\begin{equation}\label{e12}
    \bY_t = \bY_0 + \sqrt{t} \bZ \;\;\;\; \bZ  \sim\mathcal{N}(\boldsymbol{0},\boldsymbol I_d)\perp \!\!\! \perp \bY_0.
\end{equation}
If we set $\bX_0 = \bY_0$ and make the time transformation $t = e^{2s} - 1$, then it follows that $e^s \bX_s \overset{\text{d}}{=} \bY_t$.  Therefore, the Ornstein-Uhlenbeck process can be viewed as a time transformation and scaled version of the Brownian motion process. For the analysis in this paper, we will utilize the Brownian motion process given in \eqref{BM_process2}. The results can then be generalized to the Ornstein-Uhlenbeck process due to the equivalence of these two processes.

\section{Main Results}\label{main_result_section}

In this section, we present our main results of score estimation error bound and the optimal convergence rate for diffusion models. 

\subsection{Notation}\label{sec_notation}
Here we outline notations and definitions used throughout the paper (see details in Appendix~\ref{appendix_notation}). $p_0: \mathbb{R}^d \rightarrow \mathbb{R}$ represents the true data distribution's density, with $p_t(x) := p_0 * \phi_t (x)$ (convolution with Gaussian density $\phi_t$ of $\mathcal{N}(0, t\boldsymbol{I}_d)$). 
The term ``polylog'' refers to $\mathrm{polylog}(n):= (\log n) ^C$ for some constant $C$.

The Fourier transform of function $f$ is given by:
\begin{equation*}
    \mathcal{F}[f](\omega)=\int_{\mathbb{R}^d} f(x) e^{-i\langle\omega, x\rangle} d x,
\end{equation*}
where $\langle\omega, x\rangle$ is the inner product of $\omega$ and $x$ in $\mathbb{R}^d$. For multi-index $\alpha$ and a vector $\omega$, $\omega^\alpha := \prod_{i=1}^d w_i^{\alpha_i}$.
We use $TV(p, q)$ to denote the total variation distance between two probability distributions $p$ and $q$. Similarly, $\mathrm{D}_{\mathrm{KL}}(p \| q)$ denotes KL divergence.

\subsection{Assumptions}\label{sec_Assumptions}
In this subsection, we will introduce the assumptions that are imposed on the true data distribution $p_0$ in our analysis. Roughly speaking, the data distribution is assumed to be sub-Gaussian and in the Sobolev class of densities.

\begin{assumption}\label{assumption1}
    The true distribution $p_0$ is $\sigma_0$-Sub-Gaussian.
\end{assumption}

While the sub-Gaussian random variables in $1$-dimension have been well-studied, it is still necessary to clarify the definition of a sub-Gaussian random vector in higher dimensions. A natural approach to define the Sub-Gaussian distribution in higher dimensions is through the projections onto lines.

\begin{definition}[Sub-Gaussian random vectors, \citet{vershynin_2018}]\label{Sub-Gaussian_def}
     A random vector $X \in \mathbb{R}^n$ is said to be Sub-Gaussian if the one-dimensional marginals $\langle X, v \rangle$ are Sub-Gaussian random variables for all $v \in \mathbb{R}^n$. The Sub-Gaussian norm of $X$ is defined as
\begin{equation}
\|X\|_{\psi_2} = \sup_{v \in S^{n-1}} \|\langle X, v \rangle\|_{\psi_2},
\end{equation}
where $S^{n-1}$ is the unit sphere in $\mathbb{R}^{d}$ and
the Orlicz norm of random variable $Y$ with respect to the function $\psi_2(x) = e^{x^2} - 1$ is defined as:
\begin{equation}
\|Y\|_{\psi_2} = \inf\left\{ C > 0 : \mathbb{E}\left[\psi_2\left(\frac{|Y|}{C}\right)\right] \leq 1 \right\}.
\end{equation}
We say a random vector $X \in \mathbb{R}^d$ is $\sigma$-Sub-Gaussian if $\sigma :=\|X\|_{\psi_2} < \infty$.
\end{definition}
\begin{remark}
    The sub-Gaussian assumption is relatively mild. 
    It is well known that a Log-Sobolev inequality (LSI) implies sub-gaussian decay of the tails (see \citet{vershynin_2018}), but the reverse implication is not true (consider for example a distribution supported on the union of two disjoint intervals). 
    The LSI is a common assumption on true data distribution in many prior works, such as  \citet{lee2022convergence,wibisono2022convergence}.
    In the work of \citet{oko2023diffusion}, the true data distribution is assumed to have bounded support and density bounded below, which implies the LSI by the Holley-Stroock principle \cite{holley1987logarithmic}. However, the assumption that the true density is bounded below is unrealistic in practice, as we would expect the true data distribution to have many low-density regions. This can lead to challenges in score estimation, as discussed in \citet{song2019generative}. 
\end{remark}
As we will see in Theorem~\ref{theorem2}, it suffices to derive an estimation error bound for the score function $s_t(x)$ only under Assumption~\ref{assumption1}. 
This in turn implies the minimax optimal rate for the nonparametric class 
under the additional Sobolev class assumption.
\begin{assumption}\label{assumption2}
    The true distribution $p_0$ belongs to the Sobolev class of density \cite{tsybakov} with $\beta \le 2$. Specifically, for $\beta, L \in \mathbb{R}_+$, the Sobolev class of density is defined as followed:
\begin{multline*}
\mathcal{P}_{\mathcal{S}}(\beta, L) = \Big\{ p \in \mathcal{L}^1(\mathbb{R}^d) \, | \, p \ge 0, \, \int p = 1, \\
\forall \alpha \text{ with } \sum_{i=1}^d \alpha_i = \beta, \int |\omega^\alpha|^{2} |\mathcal{F}[p](\omega)|^2 \, d\omega \le (2\pi)^d L^2 \Big\}.
\end{multline*}
\end{assumption}
Note that instead of using the common definition of the Sobolev class, this definition utilizes the Fourier transformation of the density $p$, thereby allowing each $\alpha_i$ to take values not only as integers but also as positive real numbers.

\subsection{Analysis of Score Estimation Error}\label{sec_main_results}
As illustrated in Section~\ref{sec_background}, the performance of the score-based diffusion model of Algorithm~\ref{algorithm1} highly depends on the estimation accuracy of the score function $s_t(x) = \nabla \log p_t(x)$ at each time $t \ge t_0$. The following theorem provides an upper bound for score estimation error. Notably, the score estimation error decays as time increases, implying that adding sufficient noise to the samples makes the score learning procedure easier. 
\begin{theorem}\label{theorem2}
Denote the true score function $s_t(x) := \nabla \log p_t(x)$. Suppose that Assumption~\ref{assumption1} holds. Then there exists a score estimator $\hat s_t(x)$ constructed from i.i.d samples $\{x_i\}_{i=1}^n \sim p_0$ such that for any constants $T_1,T_2>0$, let $t_0 := n^{-T_1}$ and $T :=n^{T_2}$. Then for any $t_0 < t < T$, 
\begin{multline}
    \mathbb{E} \Bigg[\int_x \left\| \hat{s}_t(x) - s_t(x) \right\|^2 p_t(x) \, dx \,\Bigg] \label{main_theorem}\\
    \lesssim \mathrm{polylog}(n) \, n^{-1} t^{-\frac{d+2}{2}} \left(t^{\frac{d}{2}}+ \sigma_0^d\right),
\end{multline}
where $\lesssim$ hide the constant that does not depend on time $t$ and sample size $n$.
\end{theorem}
The formal proof can be found in Appendix~\ref{appendix_theorem2}.

\begin{remark}
    From this theorem, we can see that the use of the early stopping technique is crucial for ensuring that the analysis proceeds correctly. If we set \(t_0 = 0\), the upper bound on the right-hand side in Theorem~\ref{theorem2} would blow up, resulting in a very large score estimation error. This is because, without early stopping, the estimation error for the score function \( \nabla \log p_0(x) \) at \(t = 0\) would be excessively high due to the low-density regions in \(p_0\). We can address this issue by introducing an early stopping time \(t_0\). Properly setting the early stopping time \(t_0\) is essential to achieve the desired minimax rate, as will be shown in Corollary~\ref{main_corollary} and Theorem~\ref{main_theorem2}.
\end{remark}

After integrating over time in Theorem~\ref{theorem2}, we will get the following result, and the proof is provided in Appendix~\ref{appendix_main_corollary}.
\begin{corollary}\label{main_corollary}
Suppose that $p_0$ satisfies Assumption~\ref{assumption1}, and $T_1,T_2>0$ is arbitrary.  
\begin{enumerate}
    \item For $t_0 = n^{-T_1}$ and $T = n^{T_2}$, we have
    \begin{equation}
\begin{split}
    \int_{t=t_0}^T \mathbb{E} \bigg[ & \int_x\left\|\hat{s}_t(x)-s_t(x)\right\|^2 p_t(x) \, dx  \bigg] \, dt \\
    & \lesssim \operatorname{polylog}(n) \,n^{-1}t_0^{-d/2}.
\end{split}
\end{equation}
    \item In particular, for any $\beta>0$, let $t_0 = n^{-\frac{2}{2\beta+d}}$ and $T = n^{T_2}$, we have
    \begin{equation}
\begin{split}
    \int_{t=t_0}^T \mathbb{E} \bigg[ & \int_x\left\|\hat{s}_t(x)-s_t(x)\right\|^2 p_t(x) \, dx  \bigg] \, dt \\
    & \lesssim \operatorname{polylog}(n) \,n^{-\frac{2 \beta}{2 \beta+d}}.
    \label{e18}
\end{split}
\end{equation}

\end{enumerate}
\end{corollary}

In the first part of Corollary~\ref{main_corollary}, The upper bound for the cumulative score error is given by $n^{-1} t_0^{-d/2}$, where $t_0$ is an early stopping time that can be freely chosen. 
Notably,  Corollary~\ref{main_corollary} does not impose any assumptions on $p_0$ about nonparametric class and requires only Sub-Gaussianity. However, in the further analysis in Theorem~\ref{main_theorem2} of the TV error of the diffusion model, if we further assume that $p_0$ 
is in the Sobolev class of density with smoothness parameter $\beta$ (Assumption~\ref{assumption2}), setting $t_0 = n^{-\frac{2}{2\beta+d}}$ can adjust the error bound to $n^{-\frac{2\beta}{2\beta+d}}$.
Remark that if we choose $t_0\le n^{-\frac{2}{2\beta+d}}$ (as in \cite{oko2023diffusion}),
then the optimal rate may no longer be achieved in \eqref{e18} without further assumptions such as the density lower bound.

\subsection{Analysis of Estimation Error for Diffusion Model}\label{sec_TV_results}
In this section, we analyze the estimation error of the diffusion model in Algorithm~\ref{algorithm1}, based on the score estimation error results from Section~\ref{sec_main_results}. The following theorem concerns the bound of TV distance between the true data distribution $p_0$ and the samples generated by Algorithm~\ref{algorithm1}.
\begin{theorem}\label{main_theorem2}
    Suppose that Assumption~\ref{assumption1} and Assumption~\ref{assumption2} holds. Let $\widehat \bY_{T-t_0}$ be the output of the diffusion model in Algorithm~\ref{algorithm1} at time $T-t_0$. Let $t_0 = n^{-\frac{2}{2\beta+d}}$ and $T = n^{\frac{2\beta}{2\beta+d}}$, then there exists a score estimator $\hat s_t$ such that
    \begin{equation}
    \mathbb{E}\left[\mathrm{TV}\left(\bX_0, \widehat \bY_{T-t_0}\right)\right] \lesssim \mathrm{polylog}(n) \, n^{-\frac{\beta}{2\beta+d}}.
    \end{equation}
\end{theorem}

\begin{remark}
    This sampling error rate coincides with the classical minimax rate in nonparametric density estimation (up to logarithmic factors) 
\citep{stone1980optimal,stone1982optimal,tsybakov}, hence must be optimal. The sampling distribution $\widehat \bY_{T-t_0}$ is considered minimax optimal because the sampling error is lower-bounded by the estimation error. To see this, suppose there exists a sampler $\widehat \bZ$ that produces samples with a smaller TV distance $\mathbb{E}\left[\mathrm{TV}\left(\bX_0, \widehat \bZ\right)\right]$ than the minimax rate. We could then generate many independent samples from $\widehat \bZ$ to construct a density estimator for $p_0$ with an error smaller than the minimax rate, leading to a contradiction. Therefore, the sampling error cannot be smaller than the density estimation error.
\end{remark}

\begin{remark}
    In the proof of Theorem~\ref{main_theorem2} (see Section~\ref{subsec_proof_thm2}), we only use the Sobolev class assumption (Assumption~\ref{assumption2}) when controlling the early stopping error $\mathrm{TV}\left(\bX_0, \bX_{t_0}\right)$. Therefore, under only the sub-Gaussian assumption, we can generalize this theorem to other nonparametric classes, provided we can control the early stopping error: $\mathrm{TV}\left(\bX_0, \bX_{t_0}\right)$. Using part 1 of Corollary~\ref{main_corollary} to control the error of score estimation, the sampling error of the diffusion model for a general sub-Gaussian $p_0$ can be bounded as follows:
    \begin{multline*}
        \mathbb{E}\left[\mathrm{TV}\left(\bX_0, \widehat \bY_{T-t_0}\right)\right] \\
        \lesssim \inf_{t_0} \left(\mathrm{TV}\left(\bX_0, \bX_{t_0}\right) + \mathrm{polylog}(n) \, n^{-1/2}t_0^{-d/4}\right).
    \end{multline*}
    In Theorem~\ref{main_theorem2}, under the Sobolev class assumption, the desired rate is obtained by $t_0 =n^{- \frac{2}{2\beta+d}}$.
\end{remark}

\section{Proof Overview}\label{sec_proof_overview}
We now provide the proof sketches for Theorem~\ref{theorem2} and Theorem~\ref{main_theorem2}. Formal proofs are provided in Appendix~\ref{appendix_theorem2}  and Appendix~\ref{sec_proof_theorem3}, respectively.

\subsection{Proof sketch of Theorem~\ref{theorem2}}
\noindent\textbf{Construction of the kernel score estimator.}
We first construct the score estimator for proving Theorem~\ref{theorem2}. The score function associated with perturbed data distribution $p_t$ is $s_t(x) =\nabla \log p_t(x) =\frac{\nabla p_t(x)}{p_t(x)}$. It is natural to first construct a density estimator $\hat p_t(x)$ for $p_t(x)$, then use the plug-in estimator $\hat s_t(x) = \frac{\nabla \hat p_t(x)}{\hat p_t(x)}$ to estimate the score function. Here, we will use the kernel density estimator (KDE) to estimate $p_t(x)$. Specifically, given samples $X_i \stackrel{\text{i.i.d.}}{\sim} p_0$ and $Z_{i,t} \stackrel{\text{i.i.d.}}{\sim} \mathcal{N}\left(0,tI_d\right)$, by definition of the forward process in \eqref{BM_process}, $X_i^t = X_i + Z_{i,t}\stackrel{\text{i.i.d.}}{\sim} p_t$. Given a kernel $K_d$ with order $\ell \ge 1$ (Definition 1.3 of \citet{tsybakov}, also see Definition~\ref{kernel_order_def}) and with bounded support $[-1,1]^d$ (we can construct a kernel satisfying these properties using the Legendre polynomials, see Lemma~\ref{Bound_on_kernel} for details), the KDE for $p_t(x)$ is defined as follows:
\begin{equation*}
    \hat{p}_t(x) = \frac{1}{nh^{d}}\sum_{i=1}^n K_d\left(\frac{x-X_i^t}{h}\right),
\end{equation*}
where $h$ is the bandwidth to be specified later.
Naively estimating the score by $\frac{\nabla \hat{p}_t(x)}{\hat{p}_t(x)}$ would lead to poor performance for $x$ such that $\hat p_t(x)$ is small. 
To avoid this issue, we introduce the truncated version of the kernel score estimator: for any time $t$, let the threshold be $\rho_n:=\frac{1}{n t^{d / 2}}$. The truncated kernel score estimator is defined as follows:
\begin{equation}\label{score_estimator_eq}
   \hat{s}_t(x):=\frac{\nabla \hat{p}_t(x)}{\hat{p}_t(x)} \mathbbm{1}_{\left\{x\colon \hat{p}_t(x) \geq \rho_n\right\}}.
\end{equation}

\noindent\textbf{Mean Squared Error Analysis}. The pointwise MSE for $\hat p_t(x)$ and $\nabla \hat p_t(x)$ are defined as follows:
\begin{align*}
    \mathrm{MSE}(\hat p_t(x)) &:= \mathbb{E}\left[| \hat{p}_t(x) -  p_t(x) |^2\right],\\
    \mathrm{MSE}(\nabla\hat{p}_t(x))&:= \mathbb{E} \left[\|\nabla \hat p_t(x) - \nabla p_t(x)\|^2\right].
\end{align*}
We follow the common framework of bias-variance trade-off \citet{tsybakov} with slight modifications, which allows us to obtain a sharper bound that depends on the location $x$. The result of MSE bounds for $\hat{p}_t$ and $\nabla \hat{p}_t$ is summarized in the following proposition, the proof can be found in Appendix~\ref{appendix_MSE}.
\begin{proposition}[See also Proposition~\ref{MSEphat}, \ref{MSEphatprime}]\label{MSE_main}
    If we take the bandwidth $h = C\,\sqrt{\frac{t}{\log n}}$ \,for some constant $C$, and the order of kernel $\ell = \log n$, then
    \begin{align*}
        \mathrm{MSE}(\hat p_t(x)) 
        &\lesssim  \;\mathrm{polylog}(n)\left(\tfrac{p_t^*(x)}{n t^{\frac{d}{2}}}  + (\log n)^{-\log n }\right), \\
        \mathrm{MSE}(\nabla\hat{p}_t(x)) &\lesssim  \;\mathrm{polylog}(n)\left(\tfrac{p_t^*(x)}{n t^\frac{d+2}{2}} + (\log n)^{-\log n } t^{-1}\right).
    \end{align*}
    where $p_t^*(x):=\sup_{\|\lambda\|_{\infty} < h}p_t(x+\lambda)$ is the local maximum of the density function $p_t$.
\end{proposition}
\begin{remark}
    The bound for MSE depends on the point $x$, dimension $d$, which allows us to obtain a sharper bound in the low-density area of $p_t$. 
\end{remark}

With these results, we can proceed to derive the upper bound \eqref{main_theorem} in Theorem~\ref{theorem2}. The main challenge is to control the error in low-density areas. To address this, we consider the following three cases.
In the first case where $p_t(x) > \rho_n \log^c {n}$ for some constant $c$, the density is relatively high. Using the concentration property of the KDE $\hat p_t(x)$ (see Lemma~\ref{high_probability_bound}), we find that with high probability, $p_t(x)$ and $\hat p_t(x)$ will approximately have the same order. As a result, $\hat p_t(x) > \rho_n$, which implies $\hat s_t(x) = \frac{\nabla \hat p_t(x)}{\hat p_t(x)}$. Then by triangle inequality (See lemma~\ref{score error upper bound}),
\begin{align*}
    &\|\hat{s}_t(x) - s_t(x)\|^2 = \left\|\frac{\nabla \hat{p}_t(x)}{\hat{p}_t(x)} - \frac{\nabla p_t(x)}{p_t(x)}\right\|^2 \\
    &\lesssim \frac{\|\nabla\hat{p}_t(x)-\nabla p_t(x)\|^2 + \|s_t(x)\|^2 |\hat p_t(x)-p_t(x)|^2}{(\hat p_t(x))^2}
\end{align*}
Hence, the score error in \eqref{main_theorem} can be controlled using the MSE results for $\hat p_t(x)$ and $\nabla \hat p_t(x)$ from Proposition~\ref{MSE_main}. 

In case 2, we consider the lower density areas where $p_t(x) < \rho_n \log^{-c} {n}$. With high probability, $\hat p_t(x) < \rho_n$ and therefore, $\hat s_t(x) = 0$. The score error in \eqref{main_theorem} is now
\begin{equation}
    \int_{p_t(x) < \rho_n \log^{-c} {n}} \| s_t(x)\|^2 p_t(x) dx. \label{eq_case2}
\end{equation}
This represents the expectation of the score function's second moment, restricted to the lower-density area. For any integer $m\ge 1$, H\"older's inequality yields,
\begin{align*}
    \eqref{eq_case2} \le \left(\mathbb{E}\left[\|s_t(X_t)\|^m\right]\right)^{2/m} \left(\mathbb{P}[p(X_t)\le \rho_n\log^{-c}{n}]\right)^{1-\frac{2}{m}},
\end{align*}
where $X_t \sim p_t$. The moment of score can be bounded using the result in \citet{bobkov2019moments} (See Lemma~\ref{s_x_bound}), i.e., $\mathbb{E}\left[\|s_t(X_t)\|^m\right] \lesssim t^{-\frac{m}{2}}$. By Markov's inequality, the second term can be upper bounded in terms of the R\'enyi Entropy \cite{renyi1961measures} of $X_t$. Combining this with the R\'enyi entropy bound for the sub-Gaussian distribution we derive in Appendix~\ref{sec_renyi}, we obtain the desired result for case 2.

Finally, in case 3 where $ \rho_n \log^{-c} {n} < p_t(x) < \rho_n \log^c {n} $, we must consider both $\hat{s}_t(x) \neq 0$ and $\hat{s}_t(x) = 0$, since neither is guaranteed with high probability. However, since $p_t$ has both lower bound and upper bound, we can thereby apply the methods from both Case 1 and Case 2 to bound this term. Combining three cases we can prove the result in Theorem~\ref{theorem2}.

\subsection{Proof sketch of Theorem~\ref{main_theorem2}} \label{subsec_proof_thm2}

Here we sketch the proof of Theorem~\ref{main_theorem2}. Full details of the proof of Theorem~\ref{main_theorem2} can be found in Appendix~\ref{sec_proof_theorem3}. 

We first recall some notations
defined in Section~\ref{sec_background}:
\begin{itemize}
    \item The forward process in \eqref{BM_process} is denoted by$\left(\bX_t\right)_{t\in [0,T]}$ and $\bX_t \sim p_t$;
    \item The backward process \eqref{backward_sde_bm} is denoted by $\left(\bY_t\right)_{t\in [0,T]}$ and by definition, $\bY_t \sim p_{T-t}$;
    \item The process of Algorithm~\ref{algorithm1} is denoted by $(\widehat \bY_t)_{t\in [0,T]}$.
\end{itemize}
Let $(\bar \bY_t)_{t\in [0,T]}$ be $(\widehat \bY_t)_{t\in [0,T]}$ replacing $\widehat \bY_0 \sim \mathcal{N}(0,T\boldsymbol{I}_d)$ by $\bar \bY_0 \sim p_T$, i.e., $(\bar \bY_t)_{t\in [0,T]}$ satisfies:
\begin{equation*}
    \mathrm{d} \bar \bY_t= \hat s_{T-t}(\bar \bY_t) \mathrm{d} t+\mathrm{d} \bB_t, \;\;\;\; \bar \bY_0 \sim p_T.
\end{equation*}
Using triangle inequality of TV distance, 
\begin{align*}
    \mathbb{E}&[\mathrm{TV}(\bX_0, \widehat \bY_{T-t_0})] \leq \mathrm{TV}(\bX_0, \bX_{t_0}) 
    \\  
    &+ \mathbb{E}[\mathrm{TV}(\bY_{T-t_0}, \bar \bY_{T-t_0})] 
    + \mathrm{TV}( \bX_{T} , \mathcal{N}(0,T\boldsymbol{I}_d)).
\end{align*}
Therefore, it suffices to control three errors: the error from early stopping: $\mathrm{TV}\left(\bX_0, \bX_{t_0}\right)$, the error from score estimation: $\mathbb{E}\left[\mathrm{TV}\left(\bY_{T-t_0}, \bar \bY_{T-t_0}\right)\right]$, and the error from the initialization of backward process: $\mathrm{TV}\left( \bX_{T} , \mathcal{N}(0,T\boldsymbol{I}_d)\right)$.

\noindent\textbf{Controlling the error from early stopping}. Using standard Fourier analysis and tail bound for Sub-Gaussian distribution, the error $\mathrm{TV}\left(\bX_0, \bX_{t_0}\right)$ can be controlled by the following theorem.
\begin{theorem}[See also Theorem~\ref{theorem3}]
    Under Assumption~\ref{assumption1} and Assumption~\ref{assumption2}, let $t_0 = n^{-\frac{2}{2\beta + d}}$ and $p_{t_0} = p_0 * \Phi_{t_0}$, where $\Phi_t$ is the density of Gaussian distribution $\mathcal{N}\left(0, t \boldsymbol{I}_d\right)$ and $*$ denote the convolution operator, then there exists a constant $C$ that depends on $p_0$, $\beta$, $L$ and dimension $d$ such that
    \begin{equation*}
        \mathrm{TV}\left(p_0, p_{t_0}\right) \le C\, \mathrm{polylog}(n) \,n^{-\frac{\beta}{2 \beta+d}}.
    \end{equation*}
\end{theorem}
The proof can be found in Appendix~\ref{theorem3}. This theorem shows that for the density $p_0$ belonging to the Sobolev class of order $\beta \le 2$ (this is the only part of the argument requiring $\beta \le 2$), by setting the early stopping time $t_0$ exactly the same as that in Corollary~\ref{main_corollary}, the TV distance between the true data distribution $\bX_0 \sim p_0$ and the perturbed distribution $\bX_0 + \bB_{t_0} \sim p_{t_0}$ will not exceed the minimax optimal rate of convergence $n^{-\frac{\beta}{2\beta+d}}$.

\noindent\textbf{Girsanov's theorem for controlling the score error}. The application of Girsanov's Theorem in analyzing the convergence of diffusion models has been explored in several studies, see \citet{song2021maximum,chen2022sampling,oko2023diffusion}. In our analysis, we replaced the unknown drift term of the process $(\bY_{t})_{t \in [0,T-t_0]}$ by our kernel score estimator $\hat s_t(x)$, resulting in the new process $(\bar \bY_{t})_{t \in [0,T-t_0]}$. By Pinsker's inequality and data processing inequality, $\mathrm{TV}\left(\bY_{T-t_0}, \bar \bY_{T-t_0}\right) \lesssim \sqrt{\mathrm{D}_{\mathrm{KL}}\left( \mathbb{P}_{\bY} \,\|\, \mathbb{P}_{\bar \bY}\right)}$, where $\mathbb{P}_{\bY}$ and $\mathbb{P}_{\bar \bY}$ are path measure for $(\bY_{t})_{t \in [0,T-t_0]}$ and $(\bar \bY_{t})_{t \in [0,T-t_0]}$, accordingly. Girsanov's theorem (detailed in Appendix~\ref{Girsanov}) shows that the KL divergence between the path measures is bounded by the accumulated score error over time $t$, which can be bounded using Corollary~\ref{main_corollary}, i.e.,
\begin{align*}
    &\mathbb{E}\left[\mathrm{D}_{\mathrm{KL}}\left( \mathbb{P}_{\bY} \,\|\, \mathbb{P}_{\bar \bY}\right)\right] 
    \\
    &\lesssim \int_{t_0}^{T} \mathbb{E} \left[\int_{x\in \mathbb{R}^d} \|\hat s_t(x) - \nabla \log p_t(x)\|^2\,p_t(x) dx\right] \,dt
    \\
    &\lesssim \mathrm{polylog}(n) n^{-\frac{2\beta}{2\beta+d}}.
\end{align*}
\noindent\textbf{Error from Initialization}. Similar to the Exponential convergence of the Ornstein–Ulhenbeck process \cite{bakry2014analysis}, we show that the Brownian diffusion process in \eqref{BM_process} is close to Gaussian distribution in KL divergence for large enough time $T$:
\begin{equation*}
    \mathrm{D}_{\mathrm{KL}}\left( \bX_{T} \,\|\, \mathcal{N}(0,T\boldsymbol{I}_d)\right) \lesssim \frac{1}{T}.
\end{equation*}
Consequently, by applying Pinsker's inequality, the initialization error $\mathrm{TV}\left( \bX_{T} , \mathcal{N}(0,T\boldsymbol{I}_d)\right)$ has polynomial decay rate of $n^{-\frac{\beta}{2\beta+d}}$ for $T \ge n^{\frac{2\beta}{2\beta+d}}$.

\section{Discussions}\label{sec_discussion}

\noindent\textbf{Discretization error}. In the analysis of Section~\ref{sec_TV_results}, we assume it is possible to solve the approximated backward SDE \eqref{output_of_algorithm} in Algorithm~\ref{algorithm1}. However, in practical scenarios, solving the backward SDE \eqref{output_of_algorithm} directly is infeasible. Therefore, we need to apply numerical SDE solvers such as Euler-Maruyama discretization to approximate a solution. 

Specifically, to discretize the SDE in \eqref{output_of_algorithm}, consider the grid points $0 = t_1 < \ldots < t_N = T - t_0$ and define the step size $\gamma_k = t_{k+1} - t_k$. We then consider the following process: Let $\hat \bY_0 \sim \mathcal{N}(0, T\boldsymbol{I}_d)$, and for $t$ in each intervals $[t_k, t_{k+1}]$, define $\widehat \bY_t$ via the following SDE:
\begin{align}\label{SED_discrete}
    \mathrm{d} \widehat \bY_t= \hat s_{T-t_k}(\widehat \bY_{t_k}) \mathrm{d} t+\mathrm{d} \bB_t, \; \text{for } t \in [t_k, t_{k+1}].
\end{align}
This is equivalent to the following iterative process:
\begin{align*}
    \widehat \bY_{t_{k+1}} =  \widehat \bY_{t_k} + \gamma_k  \hat s_{T-t_k}(\widehat \bY_{t_k})  + \sqrt{\gamma_k}  \boldsymbol{z}_k
\end{align*}
for $k = 0, \dots N-1$, where $\boldsymbol{z}_k \sim \mathcal{N}(\boldsymbol{0}, \boldsymbol{I}_d)$.

This approximation introduces an extra discretization error in the convergence analysis of the diffusion model in Theorem~\ref{main_theorem2}. According to Theorem 2.2 of \citet{chen2023improved}, when the data distribution $p_0$ has finite second moment, which is satisfied under our Sub-Gaussian assumption, the discretization error can be bounded by 
\begin{align}
    d^2 \sum_{k=0}^{N-1} \frac{\gamma_{k+1}^2}{t_{k}^2}.\label{discretization_error}
\end{align}
Therefore, by taking a constant step size $t_k = t_0 + kh$, the discretization error in \eqref{discretization_error} scales as $h^2 \int_{t_0}^T \frac{1}{t^2} dt = h^2 \left(\frac{1}{t_0} - \frac{1}{T}\right)$. In our setting, $t_0$ and $T$ are both polynomial functions of $n$, the discretization error becomes negligible if we set  $h \asymp n^{-C}$ for a sufficiently large constant $C>0$. 

Since discretization is not the main focus of our paper, we refer the reader to Theorem 2.2 of \citet{chen2023improved} for a more detailed discussion.

\noindent\textbf{Future works}. In practice, instead of directly applying kernel density estimation, the unknown score function is often learned through empirical risk minimization of a certain class of neural networks using score matching. We conclude this paper with two possible avenues for further
research. First, it is worth investigating whether a neural network-based score estimator can achieve a similar score estimation error upper bound as described in Theorem~\ref{theorem2}. Additionally, the estimation error of the diffusion model, as described in Theorem~\ref{main_theorem2}, has been proven for $\beta \leq 2$. Future research could explore extending this proof to cases where $\beta > 2$. 
 
\section{Conclusion}\label{sec_conclusion}
In this paper, we prove general upper bounds for the score estimation error for Gaussian mixture distributions,
under a mild sub-Gaussian assumption for the unknown density,
using a truncated version of the Kernel Density Estimator.
Our score error bound has optimal dependence on the sample size and the variance component of the Gaussian mixture.
By applying the Girsanov theorem and adopting early stopping, 
this implies sharp bounds on the sampling error in the diffusion model.
As a consequence, for the Sobolev class of density functions, 
we demonstrate that the total variation sampling error of the diffusion model matches the minimax rate in nonparametric statistics \citep{stone1982optimal,stone1980optimal,tsybakov}. 
This removes crucial assumptions in similar recent results on the minimax optimality of diffusion models, 
such as the density lower bound.


\section*{Acknowledgements}

We would like to thank the reviewers for their valuable feedback and insightful suggestions.

\section*{Impact Statement}

This paper presents work whose goal is to advance the field of Machine Learning. There are many potential societal consequences of our work, none of which we feel must be specifically highlighted here.

\newpage

\bibliography{ref}
\bibliographystyle{icml2024}

\newpage
\appendix
\onecolumn

\section{Notation}\label{appendix_notation}
Here is a more detailed summary of notations.
$p_0: \mathbb{R}^d \rightarrow \mathbb{R}$ denotes the density function of true data distribution. For any time $t$, we denote the probability distribution $p_t(x) := p_0 * \phi_t (x)= \int p_0(y) \phi_t (x-y)dy$, where $*$ represents the convolution operator, and $\phi_t$ is the density function of Gaussian distribution $\mathcal{N}(0, t\boldsymbol{I}_d)$. Let $s_t(x):=\nabla \log p_t(x)$ denote the (Stein) score function corresponding to the probability distribution $p_t(x)$ at time $t$.

The operations $\vee$ is defined as $a \vee b := \max \{a, b\}$.
For any functions $f$ and $g$, $f(n) \asymp g(n)$ implies that there exist two positive constants $C_1$ and $C_2$ such that $C_1 g(n) \le f(n) \le C_2 g(n)$ for all $n$, $f(n)=O(g(n))$ indicates that there exists a positive real number $M$ and an integer $n_0$ such that $|f(n)| \le M g(n)$ for all $n \ge n_0$, and $f(n)=\Omega(g(n))$ implies there exists a positive constant $c$ and an integer $n_0$ such that $f(n) \ge c \cdot g(n)$ for all $n \ge n_0$. The symbols $\lesssim$ and $\gtrsim$ denotes the corresponding inequality
up to a constant $C$. In the notations used here and throughout this paper, the implicit constants (such as $C_1$, $C_2$, $M$ and $C$ here) may depend on $p_0$ but not on $t$ and $n$, unless otherwise indicated. Throughout the proofs, the constant $C$, $C_1$, $C_2, \ldots$ are used to denote positive constants and their values may change from one line to another, but are independent of everything else. 
We define $\mathrm{poly}(f(n)):= f(n)^{C}$ for some positive constant $C$. For a set $\mathcal A \subset \mathbb{R}^d$, $\left|\mathcal A \right|$ denotes its Lebesgue measure.

For a square-integrable function $f: \mathbb{R}^d \rightarrow \mathbb{R}$, denote $\|f\|_{L^2}:= \int_x f^2(x) dx$ and $\|f\|_{\infty}:= \sup_{x} |f(x)|$. For a vector $x = (x_1,\ldots,x_d) \in \mathbb{R}^d$, $\|x\|_{2} := \sqrt{\sum_{i=1}^d x_i^2}$ and $\|x\|_1 := \sum_{i=1}^n |x_i|$. The $\ell_{\infty}$ norm of a vector $x \in \mathbb{R}^d$, is defined as $\|x\|_{\infty}=$ $\max \left(\left|x_1\right|,\left|x_2\right|, \ldots,\left|x_d\right|\right)$. 
The total variation (TV) distance between to probability distribution $p$ and $q$ is defined as $\mathrm{TV}(p,q):= \frac{1}{2} \int \left|p(x) - q(x)\right| dx$.

\section{Numerical Results}\label{appendix_numerical}
For numerical experiments using real data, we firstly demonstrate the impracticality of the prior uniform-in-time assumption in a real-world example by using 28$\times$28 MNIST dataset. In this experiment, we use U-Net shaped score-net \cite{ronneberger2015u} as our estimator. Note that the finding in Figure~\ref{figure-right} matches the conclusion in \cite{oko2023diffusion} regarding the convergence rate for neural net estimators. Actually, the decay rate we showed in Figure~\ref{figure-right} follows the results we discussed in Theorem~\ref{theorem2}, i.e., the slope approximately matches the theoretical result where the index of $t$ is $-\frac{3}{2}$.

The baselines include: (1) To represent the training procedure of the reverse process, we used score-based U-Net architecture \cite{song2021scorebased} to learn the score function, and (2) we apply a specific forward process, i.e., the Brownian diffusion process, described in Section~\ref{sec_background}, Algorithm~\ref{algorithm1}. 
\begin{figure}[ht]
\vskip 0.2in
\begin{center}
    \subfigure[1-dimensional-MSE Convergence Rate]{
        \includegraphics[width=0.45\columnwidth]{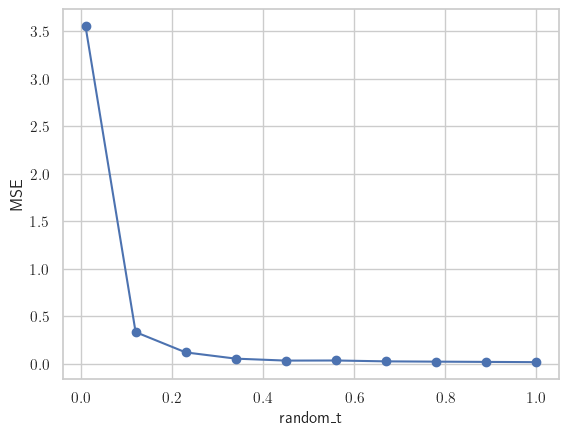}
        \label{figure-left}
    }
    \hfill 
    \subfigure[Log-Log 1-dimensional-MSE Convergence Rate V.S. Time]{
        \includegraphics[width=0.45\columnwidth]{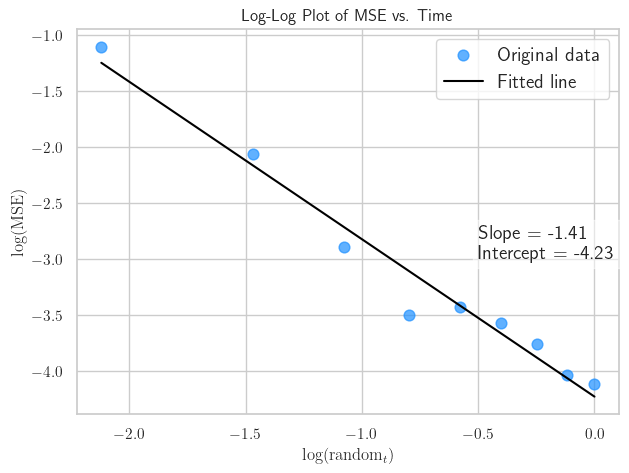}
        \label{figure-right}
    }
    \caption{Convergence Rates}
\end{center}
\vskip -0.2in
\end{figure}
Our training objective $\mathbf{s}_\theta(\mathbf{x}, t)$ is a continuous weighted combination of Fisher divergences, given by
\begin{equation*}
    \mathbb{E}_{t \in \mathcal{U}(0, T)} \mathbb{E}_{p_t(\mathbf{x})}\left[\left\|\nabla_{\mathbf{x}} \log p_t(\mathbf{x})-\mathbf{s}_\theta(\mathbf{x}, t)\right\|_2^2\right]
\end{equation*}
where $\mathcal{U}(0, T)$ denotes a uniform distribution over the time interval $[T_1, 1]$, where $T_1 >0$ is any constant.
After training a convolutional neural network with a U-shaped architecture (known as U-Net), and applying the results to linear regression as illustrated in Figure~\ref{figure-right}, we observe that the slope closely approximates the main result of the mean square error as presented in Theorem~\ref{theorem2}, particularly when substituting $d=1$ for relatively small values of $t$.
Although the U-Net score estimator differs from the kernel estimator in the proof of Theorem~\ref{theorem2}, the $-\frac{3}{2}$ slope is the minimax rate that does not depend on the choice of the score estimator.

\section{Construction and Properties of the Kernel Score Estimator}\label{appendix_score_function}
In this section, we will study the properties of the kernel score estimator we defined in \eqref{score_estimator_eq}, i.e.,
\begin{equation*}
   \hat{s_t}(x):=\frac{\nabla \hat{p_t}(x)}{\hat{p}_t(x)} \mathbbm{1}_{\left\{x\colon \hat{p}_t(x) \ge \rho_n\right\}}(x).
\end{equation*}
\subsection{Construction of Kernel Function}\label{appendix_construct_kernel}
In this section we construct kernel with order $\ell \ge 1$ that are bounded, compactly supported in $[-1,1]^d$.
\begin{definition}\citep{tsybakov}\label{kernel_order_def}
    Let $ \ell \ge 1 $ be an integer. A function $ K_d: \mathbb{R}^d \to \mathbb{R} $ is a \emph{kernel of order $ \ell $} if for any multi-index $ \alpha = (\alpha_1, \alpha_2, \ldots, \alpha_d) $ with $ |\alpha| = \alpha_1 + \alpha_2 + \ldots + \alpha_d \le \ell $, the mappings 
\begin{equation*}
   x \mapsto x^\alpha K(x) 
\end{equation*}
are integrable and satisfy the following conditions:
\begin{align*}\label{kernel_condition}
    \int_{\mathbb{R}^d} K(x) dx &= 1, \\
    \int_{\mathbb{R}^d} x^\alpha K(x) dx &= 0, \quad \text{for all } 1 \le |\alpha| \le \ell-1, \\
    \text{there exists an } \alpha \text{ with } |\alpha| &= \ell \text{ such that } \int_{\mathbb{R}^d}x^\alpha K(x) dx \neq 0.
\end{align*}  
\end{definition}

\begin{lemma}
\label{Bound_on_kernel}
    For any positive integer $\ell$, there exists kernel of order $\ell$ supported on $[-1,1]^d$, denoted as $K_d(x) := \prod_{i=1}^d K(x_i)$ such that for any $1\le i\le d$ and $\alpha \le \ell$,
    \begin{equation*}
        \int_{[-1,1]^d} K_d(x)dx=1, \;\;\;\|K_d\|_{L^2} = O(\ell^{\frac{3d} {2}}),\;\;\; \|K_d\|_{\infty} =O( \ell^{\frac{5d}{2}}),
    \end{equation*}
    \begin{equation*}
        \int_{[-1,1]}\left|K(x_i) x_i^{\alpha}\right| d x_i= O\left(\ell^{\frac{5}{2}}/\alpha\right), \;\;\; 
        \int_{[-1,1]} |K'(x_i)x_i^{\alpha}|dx =O(\ell^{\frac{5}{2}}/\alpha)
    \end{equation*}
    and
    \begin{equation*}
         \left\|\|\nabla K_d (\cdot)\|_2\right\|_{L^2}=O(\sqrt{d} \ell^{\frac{3d}{2}}),\;\;\; \|(\nabla K_d (\cdot))_i\|_{\infty}= O(\ell^\frac{5d}{2}).
    \end{equation*}
\end{lemma}
\begin{proof}
We first consider the $d=1$ case.
    We can construct such kernel using Legendre polynomials: Denote by $\{P_n(x)\}_{n=1}^\infty$ the Legendre polynomials. 
    Define $K$ by setting $K(-\infty)=0$ and
    \begin{equation*}
        K'(x) =  \sum_{i=0}^{\ell+1} a_i P_i(x),
    \end{equation*}
    where $a_i\neq 0$ only if $i$ is odd.
    Since $P_i$ is an odd function for odd $i$, by this construction we automatically have $K(\pm1)=0$.
    Then from the orthogonality properties of Legendre polynomials
    \begin{equation*}
        a_i = \frac{2i+1}{2}\int_{-1}^1 K'(x) P_i(x)dx. 
    \end{equation*}
    To ensure that $K(x)$ is kernel of order $\ell$,
    we need
    \begin{align*}
    \int_{-1}^1K'(x)dx&=0;
    \\
    \int_{-1}^1K'(x)xdx&=-\int_{-1}^1K(x)dx=-1;
    \\
    \int_{-1}^1K'(x)x^kdx&=-k \int_{-1}^1K(x)x^{k-1}dx=0,\quad k=2,\dots,\ell+1.
    \end{align*}
    Note that these are $\ell+2$ linear constraints, which uniquely determine $\{a_0,\dots,a_{\ell+1}\}$.
    (For this we need that the matrix $A$, defined as $A:=\{\int_{-1}^1 x^iP_j(x)dx\}_{i,j=0,\dots,\ell+1}$ which represents the coefficients of the system of linear equations, is invertible. This is true because $BA=\{\int_{-1}^1 P_i(x)P_j(x)dx\}_{i,j=0,\dots,\ell+1}$ for some matrix $B$.)
    Therefore,
    \begin{align*}
    a_i&=-\frac{2i+1}{2}[x^i]P_i(x)
    =-\frac{2i+1}{2}2^ii \binom{\frac{i - 1}{2}}{i},
    \end{align*}
    where $[x^i]P_i(x)$ denotes the coefficients for $x^i$ in the polynomials $P_i(x)$. Therefore,
    \begin{align*}
        |a_i| &= (2i+1)\, \frac{\left((i)!!\right)^2}{i!} \\
        &\lesssim (2i+1) \frac{2 i \left(\frac{i}{e}\right)^{i}}{\sqrt{2\pi i}\left(\frac{i}{e}\right)^i} \quad \left(i!! \sim \sqrt{2n}\left(\frac{i}{e}\right)^{i/2}\; \text{and } i! \sim \sqrt{2\pi n}\left(\frac{i}{e}\right)^{i}\right)\\
        &\lesssim (2i+1)\sqrt{i}.
    \end{align*}
    Therefore,
    \begin{equation*}
        \|K'\|_{L^2}^2 = \int K'^2(x)dx = \sum_{i=0}^{\ell+1} \frac{2}{2i+1} a_i^2 \lesssim \sum_{i=0}^{\ell+1}  \,(2i+1)i = O(\ell^3), 
    \end{equation*}
    \begin{equation*}
        \|K'\|_\infty \leq \frac{\ell +2}{2} \max |a_i| = O(\ell^{\frac{5}{2}}).
    \end{equation*}
    Last we consider $\int |K'(x)x^\alpha|dx$.
    \begin{equation*}
        \int_{[-1,1]} |K'(x)x^\alpha|dx \leq \|K'\|_{\infty}\int_{[-1,1]} |x^{\alpha}|dx \leq O(\ell^{5/2}) \,\frac{2}{\alpha+1} = O(\ell^{5/2}/\alpha).
    \end{equation*}
    Next we consider the bounds for kernel $K(x) =K(-1) + \int_{-1}^x K'(t)dt = \int_{-1}^x K'(t)dt$.
    \begin{align*}
        \|K\|_{L^2}^2 &= \int K^2(x)dx \\
        &= \int \left(\int_{-1}^x K'(t)dt\right)^2dx \\
        &\le \int\int_{-1}^x K'(t)^2dtdx \\
        &\le 2\|K'\|^2\\
        &= O(\ell^3),
    \end{align*}
    \begin{equation*}
        \|K\|_\infty = \sup_{x\in [-1,1]} \left|\int_{-1}^x K'(t)dt\right| \le 2 \|K'\|_{\infty} =O(\ell^{\frac{5}{2}}).
    \end{equation*}
    Finally,
    \begin{align*}
    \int_{[-1,1]}\left|K(x_i) x_i^{\alpha}\right| d x_i 
        &\leq \|K\|_{\infty} \int_{[-1,1]} |x^\alpha|\,dx 
        = O(\ell^{5/2}/\alpha).
    \end{align*}
    In general, for $x = (x_1,\ldots,x_d)\in \mathbb{R}^d$, let
    \begin{equation*}
    K_d(x):=\prod_{j=1}^d K\left(x_j\right), \quad (\nabla K_d(x))_j=K^{\prime}\left(x_j\right) \prod_{i \neq j} K\left(x_i\right).
    \end{equation*}
    Using the above result in 1-dimension, we can directly find that
    \begin{equation*}
        \|K_d\|_{L^2} = O(\ell^{\frac{3d} {2}}),\quad \|K_d\|_{\infty} =O( \ell^{\frac{5d}{2}})
    \end{equation*}
    and
    \begin{equation*}
        \int_{[-1,1]^d}\left|K_d(x) x_j^{\ell}\right| d x\le 
        \|K\|_{\infty}^{d-1}\int_{[-1,1]}\left|K(x_j) x_j^{\ell}\right| d x_j= O\left(\ell^{\frac{5d-2}{2}}\right), \;\;\; \text{for any }  1\le j \le d.
    \end{equation*}
    Besides,
    \begin{align*}
        \left\|\|\nabla K_d (\cdot)\|_2\right\|_{L^2}^2 &= \int \sum_{i=1}^d (\nabla K_d(x))_i^2 dx \\
        &= \sum_{i=1}^d \int  \left(K^{\prime}\left(x_i\right) \prod_{j \neq i} K\left(x_j\right) \right)^2dx \\
        &\le d \|K'\|_{L^2}^2 \|K\|_{L^2}^{2(d-1)} \\
        &=O(d \ell^{3d}).
    \end{align*}
    For any fixed $i$,
    \begin{align*}
        \|(\nabla K_d (\cdot))_i\|_{\infty} &\le  \|K'\|_{\infty} \|K\|_{\infty}^{(d-1)} = O(\ell^\frac{5d}{2}).
    \end{align*}
\end{proof}
\subsection{Mean Squared Error Analysis}\label{appendix_MSE}
Intuitively, if both the gradient estimator $\nabla \hat p_t(x)$ and the density estimator $\nabla \hat p_t(x)$ are close to their respective true values $\nabla p_t(x)$ and $p_t(x)$, the kernel score estimator $\hat s_t(x)$ defined in \eqref{score_estimator_eq} should be also close to the true score function $s_t(x) = \frac{\nabla p_t(x)}{p_t(x)}$. In this section, we study the mean squared errors (MSE) of $\hat p_t(x)$ and $\nabla \hat p_t(x)$, and it will be shown in the proof of Theorem~\ref{theorem2} that these mean squared errors are the dominant part of the score error.

The MSE for $\hat p_t(x)$ and $\nabla \hat p_t(x)$ are defined as follows:
\begin{align*}
    \mathrm{MSE}(\hat p_t(x)) &:= \mathbb{E}\left[| \hat{p}_t(x) -  p_t(x) |^2\right],\\
    \mathrm{MSE}(\nabla\hat{p}_t(x))&:= \mathbb{E} \left[\|\nabla \hat p_t(x) - \nabla p_t(x)\|^2\right].
\end{align*}
We follow the common framework of bias-variance trade-off \citet{tsybakov} with slight modifications, which allows us to obtain a sharper bound that depends on the location $x$.
\begin{proposition}[\textbf{MSE for estimator $\hat{p}_t$}]\label{MSEphat}
    If we take the bandwidth $h = \frac{\sqrt{t}}{D_n}$ for some constant $D_n$ that depends on $n$, then there exists a constant $C_2$ that depend on $p_0$ and $d$, such that the MSE for $\hat p_t(x)$ satisfies
    \begin{equation*}
        \mathrm{MSE}(\hat p_t(x)) \le C_2 \left(\frac{D_n^d\,\ell^{3d}}{nt^{\frac{d}{2}}} \, p_t^*(x) + \left(\frac{D_n}{de}\right)^{-2\ell} \ell^{- \ell + 5d -\frac{1}{2}}\right),
    \end{equation*}
    where $\ell$ is the order of kernel defined in \autoref{Bound_on_kernel}, $p_t^*(x):=\sup_{\|\lambda\|_{\infty} < h}p_t(x+\lambda)$. In particular, if we take $\ell=\Omega(\log n)$ and $D_n = C^* \sqrt{\log n}$, 
    for some constant $C^*> de$, we have
    \begin{equation*}
        \mathrm{MSE}(\hat p_t(x)) \leq C_2 \;\mathrm{polylog}(n)\left(\frac{p_t^*(x)}{n t^{\frac{d}{2}}}  + (\log n)^{-\log n }\right).
    \end{equation*}
\end{proposition}
\begin{proof} 
    By the Bias-Variance trade-off, we have  
    \begin{equation*}
        \text{MSE}(\hat p_t(x)) = \sigma^2(x) + b(x)^2,
    \end{equation*}
where
\begin{equation*}
    \sigma^2(x) = \mathbb{E}\left[\left| \hat{p}_t(x) - \mathbb{E}[\hat{p}_t(x)]  \right|^2 \right],\;\;\;b(x) = \mathbb{E}\left[\hat{p}_t(x) \right] - p_t(x).
\end{equation*}
We first analyze the variance term $\sigma^2(x)$.
Denote
\begin{equation*}
   \eta_i(x) = K_d\left(\frac{x-X_i^t}{h}\right) - \mathbb{E} \left[K_d\left(\frac{x-X_i^t}{h}\right)\right],
\end{equation*}
where $X_i^t \sim p_t$ be the $i$-th observation at specific time point $t$.
Then we have
\begin{align*}
        \mathbb{E}\left[\eta_i^2(x)\right] 
         &\le \mathbb{E}\left[K_d^2\left(\frac{x-X_i^t}{h}\right)\right] 
         \\
         &= \int_{\mathbb{R}^d} K_d^2\left(\frac{x-y}{h}\right) p_t(y)\;dy 
         \\
         &= h^d \int_{[-1,1]^d} K_d^2(z) p_t(x-hz)dz
         \\
         &\le h^d \;p_t^*(x) \int_{[-1,1]^d} K_d^2(z)dz,
\end{align*}
where
    \begin{equation*}
         p_t^*(x):=\sup_{\|\lambda\|_{\infty} \le h}p_t(x+\lambda).
    \end{equation*}
Therefore, since $\eta_i(x)$ are independent and $\mathbb {E} \eta_i(x) = 0$, 
    \begin{align*}
        \sigma^2(x) 
        &= \mathbb{E}\left[\left(\frac{1}{n h^d} \sum_{i=1}^n \eta_i(x)\right)^2\right] = \frac{1}{nh^{2d}} \mathbb{E}\left[\eta_1^{2}(x)\right] \le \frac{p_t^*(x)}{nh^{d}} \int K_d^2(z)dz \lesssim   \frac{\ell^{3d}\,p_t^*(x)}{nh^{d}},
    \end{align*}
where in the last inequality we use the result in Lemma~\ref{Bound_on_kernel}. Taking $h = \frac{\sqrt{t}}{D_n}$, we get
\begin{equation}\label{MSE_pt_var}
    \sigma^2(x) \lesssim \frac{\ell^{3d}\, D_n^d\,p_t^*(x)}{nt^{d/2}}.
\end{equation}
Next we consider the bias $b(x)$:
    \begin{align*}
        b(x) 
        &= \frac{1}{nh^d}\mathbb{E} \left[\sum_{i=1}^n K_d\left(\frac{x-X_i^t}{h}\right)\right] - p_t(x) 
        \\
        & = \frac{1}{h^{d}} \int_{\mathbb{R}^d} K_d\left(\frac{x-z}{h}\right)p_t(z)dz - p_t(x)
        \\
        &= \int_{\mathbb{R}^d} K_d\left(-u \right) [p_t(x+uh) - p_t(x)]du.
    \end{align*}
Using the Multivariate version of Taylor's Theorem \citep{folland2005higher}, we have
\begin{equation*}
p_t(x+u h)=p_t(x)+\sum_{|\alpha|=1} \frac{D^\alpha p_t(x)}{\alpha !}(u h)^\alpha+\cdots+\sum_{|\alpha|=\ell} \frac{D^\alpha p_t(x+\tau u h)}{\alpha !}(u h)^\alpha
\end{equation*}
for some $\tau\in [0,1]$, where we use the multi-index notation,
\begin{equation*}
  |\alpha|=\alpha_1+\cdots+\alpha_d = \ell, \quad u^\alpha =\prod_{i=1}^d u_i^{\alpha_i}, \quad \alpha ! =\prod_{i=1}^d \alpha_{i}!.  
\end{equation*}
Using the definition of kernel of order $\ell$, $\int K_d(-u)du=1$ and $\int u^{\alpha} K_d(u) du = 0$ for $|\alpha| \le \beta -1$, we have
\begin{align*}
       \left|b(x)\right|
       &=\left|\sum_{|\alpha| = \ell} \int_{\mathbb{R}^d} K_d(-u) \left(D^\alpha p_t(x+\tau uh) \right)\frac{(uh)^\alpha}{\alpha !} du\right| \\
       &\le h^\ell \sum_{|\alpha| = \ell} \int_{\mathbb{R}^d} \left|K_d(-u) \left(D^\alpha p_t(x+\tau uh) \right)\frac{u^\alpha}{\alpha !}\right| du \\
       &\le  h^\ell \sum_{|\alpha| = \ell} \int_{\mathbb{R}^d} \left|K_d(-u) \frac{u^\alpha}{\alpha !}\right| du \, \left(\sup_{x} |D^\alpha p_t(x)|\right).
\end{align*}
By Proposition~\ref{bound_on_pt}, we have 
\begin{equation*}
    \sup_x |D^{\alpha}p_t(x)| \leq C_1^d \|p_0\|_{\infty} \ell^{\frac{\ell}{2} + \frac{1}{4}} t^{-\ell/2}.
\end{equation*}
Using the property of the kernel $K_d$ in Lemma~\ref{Bound_on_kernel},
\begin{align*}
    \sum_{|\alpha| = \ell} \int_{\mathbb{R}^d} \left|K_d(-u) \frac{u^\alpha}{\alpha !}\right| du &= \sum_{|\alpha| = \ell} \prod_{i=1}^d \frac{1}{\alpha_i!}\int_{\mathbb{R}} \left|K(-u_i) u_i^{\alpha_i}\right| du_i 
    \\
    &\lesssim \sum_{|\alpha| = \ell} \prod_{i=1}^d \frac{1}{\alpha_i!}\frac{\ell^{5/2}}{\alpha_i} 
    \\
    &\le \ell^{5d/2} \sum_{|\alpha| = \ell}  \frac{1}{\alpha!}
    \\
    &= \ell^{5d/2} \frac{d^\ell}{\ell!}.
\end{align*}
Here in the last equality, we use the following multinomial theorem with $x_i = 1$ for all $i=1,\ldots,d$:
\begin{equation*}
    \left(\sum_{i=1}^d x_i\right)^\ell = \sum_{|\alpha| = \ell} \frac{\ell!}{\alpha!} \prod_{i=1}^d x_i^{\alpha_i}.
\end{equation*}
By Stirling's approximation,
\begin{align*}
    \ell! 
    &> \sqrt{2\pi \ell} \,\ell^\ell e^{-\ell} e^{\frac{1}{12\ell+1}} > \ell^{\ell + \frac{1}{2}} e^{-\ell}.
\end{align*}
Thus, 
\begin{align*}
        |b(x)| &\lesssim C_1^d \|p_0\|_{\infty} h^\ell\, (de)^{\ell}\ell^{- \frac{\ell}{2} +\frac{5d}{2} -\frac{1}{4}}\, t^{-\frac{\ell}{2}}.
\end{align*}
Take the bandwidth as $h = \frac{\sqrt{t}}{D_n}$, we have:
\begin{equation}\label{MSE_pt_bias}
        |b(x)| \lesssim C_1^d \|p_0\|_{\infty} \left(\frac{D_n}{de}\right)^{-\ell} \ell^{- \frac{\ell}{2} +\frac{5d}{2} -\frac{1}{4}}.
\end{equation}
Then the results follow by combining \eqref{MSE_pt_var} and \eqref{MSE_pt_bias}.
\end{proof}

\begin{proposition}[\textbf{MSE for estimator $\nabla{\hat{p_t}}$}]\label{MSEphatprime}
    If we take the bandwidth $h = \frac{\sqrt{t}}{D_n}$ for some constant $D_n$ that depends on $n$, then there exists a constant $C_3$, such that the MSE for $\nabla\hat{p}_t(x)$ satisfies
    \begin{equation*}
        \mathrm{MSE}(\nabla\hat{p}_t(x)) \le C_3 \left(\frac{ \ell^{\frac{3d}{2}}\, D_n^{d+2}}{nt^{\frac{d+2}{2}}}  p_t^*(x) + \left(\frac{D_n}{de}\right)^{-2(\ell-1)} \ell^{-\ell + 5d - \frac{1}{2}} t^{-1}\right),
    \end{equation*}
    where $\ell$ is the order of kernel defined in Lemma~\ref{Bound_on_kernel}, $p_t^*(x):=\sup_{\|\lambda\|_{\infty} < h}p_t(x+\lambda)$. In particular, if we take $\ell=\log n$ and $D_n = C \sqrt{\log n}$ for some constant $C> de$, we have
    \begin{equation*}
        \mathrm{MSE}(\nabla\hat{p}_t(x)) \le C_3 \;\mathrm{polylog}(n)\left(\frac{p_t^*(x)}{n t^\frac{d+2}{2}} + (\log n)^{-\log n } t^{-1}\right).
    \end{equation*}
\end{proposition}
\begin{proof}
We follow the same procedure above in Proposition~\ref{MSEphat}. The Bias-Variance trade-off reads 
    \begin{equation*}
        \text{MSE}(\nabla \hat p_t(x)) = \sigma^2(x) + \left\|b(x)\right\|^2,
    \end{equation*}
where
\begin{equation*}
    \sigma^2(x) = \mathbb{E}\left[\left\| \nabla \hat{p}_t(x) - \mathbb{E}\left[\nabla\hat{p}_t(x)\right]  \right\|^2 \right],\;\;\;\; b(x) = \mathbb{E}\left[\nabla \hat{p}_t(x) \right] - \nabla p_t(x).
\end{equation*}
By definition, the kernel density estimation for the gradient $\nabla p_t(x)$ is:
\begin{equation*}
    \nabla \hat{p}_t(x) = \frac{1}{nh^{d+1}}\sum_{i=1}^n \nabla K_d\left(\frac{x-X_i^t}{h}\right),
\end{equation*}
where $X^t_i$ stands for the $i-th$ sample at time $t$.
We first analyze the variance term $\sigma^2(x)$.
Denote
    \begin{equation*}
        \eta_i(x):= \nabla K_d\left(\frac{x-X_i^t}{h}\right)-\mathbb{E}\left[\nabla K_d\left(\frac{x-X_i^t}{h}\right)\right],
    \end{equation*}
    Then, we get
    \begin{align*}
    \mathbb{E}\left[\|\eta_i(x)\|^2\right] &\le \mathbb{E}\left[\left\|\nabla K_d\left(\frac{x-X_i^t}{h}\right)\right\|^2\right] 
    \\
    &= \int_{\mathbb{R}^d} \left\|\nabla K_d\left(\frac{x-y}{h}\right)\right\|^2 p_t(y)\;dy 
    \\
    &= h^d \int_{[-1,1]^d} \left\|\nabla K_d(u)\right\|^2 p_t(x-uh)du 
    \\
    &\le h^d \;\sup_{\|\lambda\|_{\infty} < h}p_t(x+\lambda) \int_{[-1,1]^d} \left\|\nabla K_d(u)\right\|^2du 
    \\
    &\lesssim h^d p_t^*(x)\,d \ell^{\frac{3d}{2}}, 
\end{align*}
where in the last inequality we use the result in Lemma~\ref{Bound_on_kernel}.
    Therefore, since $\eta_i(x)$ are independent and $\mathbb {E} \eta_i(x) = 0$, we have
    \begin{align*}
        \sigma^2(x) &= \mathbb{E}\left[\left\|\frac{1}{n h^{d+1}} \sum_{i=1}^n \eta_i(x)\right\|^2\right] 
        = \frac{1}{nh^{2d+2}} \mathbb{E}\left[\eta_1^2(x)\right] 
        \lesssim \frac{d \ell^{\frac{3d}{2}}}{nh^{d+2}}  p_t^*(x).
    \end{align*}
    Then, if we take $h = \frac{\sqrt{t}}{D_n}$, then
    \begin{equation}\label{MSE_prime_pt_var}
        \sigma^2(x) \lesssim \frac{d \ell^{\frac{3d}{2}}\, D_n^{d+2}}{nt^{\frac{d+2}{2}}}  p_t^*(x).
    \end{equation}
    Next, we consider the bias term $b(x)$, 
    \begin{align*}
        b(x) 
        &= \frac{1}{nh^{d+1}}\mathbb{E} \left[\sum_{i=1}^n \nabla K_d\left(\frac{x-X^i}{h}\right)\right] - \nabla p_t(x) \\
        &= \frac{1}{h^{d+1}} \int_{\mathbb{R}^d} \nabla K_d\left(\frac{x-y}{h}\right)p(y)dy - \nabla p_t(x) \\
        &= \frac{1}{h} \int_{\mathbb{R}^d} \nabla K_d\left(-u \right) p(x+uh)du - \nabla p_t(x).
        \end{align*}
We consider the first coordinate of $b(x)$:
\begin{equation*}
b_1(x) := \frac{1}{h}\int  K'(-u_1)\prod_{j=2}^d K(-u_j)p_t(x_1+u_1h,\ldots,x_d+u_dh)du_1\ldots du_d - \frac{\partial}{\partial x_1} p_t(x)
\end{equation*}
By Multivariate version of Taylor's Theorem \citet{folland2005higher}, we have
\begin{equation*}
    p_t(x+uh) = \sum_{|\alpha|\le \ell-1} \frac{D^\alpha p(x)}{\alpha!} (uh)^\alpha +  \sum_{|\alpha| = \ell} \frac{D^\alpha p(x+\tau uh)}{\alpha!} (uh)^\alpha, \quad \text{for some }\tau\in [0,1].
\end{equation*}
    Using $K(1) = K(-1) = 0$, $\int K_d(-u)du=1$, $\int u^\ell K_d(u) du = 0$ for $\ell \le \beta -1$ and integration by parts, we have the first coordinate of $b(x)$:
    \begin{align*}
        |b_1(x)| &= \frac{1}{h} \left|\int \sum_{|\alpha| = \ell} (\nabla K_d(-u))_1 D^\alpha p(x+\tau uh) \frac{(uh)^\alpha}{\alpha!} du \right| 
        \\
        &\le h^{\ell-1} \sum_{|\alpha| = \ell} \int_{\mathbb{R}^d} \left|(\nabla K_d(-u))_1 \left(D^\alpha p_t(x+\tau uh) \right)\frac{u^\alpha}{\alpha !}\right| du 
        \\
        &\le h^{\ell-1} \sum_{|\alpha| = \ell} \int_{\mathbb{R}^d} \left|(\nabla K_d(-u))_1 \frac{u^\alpha}{\alpha !}\right| du \, \left(\sup_{x} |D^\alpha p_t(x)|\right).
    \end{align*}
By Lemma~\ref{bound_on_pt},
\begin{equation*}
    \sup_x |D^{\alpha}p_t(x)| \leq C_1^d \|p_0\|_{\infty} \ell^{\frac{\ell}{2} + \frac{1}{4}} t^{-\ell/2}.
\end{equation*}
Using the property of the kernel $\nabla K_d$ in Lemma~\ref{Bound_on_kernel} and multinomial theorem, we have
\begin{align*}
    \sum_{|\alpha| = \ell} \int_{\mathbb{R}^d} \left|(\nabla K_d(-u))_1 \frac{u^\alpha}{\alpha !}\right| du &= \sum_{|\alpha| = \ell} \left[ \left(\frac{1}{\alpha_1!} \int_{\mathbb{R}} \left|K'(-u_1) u_1^{\alpha_1}\right| du_1 \right)
    \prod_{i=2}^d \frac{1}{\alpha_i!}\int_{\mathbb{R}} \left|K(-u_i) u_i^{\alpha_i}\right| du_i \right]
    \\
    &\lesssim \sum_{|\alpha| = \ell} \prod_{i=1}^d \frac{1}{\alpha_i!}\frac{\ell^{5/2}}{\alpha_i} 
    \\
    &\le \ell^{5d/2} \sum_{|\alpha| = \ell}  \frac{1}{\alpha!}
    \\
    &= \ell^{5d/2} \frac{d^\ell}{\ell!}.
\end{align*}
By Stirling's approximation,
\begin{align*}
    \ell! 
    &> \sqrt{2\pi \ell} \,\ell^\ell e^{-\ell} e^{\frac{1}{12\ell+1}} > \ell^{\ell + \frac{1}{2}} e^{-\ell}.
\end{align*}
Thus, 
\begin{align*}
       |b_1(x)| &\lesssim C_1^d \|p_0\|_{\infty} h^{\ell-1}\, (de)^{\ell}\ell^{- \frac{\ell}{2} +\frac{5d}{2} -\frac{1}{4}}\, t^{-\frac{\ell}{2}}.
\end{align*}
Taking the bandwidth as $h = \frac{\sqrt{t}}{D_n}$, we have:
\begin{equation*}
        |b_1(x)| \lesssim C_1^d \|p_0\|_{\infty} de\,\left(\frac{D_n}{de}\right)^{-(\ell-1)} \ell^{- \frac{\ell}{2} +\frac{5d}{2} -\frac{1}{4}} t^{-\frac{1}{2}},
\end{equation*}
and then
\begin{equation}\label{MSE_prime_pt_bias}
    \|b(x)\|^2 \le  |b_1(x)|^2 \lesssim d^2eC_1^{2d} \|p_0\|_{\infty}^2 \left(\frac{D_n}{de}\right)^{-2(\ell-1)} \ell^{-\ell + 5d - \frac{1}{2}} t^{-1}.
\end{equation}
Finally, The results follow by combining \eqref{MSE_prime_pt_var} and \eqref{MSE_prime_pt_bias}.
\end{proof}

\section{Proof of Estimation Error for Score Function}
\subsection{Proof of Theorem~\ref{theorem2}}\label{appendix_theorem2}
The proof of our main result begins with a discussion of the estimation of the target total error term. Using the kernel score estimator we constructed in Appendix~\ref{appendix_score_function}, i.e.,
\begin{equation*}
    \hat s_t(x):= \frac{\nabla \hat p_t(x)}{\hat p_t(x)} \mathbbm1_{\hat p_t(x) > \rho_n},
\end{equation*}
the following proposition shows that the score error $\|\hat s_t(x) - s_t(x)\|^2$ can be decomposed into two components: the error of the gradient estimator $\nabla\hat{p}_t(x)$ and the error of the density estimator $\hat{p}_t(x)$. Therefore, it suffices to control the error of $\nabla\hat{p}_t(x)$ and $\hat{p}_t(x)$.
\begin{lemma}\label{score error upper bound}
    For any $t>0$, the score estimator defined in \eqref{score_estimator_eq} satisfies
    \begin{align*}
        &\int_x \mathbb{E} \left[\|\hat s_t(x) - s_t(x)\|^2\right]p_t(x)dx \\
        &\leq 2\int_x \mathbb{E}\left[ \frac{\|\nabla\hat{p}_t(x)-\nabla p_t(x)\|^2 + \|s_t(x)\|^2 |\hat p_t(x)-p_t(x)|^2}{\hat p_t(x)^2}\mathbbm1_{\{\hat p_t(x) > \rho_n\}} \right] p_t(x)dx \\
        &+ \int_x \mathbb{P}\left( \hat p_t(x) \leq \rho_n  \right) \|s_t(x)\|^2p_t(x)dx.
    \end{align*}
\end{lemma}
\begin{proof}
If $x$ is such that $\hat p_t(x) > \rho_n$,
we have the decomposition
\begin{align*}
    \hat{s}_t(x) - s_t(x) &= \nabla \log \hat{p}_t(x) - \nabla \log p_t(x) \\
    &= \frac{\nabla \hat{p}_t(x)}{\hat{p}_t(x)} - \frac{\nabla p_t(x)}{p_t(x)} \\
    &= \frac{\nabla \hat{p}_t(x) p_t(x) - \nabla p_t(x) p_t(x) + \nabla p_t(x) p_t(x) - \nabla p_t(x) \hat{p}_t(x)}{\hat{p}_t(x) p_t(x)} \\
    &= \frac{\nabla \hat{p}_t(x) - \nabla p_t(x) + s_t(x)(p_t(x) - \hat{p}_t(x))}{\hat{p}_t(x)}.
\end{align*}
Therefore,
\begin{align*}
    &\quad \int_x \mathbb{E} \left[\|\hat s_t(x) - s_t(x)\|^2\mathbbm1_{\hat{p}_t(x)> \rho_n}\right]p_t(x)dx
    \nonumber\\
    &\leq 2\int_x \mathbb{E}\left[ \frac{\|\nabla\hat{p}_t(x)-\nabla p_t(x)\|^2 + \|s_t(x)\|^2 |\hat p_t(x)-p_t(x)|^2}{\hat p_t(x)^2} \mathbbm1_{\hat{p}_t(x)> \rho_n}\right] p_t(x)dx.
\end{align*}
On the other hand, if $x$ is such that $\hat p_t(x) \le \rho_n$,
we have 
$\hat s_t(x) = 0$.
Therefore,
\begin{equation*}
    \int_{x} \mathbb{E}\left[\|\hat s_t(x) - s_t(x)\|^2\mathbbm1_{\{\hat p_t(x) \leq \rho_n\}}\right]p_t(x)dx = \int_x \mathbb{P}\left( \hat p_t(x) \leq \rho_n  \right) \|s_t(x)\|^2p_t(x)dx.
\end{equation*}
\end{proof}
We further consider the score error in three cases, respectively. Recall that our score estimator is defined as 
\begin{equation*}
    \hat s_t(x):= \frac{\nabla \hat p_t(x)}{\hat p_t(x)} \mathbbm1_{\hat p_t(x) > \rho_n}.
\end{equation*}
The key challenge is to control the error in low-density regions. We therefore consider the following three cases. Let $c> 4C_d + dT_2 + 2$ be a sufficiently large constant, where is a constant that only depends on the dimension $d$.
In the first case $p_t(x) > \rho_n \log^c {n}$, the density is relatively high, so as we will see in the following proof, with high probability, $p_t(x)$ and $\hat p_t(x)$ will have the same order and as a result, $\hat p_t(x) > \rho_n$, which implies that the first term in Lemma~\ref{score error upper bound} will dominate. Whereas in case 2 we consider the lower density area $p_t(x) < \rho_n \log^{-c} {n}$, with high probability, $\hat p_t(x) < \rho_n$ and therefore, $\hat s_t(x) = 0$. Finally, in case 3 where $ \rho_n \log^{-c} {n} < p_t(x) < \rho_n \log^c {n} $, we must consider both $\hat{s}_t(x) \neq 0$ and $\hat{s}_t(x) = 0$, since neither of them is guaranteed with high probability. In this case, we will apply the methods from both Case 1 and Case 2 to bound this term.
\subsubsection{CASE 1: \texorpdfstring{$p_t(x) > \rho_n \log^c {n}$}{case1}}
\begin{lemma}\label{main_lemma1}
    For any time $n^{-T_1}<t<n^{T_2}$, and for any  $0 <\varepsilon <1 $, taking $\rho_n = \frac{1}{n\,t^{d/2}}$, the score error on the set
    \begin{equation*}
        G_1 := \{x \colon p_t(x) > \rho_n \log^c {n}\}
    \end{equation*}
    can be upper bounded by
    \begin{equation*}
        \mathbb{E}\left[\int_{G_1} \|\hat s_t(x) - s_t(x)\|^2 p_t(x) dx \right]\lesssim \mathrm{polylog}(n) \,n^{-1}t^{-\frac{d+2}{2}}(t^{\frac{d}{2}} + \sigma_0^d)\,g_1(n, \varepsilon),
    \end{equation*}
    where 
    \begin{equation*}
        g_1(n, \varepsilon):= \exp\left(\frac{1}{\varepsilon \log n}\right) (\varepsilon)^{-1}\,n^{C_1 \varepsilon},
    \end{equation*}
    where $C_1 = \frac{2+dT_2}{2}$.
    \begin{remark}
We will take $\varepsilon = \frac{\log \log n}{\log n}$ later in  Section~\ref{sec_combine} when Lemma~\ref{main_lemma1} is applied.
    \end{remark}
\end{lemma}
\begin{proof}[Proof of Lemma~\ref{main_lemma1}]
    From Lemma~\ref{score error upper bound},
    \begin{align}
        &\int_{G_1} \mathbb{E} \left[\|\hat s_t(x) - s_t(x)\|^2\right]p_t(x)dx \notag\\
        &\leq 2\int_{G_1} \mathbb{E}\left[ \frac{\|\nabla\hat{p}_t(x)-\nabla p_t(x)\|^2 + \|s_t(x)\|^2 |\hat p_t(x)-p_t(x)|^2}{\hat p_t(x)^2}\mathbbm1_{\{\hat p_t(x) >  \rho_n \}} \right] p_t(x)dx \label{e365}\\
        &+ \int_{G_1} \mathbb{P}\left( \hat p_t(x) \leq \rho_n  \right) \|s_t(x)\|^2p_t(x)dx \label{e366}
    \end{align}
    Denote
    \begin{equation*}
        I_t(x) := \frac{\|\nabla \hat {p}_t(x)-\nabla p_t(x)\|^2 + \|s_t(x)\|^2 |\hat p_t(x)-p_t(x)|^2}{\hat p_t(x)^2}.
    \end{equation*}
    By \autoref{high_probability_bound}, define the random set
    \begin{equation*}
        A := \left\{x \colon \left|\hat p_t(x) - p_t(x)\right|< C_4(\alpha)\; \mathrm{polylog}(n)\left(\sqrt{\frac{p_t^*(x)}{n t^{\frac{d}{2}}}} + \frac{1}{n t^{\frac{d}{2}}} + (\log n)^{-\frac{(\log n)}{2}}\right)\right\}, 
    \end{equation*}
    then $\{x\in A\}$ has probability greater than $1-n^{-\alpha}$ for any $\alpha > 0$, where $p_t^*(x):=\sup_{\|\lambda\|_{\infty} < h}p_t(x+\lambda)$ and $C_4(\alpha)$ is some constant that only depends on $p_0$, $d$ and $\alpha$.
    Then we consider
    \begin{align}
     I_t(x)\mathbbm1_{\{\hat p(x) > \rho_n\}} &= I_t(x)\mathbbm1_{\{\hat p(x) > \rho_n\}\cap A} + I_t(x)\mathbbm1_{\{\hat p(x) > \rho_n\}\cap A^c}.\label{eq295}
    \end{align}
    Now we split the proof into several parts:
    
    \noindent \textbf{Part 1: The first term $I_t(x)\mathbbm1_{\{\hat p(x) > \rho_n\}\cap A}$ in \eqref{eq295}}.
    Under the random set $\{x:\hat p(x) > \rho_n\}\cap A\cap G_1$, by Lemma~\ref{phat_ge_p}: 
    \begin{equation*}
    \hat p_t^2(x) \geq \left(p_t(x) -  C_4(\alpha)\; \mathrm{polylog}(n)\left(\sqrt{\frac{p_t^*(x)}{nt^{\frac{d}{2}}}} + \frac{1}{n t^{\frac{d}{2}}} + (\log n)^{-\log n}\right)\right)^2\ge \frac{1}{4} p_t^2(x).
    \end{equation*} 
    Therefore, for any $x \in G_1$,
    \begin{align*}
        I_t(x)\mathbbm1_{\{\hat p(x) > \rho_n\}\cap A} &\leq \frac{\|\nabla\hat{p}_t(x)-\nabla p_t(x)\|^2 + \|s_t(x)\|^2 |\hat p_t(x)-p_t(x)|^2}{ p_t(x)^2/4}.
    \end{align*}
    Using Proposition~\ref{MSEphat} and Proposition~\ref{MSEphatprime}:
    \begin{equation}\label{main_pt_equation}
        \mathbb{E}\left[|\hat p_t(x)-p_t(x)|^2\right] \leq C_2 \,\mathrm{polylog}(n) \left(\frac{p_t^*(x)}{n t^{d/2}} + \left(\log n\right)^{-\log n}\right),
    \end{equation}
    \begin{equation}\label{main_ptprime_equation}
        \mathbb{E}\left[\|\nabla\hat{p}_t(x)-\nabla p_t(x)\|^2\right] \le C_3 \,\mathrm{polylog}(n) \left(\frac{p_t^*(x)}{n\,t^{(d+2)/2}} + \left(\log n\right)^{-\log n} t^{-1}\right),
    \end{equation}
    Therefore,
    \begin{align}
        \int_{G_1} \mathbb{E} \left[I_t(x)\mathbbm1_{\{\hat p(x) > \rho_n\}\cap A}\right] p_t(x)dx &\lesssim \mathrm{polylog}(n) \Bigg[\frac{1}{n t^{(d+2)/2}} \int_{G_1} \frac{p_t^*(x)}{p_t(x)}dx \nonumber \\
        &\quad + \frac{1}{n t^{d/2}} \int_{G_1} \frac{p_t^*(x)}{p_t(x)}\|s_t(x)\|^2dx \nonumber \\
        &\quad + (\log n)^{- \log n} \int_{G_1} \left(\frac{t^{-1}}{p_t(x)} + \frac{\|s_t(x)\|^2}{p_t(x)}\right)dx \Bigg].\label{e375}
    \end{align}
    For the last term, we can use Lemma~\ref{bound_on_pt}, Lemma~\ref{Bound_for_G} and the fact that $t > n^{-T_1}$, to obtain:
    \begin{align*}
       \int_{G_1} \left(\frac{t^{-1}}{p_t(x)} + \frac{\|s_t(x)\|^2}{p_t(x)}\right)dx \le |G_1| \left(t^{-1}\rho_n^{-1} \log^{-c} {n}  + t^{-\frac{d}{2}}\,\rho_n^{-3} \log^{-3c} {n} \right)\lesssim \mathrm{poly}(n\log n).
    \end{align*}
    Since $(\log n)^{- \log n} 	\ll \mathrm{poly}(n\log n)$ for sufficiently large $n$, the error for this term can be ignored.
    Using Lemma~\ref{pstar_property}, for any $0<\varepsilon <1$,
    \begin{equation*}
        \frac{p^*_t(x)}{p_t^{1-\varepsilon}(x)}\leq \frac{1}{\sqrt{1-\varepsilon}} \exp\left\{\frac{1-\varepsilon}{D_n^2\varepsilon}\right\}, 
    \end{equation*}
    we have
    \begin{equation*}
        \int_{G_1} \frac{p_t^*(x)}{p_t(x)}dx = \int_{G_1} \frac{p^*_t(x)}{p_t^{1-\varepsilon}(x)} p^{-\varepsilon}_t(x)dx \leq \frac{1}{\sqrt{1-\varepsilon}} \exp\left\{\frac{1-\varepsilon}{2D_n^2\varepsilon}\right\} (\log {n}) ^{-c\varepsilon}\rho_n^{-\varepsilon} |G_1|,
    \end{equation*}
    and 
    \begin{equation*}
        \int_{G_1} \frac{p_t^*(x)}{p_t(x)}\,\|s_t(x)\|^2dx \leq \frac{1}{\sqrt{1-\varepsilon}} \exp\left\{\frac{1-\varepsilon}{D_n^2\varepsilon}\right\} (\log {n}) ^{-c\varepsilon}\rho_n^{-\varepsilon} \int_{G_1} \|s_t(x)\|^2dx.
    \end{equation*}
    Here, $|G_1|$ and $\int_{G_1} \|s_t(x)\|^2dx$ can be upper bounded by Lemma~\ref{Bound_for_G} and Lemma~\ref{bound_on_GS}, accordingly. Therefore, we have
    \begin{align*}
        \int_{G_1} \mathbb{E} &\left[I_t(x)\mathbbm1_{\{\hat p(x) > \rho_n\}\cap A}\right] p_t(x)dx \\
        &\lesssim \mathrm{polylog}(n) \frac{(t^{d/2} + \sigma_0^d)}{n\, t^{(d+2)/2}}\frac{1}{\sqrt{1-\varepsilon}} \exp\left\{\frac{1-\varepsilon}{D_n^2\varepsilon}\right\} (\log {n}) ^{-c\varepsilon}\rho_n^{-\varepsilon} \left(1+m\, (\log n)^{-\frac{c}{m} }\rho_n^{-\frac{1}{m}}\right) \\
        &\le \mathrm{polylog}(n) \frac{(t^{d/2} + \sigma_0^d)}{n\, t^{(d+2)/2}}\frac{1}{\sqrt{1-\varepsilon}} \exp\left\{\frac{1-\varepsilon}{D_n^2\varepsilon}\right\}\rho_n^{-\varepsilon} \left(1+m\,\rho_n^{-\frac{1}{m}}\right) 
    \end{align*}
    Take $m \asymp \frac{1}{\varepsilon}$, and $D_n = C\sqrt{\log n}$, then $\frac{1}{\sqrt{1-\varepsilon}} \exp\left\{\frac{1-\varepsilon}{D_n^2\varepsilon}\right\} \lesssim \exp\left(\frac{1}{\varepsilon \log n}\right)$. Using $\rho_n = n^{-1}t^{-d/2} > n^{-1}n^{-dT_2/2}$, we then have 
    \begin{align*}
        \rho_n^{-\varepsilon} \left(1+m\,\rho_n^{-\frac{1}{m}}\right) &= \rho_n^{-\varepsilon} \left(1+\varepsilon^{-1}\,\rho_n^{-\varepsilon}\right) \\
        &\le \varepsilon^{-1} \rho_n^{-2\varepsilon} \\
        &\le \varepsilon^{-1} n^{ C_1 \varepsilon}, 
    \end{align*}
    where $C_1 = \frac{2+dT_2}{2}$.
    Thus we can conclude that:
    \begin{equation*}
        \int_{G_1} \mathbb{E} \left[I_t(x)\mathbbm1_{\{\hat p(x) > \rho_n\}\cap A}\right] p_t(x)dx 
        \lesssim \mathrm{polylog}(n) \,n^{-1}t^{-\frac{d+2}{2}}(t^{\frac{d}{2}} + \sigma_0^d) \exp\left(\frac{1}{\varepsilon \log n}\right) (\varepsilon)^{-1} n^{C_1 \varepsilon}.
    \end{equation*}
    
    \noindent \textbf{Part 2: The second term $I_t(x)\mathbbm1_{\{\hat p(x) > \rho_n\}\cap A^c}$ in \eqref{eq295}}.
    The error bound in this part essentially uses the fact that $A^c$ is small by concentration inequalities. Under the random set $\{x:\hat p(x) > \rho_n\}\cap A^c\cap G_1$, using $\hat p_t(x) > \rho_n$, $p_t(x) > \rho_n \log^c {n} $ and Lemma~\ref{upper_bound_for_p_minus_phat} of uniform upper bound for $\|\nabla\hat{p}_t(x)-\nabla p_t(x)\|^2$ and $\|\nabla p_t(x)\|^2 |\hat p_t(x)-p_t(x)|^2$,
    \begin{align*}
        I_t(x)\mathbbm1_{\{\hat p(x) > \rho_n\}\cap A^c} &\le \left(\frac{\|\nabla \hat{p}_t(x)-\nabla p_t(x)\|^2}{\rho_n^2} + \frac{\|\nabla p_t(x)\|^2 |\hat p_t(x)-p_t(x)|^2}{\rho_n^2 \, p_t^2(x)}\right)\mathbbm1_{\{\hat p(x) > \rho_n\}\cap A^c}\\
        &\lesssim \frac{t^{-(d+1)}\text{poly}\log (n)}{\rho_n^2}\left(1+(\log{n})^{-2c} \rho_n^{-2}\right)\mathbbm1_{A^c}.
    \end{align*}
    Therefore,
    \begin{align}
        \int_{G_1} \mathbb{E} \left[I_t(x)\mathbbm1_{\{\hat p(x) > \rho_n\}\cap A^c}\right] p_t(x)dx &\lesssim \frac{t^{-(d+1)}\mathrm{poly}(\log n )}{\rho_n^2}\left(1+(\log{n})^{-2c}\rho_n^{-2}\right)
        \int_{G_1} \mathbb{E} \left[\mathbbm1_{A^c}\right] p_t(x)dx \notag\\
        &= \frac{t^{-(d+1)}\mathrm{poly}(\log n )}{\rho_n^2}\left(1+(\log{n})^{-2c}\rho_n^{-2}\right) \int_{G_1} \mathbb{P} \left[x \in A^c\right] p_t(x)dx \notag\\
        &\le \frac{t^{-(d+1)}\mathrm{poly}(\log n )}{\rho_n^2}\left(1+(\log{n})^{-2c}\rho_n^{-2}\right) n^{-\alpha}, \label{eq301}
    \end{align}
    where in the last inequality we use the fact that $\mathbb{P} \left[x\in A^c\right] \leq n^{-\alpha}$.
    Since $\alpha$ can be arbitrarily large and $\frac{t^{-(d+1)}\mathrm{poly}\log n }{\rho_n^2}\left(1+(\log{n})^{-2c}\rho_n^{-2}\right)= O(\mathrm{poly}(n \log n))$, we can ignore the error of this term.
    
    \noindent \textbf{Part 3: The second term $\int_{G_1} \mathbb{P}\left( \hat p_t(x) \leq \rho_n  \right) \|s_t(x)\|^2p_t(x)dx$ in \eqref{e366}}.
    Using Lemma~\ref{set_inclusion1},
    \begin{align*}
        \int_{G_1} \mathbb{P}\left( \hat p_t(x) \leq \rho_n  \right) \|s_t(x)\|^2p_t(x)dx 
        &\leq \int_{G_1} \mathbb{P}\left(x\in A^c  \right) \|s_t(x)\|^2p_t(x)dx \\
        &\leq n^{-\alpha} \,\mathbb{E}_{X\sim p_t}\left[\|s_t(X)\|^2\right] \\
        &\lesssim n^{-\alpha} t^{-1}.
    \end{align*}
    Here in the last inequality, we use Lemma~\ref{s_x_bound} to bound the expectation of moment of the score function. Since $\alpha$ can be sufficiently large and $t^{-1} = O(\mathrm{poly}(n))$, the error of this term can be negligible.
\end{proof}
\begin{lemma}\label{phat_ge_p}
    Let $A$ be the high probability set in Lemma~\ref{high_probability_bound}, i.e.,
    \begin{equation*}
        A := \left\{x\colon\left|\hat p_t(x) - p_t(x)\right|< C_4(\alpha)\; \mathrm{polylog}(n)\left(\sqrt{\frac{p_t^*(x)}{n t^{\frac{d}{2}}}} + \frac{1}{n t^{\frac{d}{2}}} + (\log n)^{-\log n}\right)\right\}.
    \end{equation*}
    For any $x \in A$, if $x$ also satisfies: $\hat p_t(x) > \rho_n$ and $p_t(x) > \log^c n \;\rho_n$, then
    \begin{equation*}
        \hat p_t(x) \ge \frac{1}{2} p_t(x).
    \end{equation*}
    sufficiently large $n$.
\end{lemma}
\begin{proof}
For any $x \in A$, since $p^*_t (x) \leq \|p\|_{\infty}$, we have
\begin{align*}
    \left|\hat p_t(x) - p_t(x)\right| &\lesssim \mathrm{polylog}(n)\left(\sqrt{\frac{1}{nt^{\frac{d}{2}}}} + \frac{1}{n t^{\frac{d}{2}}} + (\log n)^{-\log n}\right)
    \\
    &= \mathrm{polylog}(n)\left(\sqrt{\rho_n} + \rho_n + (\log n)^{-\log n}\right)
    \\
    &\lesssim \mathrm{polylog}(n)\left(\sqrt{\rho_n} + \rho_n \right)
\end{align*}
Therefore, under random set A and $G := \{x \colon p_t(x) > \log^c n \rho_n\}$,
\begin{align*}
    \frac{\left|\hat p_t(x) - p_t(x)\right|}{p_t(x)} &\lesssim \mathrm{polylog}(n)\, \frac{\sqrt{\rho_n} + \rho_n}{p_t(x)}  
    \\
    &\leq (\log^c n)^{-1} \mathrm{polylog}(n)\,\left(\rho_n^{-1/2} + 1\right) 
    \\
    &\lesssim \mathrm{polylog}(n)\, n^{O(1)},
\end{align*}
where in the last inequality we use $\rho_n^{-1/2} = n^{1/2} t^{d/4} \le n^{O(1)}$.
Therefore
\begin{align*}
    \frac{\hat p_t(x)}{p_t(x)} \geq 1 - \mathrm{polylog}(n)\, n^{O(1)}. 
\end{align*}
When $n$ is large, we can ensure that 
$\hat p_t(x) \ge \frac{1}{2}\, p_t(x)$.
\end{proof}

\begin{lemma}\label{set_inclusion1}
    Let $A$ be the high-probability random set defined in Lemma~\ref{phat_ge_p}. Define
    \begin{equation*}
        B := \{x \colon p_t(x) > \rho_n \log^c{n} \},
    \end{equation*}
    \begin{equation*}
        C := \{x \colon \hat p_t(x) < \rho_n\}.
    \end{equation*}
    Then $B \cap C \subseteq B \cap A^c$.
\end{lemma}
\begin{proof}
    For $x \in B\cap C$,
      \begin{align*}
        |p_t(x) - \hat p_t(x)| > p_t(x) - \rho_n > \frac{1}{2}\,p_t(x),
    \end{align*}  
    when $n$ is large. But by Lemma~\ref{high_probability_bound}, $A$ implies 
    $|p_t(x) - \hat p_t(x)|\lesssim (\log n)^{C_d} \sqrt{p_t^*(x)\rho_n}$ for some constant $C_d$ that only depends on $d$.
    Using Lemma~\ref{pstar_property}, for any $\varepsilon >0 $, $A\cap B$ implies 
    \begin{align*}
    |p_t(x) - \hat p_t(x)|
    &\lesssim (\log n)^{C_d} (\rho_n \log^c{n})^{-\frac{\varepsilon}{2}}\sqrt{p_t(x)\rho_n}
    \\
    &\le (\log n)^{C_d} (\rho_n \log^c{n})^{-\frac{\varepsilon}{2}}p_t(x)\log^{-\frac{c}{2}}{n}
    \\
    &= (\log n)^{C_d} (n^{-1}t^{-d/2})^{-\frac{\varepsilon}{2}}(\log {n})^{-\frac{c}{2}\left(\varepsilon + 1\right)} p_t(x)
    \\
    &\le (\log n)^{C_d} n^{\frac{\varepsilon}{2}\left(1+\frac{dT_2}{2}\right)} (\log {n})^{-\frac{c}{2}\left(\varepsilon + 1\right)} p_t(x),
    \end{align*}
    where in the last inequality we use the fact that $t \le n^{T_2}$. Therefore, by taking $\varepsilon = \frac{\log \log n}{\log n}$, we have $n^{\frac{\varepsilon}{2}\left(1+\frac{dT_2}{2}\right)} = (\log n)^{\frac{2+dT_2}{4}}$ and $(\log {n})^{-\frac{c}{2}\left(\varepsilon + 1\right)} \lesssim (\log {n})^{-\frac{c}{2}}$, we have 
    \begin{equation*}
        (\log n)^{C_d} n^{\frac{\varepsilon}{2}\left(1+\frac{dT_2}{2}\right)} (\log {n})^{-\frac{c}{2}\left(\varepsilon + 1\right)} \le (\log n)^{C_d + \frac{2+dT_2}{4} - \frac{c}{2}}.
    \end{equation*}
    Then since the constant $c$ satisfies $\frac{c}{2} > 2 C_d + \frac{2 + dT_2}{2} > C_d + \frac{2+dT_2}{4}$, for large enough $n$ we have $|p_t(x) - \hat p_t(x)| < \frac{1}{2}\, p_t(x)$
    and $B\cap C \subseteq A^c$, as desired.
\end{proof}

\subsubsection{CASE 2: \texorpdfstring{$p_t(x) < \rho_n \log^{-c} {n} $}{case2}}\label{appendix_case2}
\begin{lemma}\label{case2}
    For any time $n^{-T_1}<t<n^{T_2}$, taking $\rho_n = \frac{1}{n\,t^{d/2}}$ and for any $0<\varepsilon < 1$, the score error on the set
    \begin{equation*}
        G_2 := \{x \colon p_t(x) < \rho_n \log^{-c} {n}\}
    \end{equation*}
    can be upper bounded by
    \begin{equation*}
        \mathbb{E}\left[\int_{G_2} \|\hat s_t(x) - s_t(x)\|^2 p_t(x) dx \right]\lesssim  n^{-1}t^{-\frac{d+2}{2}}  (t^{\frac{d}{2}} + \sigma_0^d) \,g_2(n,\varepsilon),
    \end{equation*}
    where
    \begin{equation*}
        g_2(n,\varepsilon):= (\varepsilon)^{-(1+d/2)}\,n^{C_2\varepsilon}
    \end{equation*}
    for some constant $C_2$ defined in the proof that only depends on $T_1$, $T_2$ and $d$ (in turn, only depends on $p_0$).
\end{lemma}
\begin{proof}

Let $A$ denote the high probability random set defined in Lemma~\ref{high_probability_bound}, then using Lemma~\ref{set_inclusion2},
\begin{equation*}
    G_2 \cap A \subset \{x \colon \hat p_t(x) < \rho_n\}.
\end{equation*}
Then under $\{x \colon \hat p_t(x) < \rho_n\}$, by definition, $\hat s_t(x) = 0$. Therefore,
\begin{align*}
    &\quad \mathbb{E}\left[\int_{G_2} \|\hat s_t(x) - s_t(x)\|^2 p_t(x) dx \right] \\
    &= \mathbb{E}\left[\int_{G_2} \|\hat s_t(x) - s_t(x)\|^2 \mathbbm{1}_{A}(x)\, p_t(x) dx \right] + \mathbb{E}\left[\int_{G_2} \|\hat s_t(x) - s_t(x)\|^2 \mathbbm{1}_{A^c}(x)\,p_t(x) dx \right] \\
    &\leq \mathbb{E}\left[\int_{G_2} \|\hat s_t(x) - s_t(x)\|^2 \mathbbm{1}_{\{\hat p_t(x) < \rho_n\}}(x)\, p_t(x) dx \right] + \mathbb{E}\left[\int_{G_2} \|\hat s_t(x) - s_t(x)\|^2 \mathbbm{1}_{A^c}(x)\,p_t(x) dx \right] \\
    &\lesssim \int_{G_2} \| s_t(x)\|^2 p_t(x) dx  + \left(\mathbb{E}_{X \sim p_t, \{x_i\}_{i=1}^n} \left[\|\hat s_t(X) - s_t(X)\|^4\right]\right)^{1/2} \left(\int  \mathbb{P}\left(x \in A^c\right)\,p_t(x) dx\right)^{1/2}
    \\
    &\le \int_{G_2} \| s_t(x)\|^2 p_t(x) dx  + \left(\mathbb{E}_{X \sim p_t, \{x_i\}_{i=1}^n} \left[\|\hat s_t(X) - s_t(X)\|^4\right]\right)^{1/2} n^{-\alpha/2},
\end{align*}
where the last inequality follows by the Cauchy–Schwarz inequality. Therefore the error in the second term can be ignored since by Lemma~\ref{bound_on_4_score_error} the fourth moment of the score error is bounded by some $\mathrm{poly}(n^{-1}\log n)$ but $n^{-\alpha}$ can be arbitrarily small. For the first term, using Lemma~\ref{renyi_bound},
\begin{equation*}
    \int_{G_2} \| s_t(x)\|^2 p_t(x) dx \lesssim (\varepsilon)^{-(1+d/2)}\, t^{-1} \rho_n^{1-\varepsilon} \sigma_t^{d(1-\varepsilon)}.
\end{equation*}
Then the result follows by $\rho_n = \frac{1}{nt^{d/2}} = \mathrm{poly}(n)$ and $\sigma_t^d \lesssim t^{\frac{d}{2}} + \sigma_0^d \lesssim \mathrm{poly}(n)$, and the order of the polynomials only depends on $T_1$ and $T_2$.

\begin{lemma}\label{set_inclusion2}
    For any $x$ satisfying $x \in A$ and $p_t(x) < \rho_n \log^{-c} {n}$, we have $\hat p_t(x) < \rho_n$.
\end{lemma}
\begin{proof}
    By definition of the random set $A$,
    \begin{align*}
        \hat p_t(x) &\le p_t(x) +  (\log n)^{C_d} \left(\sqrt{p_t^*(x)\rho_n} + \rho_n + (\log n)^{-\frac{(\log n)}{2}}\right)
        \\
        &\lesssim \rho_n \log^{-c} {n} + (\log n)^{C_d} \sqrt{p_t^*(x)\rho_n}.
    \end{align*}
    Using Lemma~\ref{pstar_property}, for any $\varepsilon > 0$,
    \begin{align*}
        \sqrt{p_t^*(x)\rho_n} &\lesssim \sqrt{p_t^{1-\varepsilon}
        (x)\rho_n} 
        \lesssim (\log n)^{-\frac{c(1-\varepsilon)}{2}} \rho_n^{1-\frac{\varepsilon}{2}}.
    \end{align*}
    Using the fact that $t \le n^{T_2}$,
    \begin{align*}
        \rho_n^{-\frac{\varepsilon}{2}} &= (n^{-1}t^{-d/2})^{-\frac{\varepsilon}{2}} \le n^{\frac{\varepsilon}{2}\left(1+\frac{dT_2}{2}\right)},
    \end{align*}
    we have
    \begin{align*}
        \hat p_t(x) \lesssim \left(1 + (\log n)^{C_d}\,n^{\frac{\varepsilon}{2}\left(1+\frac{dT_2}{2}\right)}\,\log^{\frac{c(\varepsilon+1)}{2}}{n} \right) \rho_n \,\log^{-c}{n}.
    \end{align*}
    Then we can take $\varepsilon = \frac{\log \log n}{\log n}$ such that $n^{\frac{\varepsilon}{2}\left(1+\frac{dT_2}{2}\right)} = (\log n)^{\frac{2+dT_2}{4}}$ and $\log^{\frac{c(\varepsilon+1)}{2}}{n} \le \log^{\frac{c(1/2+1)}{2}}{n} = \log^{\frac{3c}{4}}{n}$. Therefore,
    \begin{align*}
        \left(1 + (\log n)^{C_d}\,n^{\frac{\varepsilon}{2}\left(1+\frac{dT_2}{2}\right)}\,\log^{\frac{c(\varepsilon+1)}{2}}{n} \right)\log^{-c}n &< \log^{-c}n + (\log n)^{C_d + \frac{2+dT_2}{4} - \frac{c}{4}}
        \\
        &<1,
    \end{align*}
    for sufficiently large $n$, since $c > 4 C_d + dT_2 + 2$.
    Therefore, this implies $\hat p_t(x) < \rho_n$.
\end{proof}
\end{proof}

\subsubsection{CASE 3: \texorpdfstring{$ \rho_n \log^{-c} {n} < p_t(x) < \rho_n \log^c {n} $}{case3}}
\begin{lemma}\label{case3}
    For any time $n^{-T_1}<t<n^{T_2}$ and any $0<\varepsilon<1$, the score error on the set
    \begin{equation*}
        G_3 := \{x \colon \rho_n \log^{-c} {n} < p_t(x) < \rho_n \log^c {n}\}
    \end{equation*}
    can be upper bounded by
    \begin{equation*}
        \mathbb{E}\left[\int_{G_3} \|\hat s_t(x) - s_t(x)\|^2 p_t(x) dx \right]\lesssim \mathrm{polylog}(n)\, n^{-1}\,t^{-\frac{d+2}{2}}(t^{\frac{d}{2}} + \sigma_0^d)\, g_3(n,\varepsilon),
    \end{equation*}
    where
    \begin{equation*}
        g_3(n,\varepsilon'):= (\varepsilon)^{-(1+d/2)} \exp\left(\frac{1}{ \varepsilon \log n}\right)n^{C_3\varepsilon}
    \end{equation*}
    for some constant $C_3$ defined in the proof that only depends on $T_1$, $T_2$ and $d$.
\end{lemma}

\begin{proof}
    Again, apply Lemma~\ref{score error upper bound},
    \begin{align}
        &\quad \int_{G_3} \mathbb{E} \left[\|\hat s_t(x) - s_t(x)\|^2\right]p_t(x)dx \notag
        \\
        &\leq 2\int_{G_3} \mathbb{E}\left[ I_t(x)\mathbbm1_{\{\hat p_t(x) > \rho_n\}} \right] p_t(x)dx + \int_{G_3} \mathbb{P}\left( \hat p_t(x) \leq \rho_n  \right) \|s_t(x)\|^2p_t(x)dx \label{e444}
    \end{align}
    where
    \begin{equation*}
        I_t(x) := \frac{\|\nabla\hat{p}_t(x)-\nabla p_t(x)\|^2 + \|s_t(x)\|^2 |\hat p_t(x)-p_t(x)|^2}{\hat p_t(x)^2}.
    \end{equation*}
    Now we split the proof into two parts:
    
    \noindent \textbf{Part 1: The first term $\int_{G_3} \mathbb{E}\left[ I_t(x)\mathbbm1_{\{\hat p_t(x) > \rho_n\}} \right] p_t(x)dx $ in \eqref{e444}}.
    
    First note that for any $x \in G_3$, we have $p_t(x) \log^{-c} {n}< \rho_n < p_t(x) \log^c {n}$. Therefore, 
    \begin{align*}
        \hat p_t(x) \geq \rho_n \geq p_t(x) \log^{-c} {n}.
    \end{align*}
    Then we have
    \begin{align*}
        I_t(x) \leq (\log^{2c} {n})(p_t(x))^{-2}\left(\|\nabla\hat{p}_t(x)-\nabla p_t(x)\|^2 + \|s_t(x)\|^2 |\hat p_t(x)-p_t(x)|^2\right)
    \end{align*}
    Using Proposition~\ref{MSEphat} and Proposition~\ref{MSEphatprime}:
    \begin{equation*}
        \mathbb{E}\left[|\hat p_t(x)-p_t(x)|^2\right] \le C_2 \,\mathrm{polylog}(n) \left(\frac{p_t^*(x)}{n t^{\frac{d}{2}}} + \left(\log n\right)^{-\log n}\right),
    \end{equation*}
    \begin{equation*}
        \mathbb{E}\left[\|\nabla\hat{p}_t(x)-\nabla p_t(x)\|^2\right] \le C_3 \,\mathrm{polylog}(n) \left(\frac{p_t^*(x)}{n\,t^{\frac{d+2}{2}}} + \left(\log n\right)^{-\log n} t^{-1}\right),
    \end{equation*}
    Then
    \begin{align*}
        \mathbb{E}\left[I_t(x)\right] \lesssim \mathrm{polylog}(n)\, (p_t(x))^{-2} \left[\frac{p_t^*(x)}{n\,t^{\frac{d+2}{2}}} + \frac{p_t^*(x)\|s_t(x)\|^2}{n\,t^{1/2}} + \left(\log n\right)^{-\log n}(t^{-1} + \|s_t(x)\|^2)\right]
    \end{align*}
    Integrate over $x$:
    \begin{align*}
        \int_{G_3} \mathbb{E} \left[I_t(x)\mathbbm1_{\{\hat p(x) > \rho_n\}}\right] p_t(x)dx &\lesssim \mathrm{polylog}(n)\,\Bigg[\frac{1}{n t^{\frac{d+2}{2}}} \int_{G_3} \frac{p_t^*(x)}{p_t(x)}dx \nonumber \\
        &\quad + \frac{1}{n t^{\frac{d}{2}}} \int_{G_3} \frac{p_t^*(x)}{p_t(x)}\|s_t(x)\|^2dx \nonumber \\
        &\quad + (\log n)^{- \log n} \int_{G_3} \left(\frac{t^{-1}}{p_t(x)} + \frac{\|s_t(x)\|^2}{p_t(x)} \right)dx \Bigg].
    \end{align*}
    Notice that $G_3 \subset \{x \colon p_t(x) > \rho_n \log^{-c}{n}\}:=G_3'$, which is similar to Case 1 (by replacing $\log^{c}{n}$ with $\log^{-c}{n}$). Therefore, using a similar approach in Lemma~\ref{main_lemma1} of Case 1, we know the term $(\log n)^{- \log n}$ is small than $\mathrm{poly}(n\log n)$ for sufficiently large $n$, then the error for the last term can be ignored:
    \begin{align*}
       (\log n)^{- \log n}\int_{G_3} \left(\frac{t^{-1}}{p_t(x)} + \frac{\|s_t(x)\|^2}{p_t(x)} \right)dx &\leq (\log n)^{- \log n} |G_3'| \left(t^{-1}\rho_n^{-1} \log^{c} {n}  + t^{-\frac{d}{2}}\,\rho_n^{3} \log^{3c} {n} \right) 
       \\
       &\lesssim (\log n)^{- \log n} \mathrm{poly}(n\log n) 
       \\
       &\ll \mathrm{poly}(n^{-1}).
    \end{align*}
    Again, using the similar approach in Lemma~\ref{main_lemma1} in Case 1, 
    \begin{align*}
        &\quad \frac{1}{n t^{(d+2)/2}} \int_{G_3} \frac{p_t^*(x)}{p_t(x)}dx \nonumber + \frac{1}{n\,t^{d/2}} \int_{G_3} \frac{p_t^*(x)}{p_t(x)}\|s_t(x)\|^2dx 
        \\
        &\lesssim  \frac{(t^{d/2} + \sigma_0^d)}{n \,t^{(d+2)/2}}\frac{1}{\sqrt{1-\varepsilon}} \exp\left\{\frac{1-\varepsilon}{D_n^2\varepsilon}\right\} (\log^c n)^{\varepsilon}\rho_n^{-\varepsilon} \left(1+m\, (\log^c n)^{\frac{1}{m} }\rho_n^{-\frac{1}{m}}\right).
    \end{align*}
    As before, by taking $m \asymp \frac{1}{\varepsilon}$, and $D_n = C\sqrt{\log n}$, then $\frac{1}{\sqrt{1-\varepsilon}} \exp\left\{\frac{1-\varepsilon}{D_n^2\varepsilon}\right\} \lesssim \exp\left(\frac{1}{\varepsilon \log n}\right)$. Since $\varepsilon < 1$, we know $(\log^c n)^{\varepsilon} < \log^c n$ and $(\log^c n)^{1/m} < \log^c n$. Using $\rho_n = n^{-1}t^{-d/2} > n^{-1}n^{-dT_2/2}$ we can conclude that 
    \begin{align*}
        \int_{G_3} \mathbb{E} \left[I_t(x)\mathbbm1_{\{\hat p(x) > \rho_n\}}\right] p_t(x)dx &\lesssim \mathrm{polylog}(n)\, n^{-1}\,t^{-\frac{d+2}{2}}(t^{\frac{d}{2}} + \sigma_0^d) \left(\frac{1}{\varepsilon}n^{(2+dT_2)\varepsilon} \exp\left(\frac{1}{ \varepsilon\log n}\right)\right).
    \end{align*}
    \noindent \textbf{Part 2: The second term $\int_{G_3} \mathbb{P}\left( \hat p_t(x) \leq \rho_n  \right) \|s_t(x)\|^2p_t(x)dx$ in \eqref{e444}}.

    Notice that
    \begin{align*}
        \int_{G_3} \mathbb{P}\left( \hat p_t(x) \leq \rho_n  \right) \|s_t(x)\|^2p_t(x)dx 
        &\leq \int_{\{x \;: \;p_t(x) \leq \rho_n \log^c {n}\}} \|s_t(x)\|^2p_t(x)dx.
    \end{align*}
    Therefore, we can use a similar method of case 2 as in Lemma~\ref{renyi_bound}, by replacing $\log^{-c} {n}$ with $\log^{c} {n}$:
    \begin{align*}
        \int_{\{x \;: \;p_t(x) \leq \rho_n \log^c {n}\}} \|s_t(x)\|^2p_t(x)dx 
        &\leq \mathrm{polylog}(n)\,(\varepsilon)^{-(1+d/2)}\, t^{-1} \rho_n^{1-\varepsilon} \sigma_t^{d(1-\varepsilon)}.
    \end{align*}
    Since $\rho_n = \frac{1}{nt^{d/2}} = \mathrm{poly}(n)$ and $\sigma_t^d \lesssim t^{\frac{d}{2}} + \sigma_0^d \lesssim \mathrm{poly}(n)$, and the order of the polynomials $C_2$ (the same as that in case 2) only depends on $T_1$, $T_2$ and $d$. Then can conclude that
    \begin{align*}
        \int_{G_3} \mathbb{P}\left( \hat p_t(x) \leq \rho_n  \right) \|s_t(x)\|^2p_t(x)dx &\leq \mathrm{polylog}(n)(\varepsilon)^{-(1+d/2)}\, n^{-1} t^{-\frac{d+2}{2}} \left(\sigma_0^{\frac{d}{2}}+ t^{\frac{d}{2}}\right)n^{C_2\varepsilon}.
    \end{align*}
\end{proof}

\subsubsection{Combining the Three Cases}
\label{sec_combine}
Next, combining the three cases in Lemma~\ref{main_lemma1}, Lemma~\ref{case2} and Lemma~\ref{case3} as discussed above,
    \begin{align*}
        \mathbb{E} \left[\int_x  \left\| \hat{s_t}(x) - s_t(x) \right\|^2 p_t(x) \, dx \,\right] 
        &= \mathbb{E} \left[\int_{G_1 \cup G_2 \cup G_3}  \left\| \hat{s_t}(x) - s_t(x) \right\|^2 p_t(x) \, dx \,\right] 
        \\
        &\lesssim \mathrm{polylog}(n) \,n^{-1}t^{-\frac{d+2}{2}}(t^{\frac{d}{2}} + \sigma_0^d)\,g(n,\varepsilon),
    \end{align*}
    where
    \begin{align*}
        g(n,\varepsilon): &= g_1(n,\varepsilon) + g_2(n,\varepsilon) + g_3(n,\varepsilon)
        \\
        &\le \left(\exp\left(\frac{1}{\varepsilon \log n}\right) + (\varepsilon)^{-d/2} +  (\varepsilon)^{-d/2}\exp\left(\frac{1}{ \varepsilon\log n}\right)\right) (\varepsilon)^{-1}\,n^{(C_1 \vee C_2 \vee C_3)\varepsilon}.
    \end{align*}
    Let $\varepsilon = \frac{\log \log n}{\log n}$, since $n^{\varepsilon} = n^{\frac{\log \log n}{\log n}} = \log n$, we have $g(n,\varepsilon) \le C\,\mathrm{polylog}(n)$ for some universal constant $C$. Therefore,
    \begin{align*}
        \mathbb{E} \left[\int_x  \left\| \hat{s_t}(x) - s_t(x) \right\|^2 p_t(x) \, dx \,\right] \lesssim \mathrm{polylog}(n) \,n^{-1}t^{-\frac{d+2}{2}}(t^{\frac{d}{2}} + \sigma_0^d).
    \end{align*}

\subsection{Proof of Corollary \ref{main_corollary}}\label{appendix_main_corollary}
In Theorem~\ref{theorem2}, taking the integral with respect to time $t$ from $t_0 = n^{-T_1}$ to $T = n^{T_2}$, we have:
\begin{align*}
     \int_{t\in[t_0,T]} \int_x \mathbb{E} \|\hat s_t(x) - s_t(x)\|^2p_t(x)dxdt
     & \lesssim \mathrm{polylog}(n) \,n^{-1} \int_{t_0}^T t^{-\frac{d+2}{2}} \, (t^{\frac{d}{2}} + \sigma_0^d)dt \\
     &= \mathrm{polylog}(n) \,n^{-1} \left(\log(T) - \log(t_0) - \frac{2\sigma_0^d}{d} T^{-\frac{d}{2}} + \frac{2\sigma_0^d}{d}t_0^{-\frac{d}{2}}\right) \\
     &\le \mathrm{polylog}(n) \,n^{-1} \left(\log(T) - \log(t_0) + \frac{2\sigma_0^d}{d}t_0^{-\frac{d}{2}}\right) \\
     &\lesssim \mathrm{polylog}(n) \,n^{-1}\, n^{\frac{d\,T_1}{2}}.
\end{align*}
    The second part of Corollary~\ref{main_corollary} simply follows by letting $T_1 = \frac{2}{2\beta+d}$ and noticing that $n^{-1}n^{\frac{dT_1}{2}} = n^{-\frac{2\beta}{2\beta+d}}$.

\section{Proof of Theorem \ref{main_theorem2}}\label{sec_proof_theorem3}
\subsection{Control of the error from early stopping}

\begin{theorem}\label{theorem3}
    Under Assumption~\ref{assumption1} and Assumption~\ref{assumption2}, if $\beta \in [0,2]$, $t_0 = n^{-\frac{2}{2\beta + d}}$ and $p_{t_0} = p_0 * \Phi_t$, where $\Phi_t$ is the density of Gaussian distribution in $d$-dimension, $\mathcal{N}\left(0, t \boldsymbol{I}_d\right)$ and $*$ denote the convolution operator, then there exists a constant $C$ that depends on $p_0$, $\beta$, $L$ and dimension $d$ such that
    \begin{equation*}
        \mathrm{TV}\left(p_0, p_{t_0}\right) \le C\, \mathrm{polylog}(n) \,n^{-\frac{\beta}{2 \beta+d}}.
    \end{equation*}
\end{theorem}
\begin{proof}
Firstly,  the total variance (TV) distance between $p_0$ and $p_t$ can be decomposed into two terms and by Jensen's inequality,
\begin{align}
      \mathrm{TV}\left(p_0, p_t\right)
      &=\int_{\mathbb{R}^d}\left|p_0(x)-p_t(x)\right| d x \notag
      \\
      & = \int_{\|x\|_{\infty}<\log n}\left|p_0(x)-p_t(x)\right| d x+\int_{\|x\|_{\infty}>\log n}\left|p_0(x)-p_t(x)\right| d x \notag
      \\
      & \le (2 \log n)^{d / 2} \sqrt{\int_{\mathbb{R}^d}\left|p_0(x)-p_t(x)\right|^2 d x}+\int_{\|x\|_{\infty}>\log n}\left|p_0(x)-p_t(x)\right| d x. \label{e419}
\end{align}
By the Sub-Gaussian tail bound in Lemma~\ref{Sub-Gaussian_tail}, the error of the second term in \eqref{e419} is negligible:
\begin{align}
    \int_{\|x\|_{\infty}>\log n} |p_0(x) - p_t(x)|dx &\leq \int_{\|x\|_{\infty}>\log n} p_0(x) dx+ \int_{\|x\|_{\infty}>\log n} p_t(x)dx \notag
    \\
    &\leq 2d\exp \left(-\frac{(\log n)^2}{2\sigma_0^2}\right) + 2d\exp\left(-\frac{(\log n)^2}{2(\sigma_0+\sqrt t)^2}\right) \notag
    \\
    &\leq 4d n^{-\frac{\log n}{\sigma_0^2 +1}} \notag
    \\
    &\ll n^{-\frac{\beta}{2\beta+d}}.\label{tv_term_3}
\end{align}
Next, we derive an upper bound for the first term $\int_{\mathbb{R}^d}\left|p_0(x)-p_t(x)\right|^2 d x$ in \eqref{e419}. By Plancherel's theorem,
\begin{align} 
\int \left|p_0(x)-p_t(x)\right|^2 d x 
& =\frac{1}{(2 \pi)^d} \int \left|\mathcal{F}\left[p_0\right](\omega)-\mathcal{F}\left[p_0\right](\omega)\right|^2 d \omega \notag
\\ 
& =\frac{1}{(2 \pi)^d} \int_{\|\omega\|_{\infty} \ge t^{-1 / 2}}\left|\mathcal{F}\left[p_0\right](\omega)\right|^2\left|1-\phi_t(\omega)\right|^2 d \omega \label{TV_term1}
\\ 
& +\frac{1}{(2 \pi)^d} \int_{\|\omega\|_{\infty}<t^{-1 / 2}}\left|\mathcal{F}\left[p_0\right](\omega)\right|^2\left|1-\phi_t(\omega)\right|^2 d \omega\label{TV_term2}.
\end{align}

Here $\phi_t (\omega):= \exp\left(-\frac{t^2\|\omega\|^2}{2}\right)$ is the Fourier transformation (or Characteristic function) of Gaussian density.
By the definition of the Sobolev class of density in Assumption~\ref{assumption2}, 
we have
\begin{align}
    d{(2 \pi)^d L^2} 
    & \ge \int_{\mathbb{R}^d} \sum_{i=1}^d\left|\omega_i\right|^{2 \beta}\left|\mathcal{F}\left[p_0\right](\omega)\right|^2 d \omega \notag
    \\
    & \ge \int_{\|\omega\|_{\infty} \ge t^{-1 / 2}}\sum_{i=1}^d\left|\omega_i\right|^{2 \beta}\left|\mathcal{F}\left[p_0\right](\omega)\right|^2 d \omega \notag
    \\
    & \ge t^{-\beta} \int_{\|\omega\|_{\infty} \ge t^{-1 / 2}}\left|\mathcal{F}\left[p_0\right](\omega)\right|^2 d \omega \label{sobolev_time_condition}
\end{align}
For the first term \eqref{TV_term1}, use the fact that $\left|1-\phi_t(\omega)\right| \le 2$ and \eqref{sobolev_time_condition}, substituting $t = n^{-\frac{2}{2 \beta+d}}$, we have
\begin{align}
  \frac{1}{(2 \pi)^d} \int_{\|\omega\|_{\infty} 
   t^{-1/2}}\left|\mathcal{F}\left[p_0\right](\omega)\right|^2\left|1-\phi_t(\omega)\right|^2 \, d\omega 
  & \le d{L^2} t^{\beta}
   = d{L^2} \; n^{-\frac{2\beta}{2\beta + d}}.  \label{tv_term_1}
\end{align}
For the second term in \eqref{TV_term2}, use the fact that $|1-\phi_t(\omega)|^2 = \left|1-\exp\left(-\frac{\|\omega\|^2 t}{2}\right)\right|^2 \le \frac{t^2 \|\omega\|^4}{4}$, we have
\begin{align} 
& \quad \frac{1}{(2 \pi)^d} \int_{\|\omega\|_{\infty}<t^{-1/2}}\left|\mathcal{F}\left[p_0\right](\omega)\right|^2\left|1-\phi_t(\omega)\right|^2 d \omega \notag
\\ 
& \le \frac{1}{(2 \pi)^d} \cdot \frac{t^2}{4} \int_{\|\omega\|_{\infty}<t^{-1/2}}\left|\mathcal{F}\left[p_0\right](\omega)\right|^2\|\omega\|^4 d \omega \notag 
\\ 
& =\frac{1}{(2 \pi)^d} \cdot \frac{t^2}{4} \int_{\|\omega\|_{\infty}<t^{-1 / 2}}\left|\mathcal{F}\left[p_0\right](\omega)\right|^2\|\omega\|^{2 \beta} \cdot \|\omega\|^{4-2 \beta} d \omega \notag
\\
& \le \frac{1}{(2 \pi)^d} \cdot \frac{t^2}{4} \sup _{\|\omega\|_{\infty}<t^{-1/2}}\|\omega\|^{4-2 \beta} \int \left|\mathcal{F}\left[p_0\right](\omega)\right|^2 \|\omega\|^{2\beta} d \omega. \label{TV_term2_mid}
\end{align}
Since $\beta \le 2$, we have 
\begin{align}
\sup _{\|\omega\|_{\infty}<t^{-1/2}} \|\omega\|^{4-2 \beta}
& =\sup _{\|\omega\|_{\infty}<t^{-1/2}} \left(\sum_{i=1}^d\left|\omega_i\right|^{2}\right)^{2 -\beta} 
\le \left(d\left\|\omega\right\|^2_{\infty}\right)^{2-\beta} 
\le d^{2-\beta} t^{\beta-2}. \label{TV_term2_mid_1}
\end{align}
Besides, since $\|\omega\|^{2\beta} = \left(\sum_{i=1}^d |\omega_i|^2\right)^\beta \le \left(d^{\beta-1}\vee 1\right)\sum_{i=1}^d |\omega_i|^{2\beta}$, and by the definition of Sobolev class of densities, we have
\begin{align}
    \int \left|\mathcal{F}\left[p_0\right](\omega)\right|^2 \|\omega\|^{2\beta} d \omega 
    &\le 
    \left(d^{\beta-1}\vee 1\right) \sum_{i=1}^d \int  \left|\mathcal{F}\left[p_0\right](\omega)\right|^2 |\omega_i|^{2\beta} d \omega \notag 
    \\
    &\le d \left(d^{\beta-1}\vee 1\right) (2\pi)^d L^2. \label{TV_term2_mid_2}
\end{align}
Therefore, combining \eqref{TV_term2_mid_1} and \eqref{TV_term2_mid_2},
\begin{align}
    \eqref{TV_term2_mid} \le \frac{L^2 d^{3-\beta}\left(d^{\beta-1}\vee 1\right)}{4} \,t^\beta
    = \frac{L^2 d^{3-\beta}\left(d^{\beta-1}\vee 1\right)}{4} \, n^{-\frac{2 \beta}{2\beta+d}}. \label{tv_term_2}
\end{align}
As a result of combining the above terms \eqref{tv_term_1}, \eqref{tv_term_2} and \eqref{tv_term_3}, we are able to deduce that the total error between $p_0$ and $p_t$ is $\mathrm{TV}\left(p_0, p_t\right) \lesssim \text{polylog}(n) n^{-\frac{\beta}{2 \beta+d}}$.
\end{proof}

\subsection{Girsanov's Theorem}\label{Girsanov}
We can translate the cumulative score error to the KL divergence error by using the following Girsanov Theorem.

\begin{theorem}[Girsanov Theorem, \citet{le2016brownian}]
     Define the stochastic process $(\mathcal {L}_t)_{t\in[0,T]}$ by $\mathcal{L}_t := \int_0^t \,\langle b_s , dB_s\rangle$, where $B$ is a $\mathbb{Q}$-Brownian motion. Assume that $E_{\mathbb{Q}} \int_0^T \|b_s\|^2 \, ds < \infty$. Then, $L$ is a $Q$-martingale in $L^2(\mathbb{Q})$. Moreover, if
\begin{equation*}
E_\mathbb{Q}\, \mathcal{E}(\mathcal{L})_T = 1,
\end{equation*}
where
\begin{align*}
    \mathcal{E}(\mathcal{L})_t &:= \exp\left(\mathcal{L}_t - \frac{1}{2}[\,\mathcal{L}\,]_t\right)
    \\
    &=\exp\left(\int_0^t \,\langle b_s , dB_s\rangle - \frac{1}{2} \int_0^t \|b_s\|^2 \, ds\right).
\end{align*}
Then $\mathcal{E}(\mathcal{L})_T$ is also a $\mathbb Q$-martingale and the process
\begin{equation*}
\Tilde{B}_t := B_t - \int_0^t b_s \, ds
\end{equation*}
is a Brownian motion under the measure $\mathbb P$ defined by the Radon-Nikodym derivative: 
\begin{equation*}
    d\,\mathbb P := \mathcal{E}(\mathcal{L})_T \,d\, \mathbb Q.
\end{equation*}
\end{theorem}
The Girsanov theorem can be used to convert a Brownian motion with drift into a standard Brownian motion by changing the measure. In this paper, we aim to derive the KL divergence between two stochastic processes, $(\bX_t)_{t\in[0,T]}$ and $(\bY_t)_{t\in[0,T]}$. These processes have different drift terms. More specifically, $(\bX_t)_{t\in[0,T]}$ and  $( \bY_t)_{t\in[0,T]}$ are solutions to the following two SDEs:
\begin{equation*}
    d\,\bX_t = b^{\bX}(\bX_t,t) \,dt + d \bB_t, \;\;\; \bX_0 \sim p,
\end{equation*}
\begin{equation*}
    d\,\bY_t = b^{\bY}(\bY_t,t) \,dt + d \bB_t; \;\;\; \bY_0 \sim p.
\end{equation*}
Denote by $\mathbb{P}_{\bX}$ and $\mathbb{P}_{\bY}$ the path measures of $(\bX_t)_{t\in[0,T]}$ and  $( \bY_t)_{t\in[0,T]}$, respectively. Since the initial distributions of these two processes are identical, the only difference for these two processes is their drift terms $b^{\bX}$ and $b^{\bY}$. By Girsanov theorem, we can derive the Radon-Nikodym derivative for their path measure $\frac{d\,\mathbb{P}_{\bX}}{d\,\mathbb{P}_{\bY}}$, which allows us to compute the KL divergence for these two processes:
\begin{equation*}
\mathrm{D}_{\mathrm{KL}}\left(\mathbb{P}_{\bX}\,\|\,\mathbb{P}_{\bY}\right) = \mathbb{E}_{\mathbb{P}_{\bX}}\left[\log \frac{d\,\mathbb{P}_{\bX}}{d\,\mathbb{P}_{\bY}}\right].
\end{equation*}

The above discussion is summarized in the following lemma, as described in \citet{chen2022sampling} and \citet{oko2023diffusion}.
\begin{lemma}[\citet{chen2022sampling},\citet{oko2023diffusion}]\label{Girsanov_lemma}
    Let $p$ be any probability distribution and let $(\bX_t)_{t\in[0,T]}$,  $( \bY_t)_{t\in[0,T]}$ be solutions to the following two SDEs:
\begin{equation*}
    d\,\bX_t = b^{\bX}(\bX_t,t) \,dt + d \bB_t, \;\;\; \bX_0 \sim p,
\end{equation*}
\begin{equation*}
    d\,\bY_t = b^{\bY}(\bY_t,t) \,dt + d \bB_t; \;\;\; \bY_0 \sim p.
\end{equation*}
We denote the distribution of $\bX_t$ and $\bY_t$ as $p_t^{\bX}$ and $p_t^{\bY}$ and the path measure of $(\bX_t)_{t\in[0,T]}$ and $(\bY_t)_{t\in[0,T]}$ as $\mathbb{P}_{\bX}$ and $\mathbb{P}_{\bY}$, respectively.
\begin{enumerate}
    \item Suppose the following Novikov’s condition holds:
    \begin{equation}\label{Novikov}
        \mathbb{E}_{\mathbb{P^{\bX}}}\left[\exp\left(\int_{0}^{T}\|b^{\bX}(\bX_t,t) - b^{\bY}(\bX_t,t)\|^2\,dt\right)\right] < \infty.
    \end{equation}
    Then the Radon-Nikodym derivative of $\mathbb{P^{\bX}}$ with respect to $\mathbb{P^{\bY}}$ is
    \begin{equation*}
        \frac{d\,\mathbb{P_{\bX}}}{d\,\mathbb{P_{\bY}}} (\bZ) = \exp\left(\frac{1}{2}\int_0^T \|b^{\bX}(\bZ_t,t) - b^{\bY}(\bZ_t,t)\|^2\,dt - \int_0^T \left(b^{\bX}(\bZ_t,t) - b^{\bY}(\bZ_t,t)\right)\, d\bB_t\right).
    \end{equation*}
    Therefore we have
    \begin{align*}
        \mathrm{D}_{\mathrm{KL}}\left(\mathbb{P_{\bX}} \,\|\,\mathbb{P_{\bY}}\right) &= \frac{1}{2}\,\mathbb{E}_{\mathbb{P_{\bX}}} \left[\int_0^T \|b^{\bX}(\bX_t,t) - b^{\bY}(\bX_t,t)\|^2\,dt \right]\\
        &= \frac{1}{2}\, \int_x\int_0^T p_t^{\bX}(x)\|b^{\bX}(x,t) - b^{\bY}(x,t)\|^2\,dt\, dx.
    \end{align*}
    \item If $\int_x\int_0^T p_t^{\bX}(x)\|b^{\bX}(x,t) - b^{\bY}(x,t)\|^2\,dt\, dx \leq C$ holds for some constant $C$, then
    \begin{equation*}
        \mathrm{D}_{\mathrm{KL}}\left(\mathbb{P_{\bX}} \,\|\,\mathbb{P_{\bY}}\right) \leq \frac{1}{2}\, \int_x\int_0^T p_t^{\bX}(x)\|b^{\bX}(x,t) - b^{\bY}(x,t)\|^2\,dt\, dx,
    \end{equation*}
    even if Novikov's condition \eqref{Novikov} is not satisfied.
\end{enumerate}
\end{lemma}

\subsection{Proof of Theorem \ref{main_theorem2}}\label{appendix_proof_main_theorem_2}
We first recall some notations of the stochastic processes we defined in Section~\ref{sec_background}. The forward process \eqref{BM_process} is denoted by $\left(\bX_t\right)_{t\in [0,T]}$ and $\bX_t \sim p_t$; The backward process \eqref{backward_sde_bm} is denoted by $\left(\bY_t\right)_{t\in [0,T]}$ and by definition, $\bY_t \sim p_{T-t}$; The process of Algorithm~\ref{algorithm1} is denoted by $(\hat \bY_t)_{t\in [0,T]}$.
Let $(\bar \bY_t)_{t\in [0,T]}$ be $(\hat \bY_t)_{t\in [0,T]}$ replacing $\hat \bY_0 \sim \mathcal{N}(0,T\boldsymbol{I}_d)$ by $\bar \bY_0 \sim p_T$, i.e., $(\bar \bY_t)_{t\in [0,T]}$ satisfies:
\begin{equation*}
    \mathrm{d} \bar \bY_t= \hat s_{T-t}(\bar \bY_t) \mathrm{d} t+\mathrm{d} \bB_t, \;\;\;\; \bar \bY_0 \sim p_T.
\end{equation*}
Now we start the proof. By the triangle inequality, 
\begin{align*}
    \mathbb{E}\left[\mathrm{TV}\left(\bX_0, \hat \bY_{T-t_0}\right)\right] &\leq \mathrm{TV}\left(\bX_0, \bX_{t_0}\right) + \mathrm{TV}\left(\bX_{t_0}, \bY_{T-t_0}\right)\\  &+ \mathbb{E}\left[\mathrm{TV}\left(\bY_{T-t_0}, \bar \bY_{T-t_0}\right)\right] 
    + \mathbb{E}\left[\mathrm{TV}\left( \bar \bY_{T-t_0} , \hat \bY_{T-t_0}\right)\right].
\end{align*}
By Theorem~\ref{theorem3}, $\mathrm{TV}\left(\bX_0, \bX_{t_0}\right) \lesssim \mathrm{polylog}(n) \, n^{-\frac{\beta}{2\beta+d}}$. By definition of the backward process \eqref{backward_sde_bm}, $\mathrm{TV}\left(\bX_{t_0}, \bY_{T-t_0}\right) = 0$. For the third term, by Pinsker's inequality and data-processing inequality,
\begin{equation*}
    \mathrm{TV}\left(\bY_{T-t_0}, \bar \bY_{T-t_0}\right) \lesssim \sqrt{\mathrm{D}_{\mathrm{KL}}\left(\bY_{T-t_0}\| \bar \bY_{T-t_0}\right)} \leq \sqrt{\mathrm{D}_{\mathrm{KL}}\left(\mathbb{P}_{\bY}\| \mathbb{P}_{\bar \bY}\right)},
\end{equation*} 
where $\mathbb{P}_{\bY}$ and $\mathbb{P}_{\bar \bY}$ are the path measure for $(\bY_t)_{t\in [0,T-t_0]}$ and $(\bar \bY_t)_{t \in [0,T-t_0]}$. Then using the second part of 
Lemma~\ref{Girsanov_lemma} by taking $b_t^{\bX} = \nabla \log p_{T-t}(\bx)$ and $b_t^{\bY} = \hat s_{T-t}(\bx)$ (see Section~\ref{Girsanov} for details), we have
\begin{align*}
\mathrm{D}_{\mathrm{KL}}\left(\mathbb{P}_{\bY}\| \mathbb{P}_{\bar \bY}\right) &\leq \frac{1}{2} \int_{0}^{T-t_0} \mathbb{E}_{\bx \sim p_t} \left[\|\hat s_{T-t}(\bx) - \nabla \log p_{T-t}(\bx)\|^2 \right]\,dt \\
&= \frac{1}{2} \int_{t_0}^{T} \mathbb{E}_{\bx \sim p_t}\left[ \|\hat s_t(\bx) - \nabla \log p_t(\bx)\|^2\right]\,dt.
\end{align*}
Then $\mathbb{E} \left[ \mathrm{D}_{\mathrm{KL}}\left(\mathbb{P}_{\bY}\| \mathbb{P}_{\bar \bY}\right) \right]$ can be bounded by Corollary~\ref{main_corollary}.
Since the only difference of $\bar \bY_{T-t_0}$ and $\hat \bY_{T-t_0}$ is their initial distribution, we have
\begin{equation*}
     \mathbb{E}\left[\mathrm{TV}\left( \bar \bY_{T-t_0} , \hat \bY_{T-t_0}\right)\right] \leq \mathrm{TV}\left( \bX_{T} , \mathcal{N}(0,T\boldsymbol{I}_d)\right).
\end{equation*}
Again, by Pinsker's inequality,
\begin{equation*}
    \mathrm{TV}\left( \bX_{T} , \mathcal{N}(0,T\boldsymbol{I}_d)\right) \leq \sqrt{\mathrm{D}_{\mathrm{KL}}\left( \bX_{T} \,\|\, \mathcal{N}(0,T\boldsymbol{I}_d)\right)}.
\end{equation*}
Using Jensen's inequality,
\begin{align*}
    \mathrm{D}_{\mathrm{KL}}\left( \bX_{T} \,\|\, \mathcal{N}(0,T\boldsymbol{I}_d)\right) &\leq \mathbb{E}_{\bx\sim p_0}\left[\mathrm{D}_{\mathrm{KL}}\left( \bX_{T} | \bX_{0} = \bx \,\| \,\mathcal{N}(0,T\boldsymbol{I}_d)\right) \right] \\
    &= \frac{1}{2T} \,\mathbb{E}_{\bx\sim p_0}\left[ \bx ^T \bx\right].
\end{align*}
Here in the last equality, we use the fact that $\bX_T | \bX_0 = \bx \sim \mathcal{N}(\bx,T\boldsymbol{I}_d)$ and the KL divergence between two Gaussian distributions. $\mathbb{E}_{\bx\sim p_0}\left[ \bx ^T \bx\right]$ is finite since $p_0$ is Sub-Gaussian. Therefore by taking $T = n^{\frac{2\beta}{2\beta+d}}$
\begin{equation*}
    \mathbb{E}\left[\mathrm{TV}\left( \bar \bY_{T-t_0} , \hat \bY_{T-t_0}\right)\right] \lesssim n^{-\frac{\beta}{2\beta+d}}.
\end{equation*}

\section{Auxiliary Results}\label{sec:Auxiliary Results}
\subsection{Sub-Gaussian Properties of \texorpdfstring{$p_t(x)$}{pt}}
We defined the Sub-Gaussian random vectors in Definition \ref{Sub-Gaussian_def}. Sub-Gaussian random vectors in higher dimensions also retain similar properties to those of 1-dimensional Sub-Gaussian random variables. The following lemma shows that the sum of two Sub-Gaussian random vectors is still Sub-Gaussian.
\begin{lemma}\label{Sub-Gaussian_sum} 
Let $X, Y \in \mathbb{R}^d$ be independent Sub-Gaussian random vectors with Sub-Gaussian norms $\|X\|_{\psi_2}$ and $\|Y\|_{\psi_2}$ respectively. Then $X+Y$ is Sub-Gaussian and 
\begin{equation*}
    \|X+Y\|_{\psi_2} \leq \|X\|_{\psi_2} + \|Y\|_{\psi_2}.
\end{equation*}
\end{lemma}
\begin{proof}
    We first prove that for $X, Y \in \mathbb{R}$, we have 
    \begin{equation*}
        \|X+Y\|_{\psi_2} \leq \|X\|_{\psi_2} + \|Y\|_{\psi_2}.
    \end{equation*}
    In fact, notice that the function $\psi_2(x) = e^{x^2} - 1$ is convex and increasing. Denote $\sigma_X := \|X\|_{\psi_2}$ and $\sigma_Y := \|Y\|_{\psi_2}$, then by triangle inequality and Jensen's inequality,
    \begin{align*}
        \psi_2\left(\frac{|X+Y|}{\sigma_X + \sigma_Y}\right) &\leq \psi_2\left(\frac{|X|}{\sigma_X + \sigma_Y} + \frac{|Y|}{\sigma_X + \sigma_Y}\right) \\
        &= \psi_2\left(\frac{\sigma_X}{\sigma_X + \sigma_Y} \frac{|X|}{\sigma_X} + \frac{\sigma_Y}{\sigma_X + \sigma_Y} \frac{|Y|}{\sigma_Y}\right)\\
        &\leq \frac{\sigma_X}{\sigma_X + \sigma_Y} \psi_2\left(\frac{|X|}{\sigma_X}\right) + \frac{\sigma_Y}{\sigma_X + \sigma_Y}\psi_2\left( \frac{|Y|}{\sigma_Y}\right).
    \end{align*}
    Take expectation of both sizes, and use the definition of Orlicz norm,
    \begin{align*}
        \mathbb E \left[\psi_2\left(\frac{|X+Y|}{\sigma_X + \sigma_Y}\right)\right] &\leq \frac{\sigma_X}{\sigma_X + \sigma_Y} \mathbb{E} \left[\psi_2\left(\frac{|X|}{\sigma_X}\right)\right] + \frac{\sigma_Y}{\sigma_X + \sigma_Y} \mathbb{E} \left[\psi_2\left( \frac{|Y|}{\sigma_Y}\right)\right] \\
        &\leq  \frac{\sigma_X}{\sigma_X + \sigma_Y} +  \frac{\sigma_Y}{\sigma_X + \sigma_Y} \\
        &= 1.
    \end{align*}
    Therefore,
    \begin{equation*}
        \|X+Y\|_{\psi_2} \leq \sigma_X + \sigma_Y.
    \end{equation*}
    In general case where $X, Y \in \mathbb{R}^d$, for any $v \in S^{d-1}$,
    \begin{align*}
        \|\langle X+Y, v \rangle \|_{\psi_2} &=  \|\langle X, v \rangle + \langle Y, v \rangle\|_{\psi_2} \\
        &\leq \|\langle X, v \rangle \|_{\psi_2} + \|\langle Y, v \rangle\|_{\psi_2} \\
        &\leq \|X\|_{\psi_2} + \|Y\|_{\psi_2}.
    \end{align*}
    The result follows from taking the supremum with respect to $v \in S^{d-1}$ of both sizes and the definition of Orlicz norm.
\end{proof}
The next lemma provides a tail bound for a Sub-Gaussian random vector, analogous to the one-dimensional tail bound.
\begin{lemma}\label{Sub-Gaussian_tail}
    Suppose $X \in \mathbb{R}^d$ is $\sigma$-Sub-Gaussian with $\mathbb{E} X = 0$, then for any $t \in \mathbb{R}$,
    \begin{equation*}
        \mathbb{P} \left(\|X\|_{\infty} > t\right) \le 2d\exp\left(-\frac{t^2}{2\sigma^2}\right).
    \end{equation*}
\end{lemma}
\begin{proof}
    Let $v=(1,0, \ldots,0) \in \mathbb{R}^d$. 
    \begin{align*}
         \mathbb{E} \exp\left(\frac{|X_1|^2}{\sigma^2}\right) = \mathbb{E} \exp\left(\frac{|\langle X, v \rangle|^2}{\sigma^2}\right) \le 2.
    \end{align*}
    Then $X_1$ is $\sigma$-Sub-Gaussian. Similarly, $X_i$ is $\sigma$-Sub-Gaussian for all $i = 1,\ldots,d$.
    Therefore,
    \begin{align*}
        \mathbb{P} \left(\|X\|_{\infty} > t\right) &\le 
        \sum_{i=1}^d \,\mathbb{P} \left(|X_i| > t\right) 
        \\
        &\le \sum_{i=1}^d \, 2\, e^{-\frac{t^2}{2\sigma^2}} 
        \\
        &= 2d\, e^{-\frac{t^2}{2\sigma^2}}.
    \end{align*}
\end{proof}
Let $X_0 \sim p_0$, $Z_t \sim \mathcal{N}(0,t\boldsymbol{I}_d)$ and $X_t = X_0 + Z_t \sim p_t$. The following results of the Sub-Gaussian property of $X_t$ can be directly seen from Lemma~\ref{Sub-Gaussian_sum}.
\begin{lemma}\label{Sub-Gaussianforpt} \textbf{[Sub-Gaussian property for $p_t(x)$]}
     Suppose $p_0$ is $\sigma_0$-Sub-Gaussian, and $p_t = p_0 * \phi_t$, where $\phi_t$ is the density of $\mathcal{N}(0,t\boldsymbol{I}_d)$ and $*$ denotes the convolution. Then $p_t$ is also Sub-Gaussian and for $X_t \sim p_t$,  $\|X_t\|_{\psi_2} \le \sigma_0+\sqrt{t}$. 
\end{lemma}
\begin{proof}
    Let $Z \sim \mathcal{N}(0,t\boldsymbol{I}_d)$. For any $v \in S^{d-1}$, since $\|v\|=1$, $\langle Z,v \rangle \sim \mathcal{N}(0,t)$. Using the fact that the Sub-Gaussian norm for a one-dimensional Gaussian random variable is equal to its square root of the variance, we have $\|\langle Z,v \rangle\|_{\psi_2} =\sqrt{t}$. Therefore,
    \begin{align*}
        \|Z\|_{\psi_2} = \sup_{v \in S^{d-1}} \| \langle Z, v \rangle \|_{\psi_2} = \sqrt{t}.
    \end{align*}
    If $X\sim p_0$, then $X+Z \sim p_t$. By Lemma~\ref{Sub-Gaussian_sum}, 
    \begin{equation*}
        \|X+Z\|_{\psi_2} \le \|X\|_{\psi_2} + \|Z\|_{\psi_2} = \sigma_0 + \sqrt{t}.
    \end{equation*}
\end{proof}
\subsection{Bounds on \texorpdfstring{$\|D^{\alpha}p_t(x)\|_{\infty}$}{Derivative of pt}}
The following lemma provides an upper bound for any derivatives of $p_t$.
\begin{lemma}\label{bound_on_pt}
    Supposed $\|p_0\|_{\infty}:=\sup_{x} p_0(x) < \infty$, and let $\alpha = (\alpha_1,\ldots, \alpha_d)$ with $\ell = \sum_{i=1}^d \alpha_i$, then
    \begin{equation*}
        \sup_x |D^{\alpha}p_t(x)| \leq C_1^d \|p_0\|_{\infty} \ell^{\frac{\ell}{2} + \frac{1}{4}} t^{-\ell/2},
    \end{equation*}
    for some universal constant $C_1$.
\end{lemma}
\begin{proof}
Denote $ \phi_t (x) := \frac{1}{(2\pi t)^{d/2}}\exp\left\{-\frac{\|x\|^2}{2t}\right\} $ as the Gaussian density function with variance $t\boldsymbol{I}_d$.
\begin{align*}
    |D^{\alpha} p_t(x)| 
    &
    \le \int_y |p(y) D^{\alpha} \phi_t(x-y)|dy \\
    & \le \|p_0\|_{\infty} \int_x | D^{\alpha} \phi_t(x)|dx.
\end{align*}    
Next, we establish an upper bound for $\int_x | D^{\alpha} \phi_t(x)|dx = \int_x | \prod_{i=1}^d\frac{\partial^{\alpha_i}}{\partial x_i^{\alpha_i}} \phi_t(x)|dx$ using the Hermite polynomials.
The Hermite polynomials of order $n$ are defined as follows:
\begin{equation}\label{Hermite}
    H_n(x):=(-1)^n e^{x^2} \frac{d^n}{dx^n}e^{-x^2}.
\end{equation}
Then we have:
\begin{equation*}
    D^{\alpha} \phi_t(x) = \left(-\frac{1}{\sqrt{2t}}\right)^{|\alpha|} \prod_{i=1}^d H_{\alpha_i}\left(\frac{x_i}{\sqrt{2t}}\right) \phi_t(x).
\end{equation*}
Using Lemma~\ref{Hermiteinequality} for Hermite polynomials, we have
\begin{equation*}
    |D^{\alpha} \phi_t(x)| \le C^d t^{-\ell/2} \left(\prod_{i=1}^d \alpha_i!\right)^{1/2} e^{\|x\|^2/4t} \phi_t(x),
\end{equation*}
Therefore, using $\prod_{i=1}^d \alpha_i! \le (\sum_{i=1}^d \alpha_i)! = \ell !$, we have,
\begin{align*}
    \int_x |D^{\alpha} \phi_t(x)|dx 
    &\le C^d t^{-\ell/2} (\ell!)^{1/2} \int_x e^{\|x\|^2/4t} \phi_t(x)dx 
    \\
    &= C^d t^{-\ell/2} (\ell!)^{1/2} \int_x \frac{1}{(2\pi t)^{d/2}} e^{-\|x\|^2/4t}dx 
    \\
    &= C^d 2^{d/2} (\ell!)^{1/2} t^{-\ell/2}.
\end{align*}
By Stirling's approximation,
\begin{equation*}
    n! \le \sqrt{2\pi n} \left(\frac{n}{e}\right)^n e^{\frac{1}{12n}},
\end{equation*}
we have
\begin{equation*}
    (\ell!)^{1/2} \le (2\pi)^{1/4}\;e^{\frac{1-12{\ell^2}}{24 \ell}} \;\ell^{\frac{\ell}{2} + \frac{1}{4}} \le (2\pi)^{1/4}\,\ell^{\frac{\ell}{2} + \frac{1}{4}}.
\end{equation*}
Therefore,
\begin{align*}
    \int_x |\phi^{(\ell)}_t(x)|dx 
    &\le C^d 2^{d/2} (2\pi)^{1/4} \; \ell^{\frac{\ell}{2} + \frac{1}{4}} t^{-\ell/2}.
\end{align*}
\end{proof}
\begin{lemma}\label{Hermiteinequality}\citet{indritz1961inequality}
    The Hermite polynomials defined in \eqref{Hermite} satisfies
    \begin{equation*}
        H_n(x) \le C \left(2^n n!\right)^{1/2} e^{x^2/2}
    \end{equation*}
    for some universal constant $C$.
\end{lemma}
\subsection{Concentration Inequality for 
\boldmath\texorpdfstring{$\hat{p}_t$}{phat}}\label{appendix_concentration}
\begin{lemma}\label{high_probability_bound}
    There exist a constant $C_4$ that only depends on $p_0$ and dimension $d$ such that the kernel density estimator $\hat p_t (x)$ with the kernel of order $\ell$ we defined in Lemma~\ref{Bound_on_kernel} and the choice of $h = \frac{\sqrt{t}}{D_n}$ satisfies
    \begin{equation*}
        \left|\hat p_t(x) - p_t(x)\right|< C_4\left( \sqrt{\frac{D_n^d\, \ell^{3d}\, p_t^*(x) \log(1/\delta)}{nt^{\frac{d}{2}}}} +\frac{D_n^d \ell^{5d/2} \log(1/\delta)}{nt^{\frac{d}{2}}}+\left(\frac{D_n}{de}\right)^{-\ell} \ell^{- \frac{\ell}{2} +\frac{5d}{2} -\frac{1}{4}}\right),
    \end{equation*}
    with probability at least $1-\delta$,
    where $p_t^*(x):=\sup_{\|\lambda\|_{\infty} < h}p_t(x+\lambda)$.
    Furthermore, by choosing $\ell=\log n$, $D_n = C\sqrt{\log n}$ and $\delta = n^{-\alpha}$ for any positive constant $\alpha$, there exists a constant $C_5(\alpha)$ that depends on $p_0$, $d$ and $\alpha$, with probability at least $1-n^{-\alpha}$,
    \begin{equation*}
        \left|\hat p_t(x) - p_t(x)\right|< C_5(\alpha)\; \mathrm{polylog}(n)\left(\sqrt{\frac{p_t^*(x)}{nt^{\frac{d}{2}}}} + \frac{1}{ nt^{\frac{d}{2}}} + (\log n)^{-\log n}\right).
    \end{equation*}
\end{lemma}
\begin{proof}
Take $Y_i = \frac{1}{h^d}K\left(\frac{x-X_i}{h}\right)$, then
\begin{equation*}
    \left|\frac{1}{n}\sum_{i=1}^n Y_i - \mathbb{E}\, Y\right| = \left|\hat{p}_t (x) - \mathbb{E}\,\hat p_t(x)\right|,
\end{equation*}
Using Lemma~\ref{Bound_on_kernel}, we have
\begin{equation*}
    \left|Y_i\right| \le \frac{1}{h^d} \|K_d\|_{\infty} \lesssim \frac{\ell^{\frac{5d}{2}}}{h},
\end{equation*}
\begin{equation*}
    \mathrm{Var}\left(Y_i\right) = \frac{1}{h^{2d}}\mathrm{Var}\left(K_d\left(\frac{x-X_i}{h}\right)\right) \le \frac{1}{h^d} \int K_d^2(u)du\sup_{\|\lambda\|_{\infty} < h}p_t(x+\lambda) \lesssim \frac{\ell^{3d}}{h^d}\; p_t^*(x),
\end{equation*}
where $p_t^*(x):=\sup_{\|\lambda\|_{\infty} < h}p_t(x+\lambda)$. Next we use the following Bernstein's inequality.
\begin{proposition}[Bernstein's inequality]
Suppose that $Y_i$ are iid with mean $\mu$, $\mathrm{Var}\left(Y_i\right)\leq \sigma^2$ and $\mathbb{P}\left(Y_i \leq M\right) = 1$. Then
\begin{equation*}
    \mathbb{P}\left(\left|\frac{1}{n}\sum_{i=1}^n Y_i - \mu\right|>\varepsilon\right) \leq 2 \exp\left\{-\frac{n\varepsilon^2}{2\sigma^2+2M\varepsilon/3}\right\}.
\end{equation*}
Furthermore, with probability at least $1-\delta$,
\begin{equation*}
    \left|\frac{1}{n}\sum_{i=1}^n Y_i - \mu\right| < \sqrt{\frac{2\sigma^2 \log(1/\delta)}{n}} + \frac{2M\log(1/\delta)}{3n}.
\end{equation*}
\end{proposition}
By Bernstein's inequality,
\begin{equation}\label{eq218}
    \mathbb{P}\left(\left|\hat p_t(x) - \mathbb{E} \,\hat p_t(x)\right| < \sqrt{\frac{2\, \ell^{3d}\, p_t^*(x) \log(1/\delta)}{nh^d}} +\frac{2 \ell^{5d/2} \log(1/\delta)}{3nh^d}\right) > 1-\delta
\end{equation}
By triangle inequality,
\begin{align}\label{eq219}
    |\hat p_t(x) - p_t(x)| &\leq |\hat p_t(x) - \mathbb{E} \,\hat p_t(x)| + |\mathbb{E} \,\hat p_t(x) - p_t(x)|
\end{align}
Using the result of bias term in Proposition~\ref{MSEphat}, and by taking $h = \frac{\sqrt{t}}{ D_n}$, we have
\begin{align*}
    |\mathbb{E} \,\hat p_t(x) - p_t(x)| \leq  
    C_2 \left(\frac{D_n}{de}\right)^{-\ell} \ell^{- \frac{\ell}{2} +\frac{5d}{2} -\frac{1}{4}}.
\end{align*}
Therefore, there is a constant $C_4$ such that with probability at least $1-n^{-\alpha}$,
\begin{equation*}
    \left|\hat p_t(x) - p_t(x)\right|< C_4 \left( \sqrt{\frac{D_n^d\, \ell^{3d}\, p_t^*(x) \log(1/\delta)}{nt^{\frac{d}{2}}}} +\frac{D_n^d \ell^{5d/2} \log(1/\delta)}{nt^{\frac{d}{2}}}+\left(\frac{D_n}{de}\right)^{-\ell} \ell^{- \frac{\ell}{2} +\frac{5d}{2} -\frac{1}{4}}\right).
\end{equation*}
By choosing the order $\ell=\Omega(\log n)$, $D_n = C\sqrt{\log n}$ and $\delta = n^{-\alpha}$ for some positive constant $\alpha$, we have with probability at least $1-n^{-\alpha}$,
\begin{equation*}
    \left|\hat p_t(x) - p_t(x)\right|< C_5(\alpha)\; \mathrm{polylog}(n)\left(\sqrt{\frac{p_t^*(x)}{nt^{\frac{d}{2}}}} + \frac{1}{nt^{\frac{d}{2}}} + (\log n)^{-\log n}\right).
\end{equation*}
\end{proof}
\subsection{Bounds on the Tail Density}
In this section, we provide the proofs for the bound of $|G_1|$ and $\int_{G_1} \|s_t(x)\|^2dx$ in the proofs in Lemma~\ref{main_lemma1}, where $G=\{x\colon p_t(x) > \rho_n \log^c {n} \}$ and $\rho_n = \frac{1}{n \, t^{d/2}}$. 
To begin with, the following lemma shows that for sub-Gaussian density function $p_t$, for sufficiently large $x$ satisfying $\|x\|_{\infty} \gtrsim \sqrt{\log n}$, the tail density can be controlled. 
Specifically, we show that $p_t(x)$ is bounded above by a polynomial decay in $n^{-1}$: $p_t(x) \lesssim \mathrm{poly}(n^{-1})$. This lemma is crucial for the proof of upper bounds on $|G_1|$ and $\int_{G_1} \|s_t(x)\|^2dx$.
\begin{lemma}\label{G_d_dim}
    For any $D > 0$ and $t > n^{-T_1}$, if $x \in \mathbb{R}^d$ satisfies $\|x\|_{\infty} \ge D \sigma_t \sqrt{\log n}$ and $n \ge \exp \left( \frac{\|p_0\|^2_{\infty}}{4D^2}\right)$, where $\sigma_t = \sigma_0 + \sqrt{t}$, then
    \begin{equation*}
        p_t(x) \leq 2 \sqrt{2} n^{-\frac{D}{2} + \frac{dT_1}{2}}
    \end{equation*}
\end{lemma}
\begin{proof}
    \textbf{Step 1}: Suppose $d=1$.
    From Lemma~\ref{Sub-Gaussianforpt}, $X_t\sim p_t(x)$ is Sub-Gaussian with Orlicz norm $\|X_t\|_{\psi_2} \leq \sigma_t = \sigma_0 + \sqrt{t}$. Then by the sub-Gaussian property in Definition \ref{Sub-Gaussian_def},
    \begin{equation*}
    \int e^{\frac{x^2}{\sigma_t^2}} p_t(x) d x \le 2.
    \end{equation*}
    Consider two cases. If $t\ge 1$, fix $x_0$ that satisfies $|x_0| > D\sigma_t \sqrt{\log n}$. Define the following two sets:
    \begin{equation*}
        \mathcal{A}:=\left\{x \colon |x-x_0|<\frac{D\sigma_t}{2} \sqrt{\log n}\right\},
    \end{equation*}
    \begin{equation*}
        \mathcal{B}:=\left\{x \colon |x-x_0|<\frac{p_t(x_0) \sqrt{t}}{4} \right\},
    \end{equation*}
    If $n>\exp \left(\frac{\|p_0\|_\infty^2}{4D^2}\right)$, then $\frac{D\sigma_t}{2}\sqrt{\log n} \ge \frac{D\sigma_t\|p_0\|_\infty}{4D} \geq \frac{p_t(x_0) \sqrt{t}}{4}$, where the last inequality is due to the fact that
    \begin{equation*}
        p_t(x) = \int p(y) \phi_t(x-y)dy \leq \|p_0\|_\infty\;\;\;\;\; \text{for any } x. 
    \end{equation*}
    Therefore $\mathcal{B} \subset \mathcal{A}$. Next, we use the Lipschitz property of $p_t$ to control $p_t(x)$ for all $x \in \mathcal{B}$. Specifically, for any $x < x'$,
    \begin{align}
        |p_t(x) - p_t(x')| &\leq \int p_0(y)|\phi_t(x-y)- \phi_t(x'-y)|dy \notag
        \\
        &\leq |x-x'| \sup_x |\phi_t'(x)| \notag
        \\
        &\leq \frac{1}{t}|x-x'|,\label{eq261}
    \end{align}
    where in the second inequality we use the mean value theorem, and in the third inequality we use the fact that
    \begin{align*}
        \sup_x |\phi_t'(x)| = \frac{1}{t\sqrt{2\pi t} }  \sup_x \left\{ x \exp\left({-\frac{x^2}{2t}}\right)\right\} = \frac{e^{-1/2}}{\sqrt{2\pi}} \,t^{-1}.
    \end{align*}
    Then for any $y \in \mathcal{B}$, using \eqref{eq261}, since $t>1$, we have
    \begin{align}
        p_t(y) &\ge p_t(x_0) - \frac{1}{t} |x_0-y| \notag
        \\
        &\geq p_t(x_0) - \frac{1}{\sqrt t} |x_0-y| \notag
        \\
        &\ge p_t(x_0) -  \frac{1}{\sqrt t} \frac{p_t(x_0 ) \sqrt{t}}{2} \notag \\
        &= \frac{1}{2}\, p_t(x_0). \label{eq264}
    \end{align}
    Then using the definition of Sub-Gaussian and \eqref{eq264},
    \begin{align*}
    2 \geq \int e^{x^2/\sigma_t^2} p_t(x) \,d x 
    & \ge  \inf_{x\in\mathcal{A}}e^{x^2/\sigma_t^2} \int_{\mathcal{A}} p_t(x)\,dx
    \\
    & \ge e^{\frac{D^{2}}{4} \log n} \;
     \int_{\mathcal{B}} p_t(x)\,dx \\
    & \ge n^{D^2/4} \;|\mathcal{B}| \;\frac{1}{2} p_t(x_0) \\
    & = \frac{1}{4}\;n^{D^2/4} p_t^2(x_0) \sqrt{t}.
    \end{align*}
    This implies that
    \begin{align*}
        p_t(x_0) \leq 2\sqrt{2} n^{-D/2} t^{-1/4} \leq 2\sqrt{2} n^{-D/2}. 
    \end{align*}
    
    If $n^{-\frac{2}{2\beta + d}} < t < 1$, define
    \begin{equation*}
        \mathcal{C}:=\left\{x \colon |x-x_0|<\frac{p_t(x_0) t}{4} \right\}.
    \end{equation*}
    Similarly, if $n>\exp \left(\frac{\|p_0\|_\infty^2}{4D^2}\right)$, then $\frac{D\sigma_t}{2}\sqrt{\log n} \ge \frac{D\sigma_t\|p_0\|_\infty}{4D} \geq \frac{p_t(x_0) \sqrt{t}}{4}\geq \frac{p_t(x_0) 
     \,t}{4}$ since $t<1$. Therefore $\mathcal{C} \subset \mathcal{A}$. Then for any $y \in \mathcal{C}$, using \eqref{eq261},
    \begin{align}
        p_t(y) &\geq p_t(x_0) - \frac{1}{t} |x_0-y| \notag
        \\
        &\ge p_t(x_0) -  \frac{1}{t} \frac{p_t(x_0 ) t}{2} \notag
        \\
        &= \frac{1}{2} p_t(x_0). \label{eq285}
    \end{align}
    Then using the definition of Sub-Gaussian and \eqref{eq285},
    \begin{align*}
    2 \geq \int e^{x^2/\sigma_t^2} p_t(x) \,d x 
    & \ge  \inf_{x\in\mathcal{A}}e^{x^2/\sigma_t^2} \int_{\mathcal{A}} p_t(x)\,dx
    \\
    & \ge e^{\frac{D^{2}}{4} \log n} \;
     \int_{\mathcal{C}} p_t(x)\,dx \\
    & \ge n^{D^2/4} \;|\mathcal{C}| \;\frac{1}{2} p_t(x_0) \\
    & = \frac{1}{4}\;n^{D^2/4} p_t^2(x_0) t.
    \end{align*}
    This implies that
    \begin{align*}
        p_t(x_0) \leq 2\sqrt{2} n^{-D/2} t^{-1/2} \leq 2\sqrt{2} n^{-\frac{D}{2} + \frac{T_1}{2}}. 
    \end{align*}
    
    \noindent\textbf{Step 2}: In general case $d \ge 1$,
    since $\|x\|_{\infty} \ge D \sigma_t \sqrt{\log n}$, without loss of generality, we suppose $|x_1| \ge D \sigma_t \sqrt{\log n}$. Then by definition of $p_t(x)$,
    \begin{align*}
        p_t(x) &= \int_{y \in \mathbb {R}^d} p_0(y) \prod_{i=1}^d \phi_t(x_i-y_i) dy \\
        &\le (2\pi t)^{-\frac{d-1}{2}} \int_{y \in \mathbb {R}} p_0(y) \phi_t(x_1-y) dy \\
        &= (2\pi t)^{-\frac{d-1}{2}} p_t(x_1).
    \end{align*}
    Using the 1-dimensional results, $p_t(x_1) \leq 2\sqrt{2} \;n^{-\frac{D}{2} +  \frac{1}{2\beta + d} \mathbbm{1}_{t<1}}$, and use the fact that $t > n^{-T_1}$, we have
    \begin{equation*}
        p_t(x) \leq 2 \sqrt{2} n^{-\frac{D}{2} + \frac{dT_1}{2}}.
    \end{equation*}
\end{proof}
\begin{lemma}\label{Bound_for_G}
    Let $G=\{x\colon p_t(x) > \rho_n \log^c {n} \}$ where $\rho_n = \frac{1}{n \, t^{d/2}}$ and $n^{-T_1}<t < n^{T_2} $, 
    \begin{align*}
        |G| \le D (t^{\frac{d}{2}} + \sigma_0^d) (\log n)^{\frac{d}{2}}
    \end{align*}
    for some constant $D$ only depending on $d$, $T_1$ and $T_2$.
\end{lemma}
\begin{proof}
    For any $t\ge n^{-T_1}$ and constant $C > 0$. Let $D = 2C + 2dT_1$. Then by Lemma~\ref{G_d_dim}, for any $n \ge \exp\left(\frac{\|p_0\|_{\infty}^2}{4D^2}\right)$,
    \begin{align*}
       p_t(x) \ge 2\sqrt{2} \;n^{-C} 
    \end{align*}
    implies
    \begin{align*}
        \|x\|_{\infty} \le D\,\sigma_t \;\sqrt{\log n} \le \,D\, (\sqrt{t} + \sigma_0) \;\sqrt{\log n}.
    \end{align*}
    Since 
    $\rho_n = n^{-1}t^{-d/2} \ge n^{-(1+\frac{dT_2}{2})}$, we can take $C = 1+\frac{dT_2}{2}$ and conclude that
    \begin{align*}
        |G| &\le |\{x: \|x\|_{\infty} \le \,D\, (\sqrt{t} + \sigma_0) \;\sqrt{\log n}\}| \\
        &\lesssim  (\sqrt{t} + \sigma_0)^d (\log n)^{\frac{d}{2}} \\
        &\le 2^{d-1}(t^{\frac{d}{2}} + \sigma_0^d) (\log n)^{\frac{d}{2}}. 
    \end{align*}
\end{proof}
Then using the bound for expectation of moment of score function $\mathbb{E}\left[\|s_t(X)\|^m\right]$ in Lemma~\ref{s_x_bound} and combining \autoref{Bound_for_G}, we can further provide an upper bound for 
\begin{equation}
   \int_{G} \|s_t(x)\|^2 dx.
\end{equation}
\begin{lemma}\label{bound_on_GS}
    For any $m>1$,  define the set
    \begin{equation*}
        G := \{x: p_t(x) > \rho_n\,\log^c {n}\},
    \end{equation*}
    then
    \begin{equation*}
        \int_{G} \|s_t(x)\|^2 dx \le C_6\, m  \; (\log {n})^{-\frac{c}{m} }\rho_n^{-\frac{1}{m}}\,t^{-1}\,(t^{\frac{d}{2}} + \sigma_0^d)\, (\log n)^{\frac{d}{2}},
    \end{equation*}
    for some constant $C_6$ depends on $D_1$ and $p$.
\end{lemma}
\begin{proof}
    By H\"older's inequality,
    \begin{align*}
        \int_{G} \|s_t(x)\|^2 dx &= \int_{G} \|s_t(x)\|^2 \frac{1}{p_t(x)}p_t(x)dx \\
        &\leq \left(\int_{G}\|s_t(x)\|^{2m}p_t(x)dx\right)^{1/m}\left(\int_{G}p_t(x)^{-m'}p_t(x)dx\right)^{1/m'},
    \end{align*}
    where $m,m'>1$ and satisfies $\frac{1}{m}+\frac{1}{m'}=1$.
    By Lemma~\ref{s_x_bound},
    \begin{equation*}
    \left(\int_{G}\|s_t(x)\|^{2m}p_t(x)dx\right)^{1/m} = \left(\mathbb{E}\left[\|s_t(X)\|^{2m}\right]\right)^{1/m}\leq t^{-1} \left((2m-1)!!\right)^{1/m} \lesssim m\; t^{-1},
    \end{equation*}
    where in the last inequality we use Stirling's approximation.
    By \autoref{Bound_for_G},
    \begin{align*}
        \left(\int_{G}p_t(x)^{-m'}p_t(x)dx\right)^{1/m'} &= \left(\int_{G}p_t(x)^{1-m'}dx\right)^{1/m'} \\
        &\leq \left(|G| (\log^c {n})^{1-m'}\rho_n^{1-m'}\right)^{1/m'}\\
        &\leq \left(D_1 (t^{\frac{d}{2}} + \sigma_0^d) (\log n)^{\frac{d}{2}} \right)^{1-1/m} (\log^c {n})^{-\frac{1}{m} }\rho_n^{-\frac{1}{m}}\\
        &\leq D_1 (t^{\frac{d}{2}} + \sigma_0^d) (\log n)^{\frac{d}{2}} (\log {n})^{-\frac{c}{m} }\rho_n^{-\frac{1}{m}}.
    \end{align*}
\end{proof}

\subsection{Bounds on the Squared Error}
The next proposition provides uniform upper bounds over $x\in \mathbb{R}$ for 
\begin{equation*}
    \|\nabla\hat{p}_t(x) - \nabla p_t(x)\|^2
\end{equation*}
and
\begin{equation*}
    \|\nabla\hat{p}_t(x)\|^2 \left|\hat p_t(x) -p_t(x) \right|^2.
\end{equation*}
\begin{lemma}\label{upper_bound_for_p_minus_phat}
    There exists  constants $C_5$ depending on $d$ and $p_0$ such that with probability 1,
    \begin{equation*}
        \sup_{x} \|\nabla\hat{p}_t(x) - \nabla p_t(x)\|^2 \leq C_5 \left( \|p_0\|_{\infty}^2 t^{-1} + \ell^{5d} h^{-(2d+2)} \right), 
    \end{equation*}
    \begin{equation*}
        \sup_{x} \|\nabla p_t(x)\|^2 \left|\hat p_t(x) -p_t(x) \right|^2 \leq C_5 \left( \|p_0\|_{\infty}^2 t^{-1}(\ell^{5d} h^{-2d} + 2 \|p_0\|_{\infty}^2) \right).
    \end{equation*}
    In particular, if $h = \frac{\sqrt{t}}{D_n}$, $D_n = C\sqrt{\log n}$ and $\ell = \log n$, then the two terms above are $O(t^{-2}\mathrm{poly}(\log n ))$.
    Furthermore, if $\log\frac1{t}=O(\log n)$, then the two terms above are at most polynomial in $n$.
\end{lemma}
\begin{proof}
By Lemma~\ref{bound_on_pt},
\begin{align*}
    \| \nabla p_t(x)\|^2 &= \sum_{i=1}^d \left|\frac{\partial}{\partial x_d} p_t(x)\right|^2 \leq d\, C_1^{2d}\|p\|_{\infty}^2 t^{-1}.
\end{align*}
For the derivative of kernel density estimation,
\begin{equation*}
\nabla \hat{p}_t(x) = \frac{1}{nh^{d+1}}\sum_{i=1}^n \nabla K_d\left(\frac{x-X_i^t}{h}\right),
\end{equation*}
By Lemma~\ref{Bound_on_kernel},
\begin{equation*}
   \|\nabla \hat{p}_t(x)\|^2 \le \frac{d\|\nabla K_d(\cdot)_1\|_{\infty}^2}{h^{2d+2}}\lesssim \frac{d\ell^{5d}}{h^{2d+2}}. 
\end{equation*}
Therefore
\begin{equation*}
    \|\nabla\hat{p}_t(x) - \nabla p_t(x)\|^2 \leq 2d\, C_1^{2d}\|p\|_{\infty}^2 t^{-1} + 2 C_2 \,\frac{d\ell^{5d}}{h^{2d+2}}.
\end{equation*}
Besides, since
\begin{equation*}
    \hat p_t(x) = \frac{1}{nh^d} \sum_{i=1}^n K_{d}\left(\frac{x-X_i^t}{h}\right), 
\end{equation*}
we have
\begin{equation*}
        |\hat p_t(x) - p_t(x)|^2 \leq 2 \|K\|_{\infty}^2h^{-2s} + 2 \|p\|_{\infty}^2 \lesssim \ell^{5d}h^{-2d} + 2 \|p\|_{\infty}^2.
\end{equation*}
\end{proof}
Moreover, the fourth moments of the score error $\mathbb{E}_{x \sim p_t, \{x_i\}_{i=1}^n}\|\hat s_t(x) - s_t(x)\|^4$ can also be bounded.
\begin{lemma}\label{bound_on_4_score_error}
    There exists some constants $C_6$ depending on $d$ and $p_0$ such that
    \begin{equation*}
        \mathbb{E}_{X \sim p_t, \{x_i\}_{i=1}^n}\|\hat s_t(X) - s_t(X)\|^4 \le C_6 \left(t^{-2}+\rho_n^{-4} h^{-4(d+1)}\right).
    \end{equation*}
    In particular, if $h = \frac{\sqrt{t}}{D_n}$, $D_n = \sqrt{\log n}$, $\ell = \log n$ and
    $\log\frac1{t}=O(\log n)$,
    then the right side is at most polynomial in $n$.
\end{lemma}
\begin{proof}
    \begin{align*}
        \mathbb{E}_{X \sim p_t, \{x_i\}_{i=1}^n}\|\hat s_t(X) - s_t(X)\|^4 &\lesssim \mathbb{E}_{X \sim p_t, \{x_i\}_{i=1}^n}\|\hat s_t(X)\|^4 + \mathbb{E}_{X \sim p_t}\|s_t(X)\|^4.
    \end{align*}
    Lemma~\ref{s_x_bound} shows that the moment of score is bounded:
    \begin{equation*}
        \mathbb{E}_{X \sim p_t}\|s_t(X)\|^4 \lesssim t^{-2} d^{2}.
    \end{equation*}
    Using the definition of $\hat s_t(x) = \frac{\nabla \hat p_t(x)}{\hat p_t(x)} \mathbbm1_{\hat p_t(x) >  \rho_n}$ and Lemma~\ref{Bound_on_kernel},
    \begin{align*}
        \mathbb{E}_{X \sim p_t, \{x_i\}_{i=1}^n}\|\hat s_t(X)\|^4 &\le \rho_n^{-4} \mathbb{E}_{X \sim p_t, \{x_i\}_{i=1}^n}\left\| \nabla \hat p_t(X) \right\|^4 
        \\
        &\le \rho_n^{-4} \left\| \frac{1}{nh^{d+1}}\sum_{i=1}^d \nabla K_d\left(\frac{x-X_i^t}{h}\right)\right\|^4 
        \\
        &\lesssim \rho_n^{-4} h^{-4(d+1)} \ell^{5d}.
    \end{align*}
\end{proof}
\subsection{Control of the Score moments via the R\'enyi Entropies}\label{sec_renyi}
\subsubsection{R\'enyi entropies for Sub-Gaussian Random vector}
\begin{definition}
    The $\alpha$-R\'enyi entropy for random vector $X$ with distribution $p$ is defined as
    \begin{equation*}
        h_\alpha(X) := \frac{1}{1-\alpha} \log \int (p(x))^\alpha \, dx.
    \end{equation*}
\end{definition}
The next proposition shows that sub-Gaussian random variables have finite R\'enyi entropy and that the maximum R\'enyi entropy is attainable by the Gaussian distribution.
\begin{lemma}\label{Renyi_bound_1d}
Suppose that $X \in \mathbb{R}$ is $\sigma$-Sub-Gaussian random variable.
Let $\alpha\in(0,1)$ and $p$ denotes the density of $X$. Then
\begin{align*}
h_{\alpha}(X)\le 
\frac{1}{1-\alpha}\log \left(\left(2\pi\sigma^2\right)^{\frac{1-\alpha}{2}}  (1-\alpha)^{\frac{1-\alpha}{2}}\alpha^{-\frac{1}{2}}\right).
\end{align*}
\end{lemma}
\begin{proof}
By the sub-Gaussian property in 1 dimension, $X$ satisfies
\begin{equation}\label{subgaussian_constrain}
    \mathbb{E}[e^{X^2/\sigma^2}]\le 2.
\end{equation}
The map $p(\cdot)\mapsto \int p^{\alpha}dx$ is concave for $\alpha\in(0,1)$, so the maximum of $h_{\alpha}(X)$ under the given Sub-Gaussian constraint \eqref{subgaussian_constrain} is achieved at the stationary point via calculus of variations. Namely, define the Lagrangian with Lagrange multipliers $\eta >0$:
\begin{equation*}
    \mathcal{L}(p(x), \eta) = \int p(x)^\alpha \, dx + \eta \left( 2 - \int p(x) \exp\left(\frac{x^2}{\sigma^2}\right) \, dx \right).
\end{equation*}
The functional derivative of  $\mathcal{L}$ with respect to $p(x)$ is:
\begin{equation*}
    \frac{\delta \mathcal{L}}{\delta p(x)} = \alpha p(x)^{\alpha - 1} - \eta \exp\left(\frac{x^2}{\sigma^2}\right) = 0.
\end{equation*}
Therefore the maximum of $h_\alpha(X)$ is attained by
\begin{align*}
\tilde p(x)=Z e^{-\frac{x^2}{(1-\alpha)\sigma^2}} 
\end{align*} 
for some normalizing constant $Z>0$. 
As $p$ is a probability measure, we can verify that $Z=\frac1{\sqrt{\pi(1-\alpha)}\sigma}$. Therefore, for all $\sigma$-sub-Gaussian random variable $X$,
\begin{equation*}
    h_{\alpha}(X) \le \frac{1}{1-\alpha} \log \int (\tilde p(x))^\alpha \, dx = \frac{1}{1-\alpha}\log \left(\left(2\pi\sigma^2\right)^{\frac{1-\alpha}{2}}  (1-\alpha)^{\frac{1-\alpha}{2}}\alpha^{-\frac{1}{2}}\right).
\end{equation*}
\end{proof}
Next, we generalize this result to $d$-dimension $\sigma$-sub-Gaussian random vector $X$ using the following property of R\'enyi entropy.
\begin{lemma}\label{Renyi_bound}
Suppose that $X \in \mathbb{R}^d$ is $\sigma$-Sub-Gaussian random vector.
Let $\alpha\in(0,1)$ and $p$ denotes the density of $X$. Then for $\alpha \rightarrow 0$, we have
\begin{align*}
h_{\alpha}(X)\le 
\log \left( C_d \sigma^{d} \alpha^{-\frac{d}{2}}\right)
\end{align*}
for some constant $C_d$ that only depends on the dimension $d$.
\end{lemma}
\begin{proof}
    We define a transformation $T: \mathbb{R}^d \to S^{d-1} \times \mathbb{R}^+$ by $T(x) = (u, r)$, where $u = x / \|x\|$ and $r = \|x\|$. This transformation converts Cartesian coordinates to spherical coordinates, with $u$ representing the direction and $r$ the radius. The determinant of the Jacobian of this transformation is $|J_{T^{-1}}(r,u)| = r^{d-1}$. Therefore, by change of variable, the joint distribution of $(U,R) = (\frac{X}{\|X\|}, \|X\|)$ is given by
    \begin{equation*}
        p_{U,R}(u,r) = p_{X}(T^{-1}(u,r)) |J_{T^{-1}}(r,u)| = p_X(ur)r^{d-1}.
    \end{equation*}
    Next, we formulate the R\'enyi entropy of $X$ in terms of the joint distribution of $(X/\|X\|,\|X\|)$:
    \begin{align*}
        h_\alpha(X) 
        &= \log \left(\mathbb{E}\left[\left(\frac{1}{p_X(X)}\right)^{1-\alpha}\right]\right)^{\frac{1}{1-\alpha}}
        \\
        &= \log \left(\mathbb{E}\left[\left(\frac{\|X\|^{d-1}}{p_{U,R}(X/\|X\|,\|X\|)}\right)^{1-\alpha}\right]\right)^{\frac{1}{1-\alpha}}
        \\
        &\le \log \left(\mathbb{E}\left[\left(\frac{1}{p_{U,R}(X/\|X\|,\|X\|)}\right)^{p}\right]\right)^{\frac{1}{p}} + \log \left(\mathbb{E}\left[\left(\|X\|^{(d-1)}\right)^{q}\right]\right)^{\frac{1}{q}} 
        \\
        &= h_{1-p}\left(\frac{X}{\|X\|},\|X\|\right) + \frac{1}{q} \log \mathbb{E}\left[\|X\|^{q(d-1)}\right],
    \end{align*}
    where in the inequality term we apply the H\"older's inequality with $\frac{1}{1-\alpha} = \frac{1}{p} + \frac{1}{q}$ for $0<p<1$ and $q > 1$. Combining Proposition 4.9 and Proposition 4.12 in \citet{Berens2013Conditional}, we obtain an R\'enyi Entropy upper bound for the join distribution $\left(\frac{X}{\|X\|},\|X\|\right)$
    \begin{equation*}
        h_{1-p}\left(\frac{X}{\|X\|},\|X\|\right) \le h_{1-p}\left(\|X\|\right) + h_{0}\left(\frac{X}{\|X\|}\right).
    \end{equation*}
    The last term is bounded by a constant $\log C_d$ that only depends on the dimension because the direction vector $\frac{X}{\|X\|}$ has bounded supported on a unit sphere $S^{d-1}$ \citep{van2014renyi}. 
    Notice that $\|X\|\in \mathbb{R}$ is $\sqrt{d}\sigma$-sub-Gaussian since by H\"older's inequality,
    \begin{align*}
        \mathbb{E}\left[\exp\left(\frac{\|X\|^2}{d\sigma^2}\right)\right]
        \le \prod_{i=1}^d\left(\mathbb{E}\left[\exp\left(\frac{X_i^2}{\sigma^2}\right)\right]\right)^{\frac{1}{d}}
        \le 2.
    \end{align*}
    Therefore, using the result of one dimension in Lemma~\ref{Renyi_bound_1d},
    \begin{equation*}
        h_{1-p}(\|X\|) \le \frac{1}{p}\log \left(\left(2\pi d \sigma^2\right)^{\frac{p}{2}}  p^{\frac{p}{2}}(1-p)^{-\frac{1}{2}}\right).
    \end{equation*}
    The moments of $\|X\|$ can also be bounded since it is sub-Gaussian:
    \begin{align*}
        \mathbb{E} \left[\|X\|^{q(d-1)}\right] 
        &= \int_{0}^\infty \mathbb{P} \left( \|X\|^{q(d-1)} > u\right) du
        \\
        &= q(d-1) \int_{0}^\infty \mathbb{P} \left( \|X\| > t\right) t^{q(d-1)-1}dt
        \\
        &\le q(d-1) \int_{0}^\infty 2\exp\left(-\frac{t^2}{2d\sigma^2}\right) t^{q(d-1)-1}dt
        \\
        &= q(d-1) (2d\sigma^2)^{\frac{q(d-1)}{2}} \Gamma\left(\frac{q(d-1)}{2}\right)
        \\
        &\le 3 q(d-1) (2d\sigma^2)^{\frac{q(d-1)}{2}} \left(\frac{q(d-1)}{2}\right)^{\frac{q(d-1)}{2}}
    \end{align*}
    where in the last inequality we use $\Gamma(x)\le 3x^x$ for all $x\ge 1/2$.

    Combining these results and take $p=1-\frac{\alpha}{2}$ and $q=\frac{2}{\alpha}+\alpha-3$, when $\alpha$ is small we have
    \begin{equation*}
        h_{\alpha}(X) \le \log \left(\tilde C_d \sigma^{d} \alpha^{-\frac{d}{2}}\right)
    \end{equation*}
    for some constant $\tilde C_d$ that only depends on the dimension $d$.
\end{proof}
\subsubsection{Controlling the Expectation of Score function within Lower Density Areas}
For any random vector, we denote $I_m(X):=\mathbb{E}[\|s_X(X)\|^m]$, where $s_X(x):= \nabla \log p_X(x)$. For any random vector $X$, we can use the following proposition to control any order of moments of the score function $I_m(X)$.
\begin{lemma}\label{s_x_bound}
Let $X$ and $G$ be independent random vectors such that $X \sim p_0$, $G \sim \mathcal{N}(0,\boldsymbol{I}_d)$ and $Y = X + \sqrt{t} G$, then we have
\begin{align*}
I_m(Y)&\le \min\{I_m(X),I_m(\sqrt{t}G)\}
\\
&\le I_m(\sqrt{t}G)
\\
&\le t^{-\frac{m}{2}} d^{\frac{m}{2}} (m-1)!!
\end{align*}
\end{lemma}
\begin{proof}
The proof follows from a simple convexity argument of Corollary~3.2 in
\citet{bobkov2019moments}.\\
According to \citet{bobkov2019moments}[Corollary~3.2], for all independent random variables $X, Y$, $I_k(X+Y) \leq \min \left\{I_k(X), I_k(Y)\right\}$.
We can generalize this result to higher dimensions. Here we have 
\begin{align*} I_m(X) 
& =\int\|s_X(x)\|^m p_X(x) d x 
\\ 
& =\int \frac{\|\nabla p(x)\|^m}{p^{m-1}(x)} d x
\end{align*}
where the items are vectors instead of scalars. 
    In \citet{bobkov2019moments}, follows the fact that the homogeneous function $R(u, v)=u^k / v^{k-1}$ is convex on the plane $u \in \mathbb{R}, v\in \mathbb{R}$. We need to extend this to the convexity of $(u, v) \longmapsto \frac{\|u\|^m}{v^{m-1}}$ where $u$ is vector from $\mathbb{R}^d$ vector space. In fact, as in the constraints in classical convex optimization problems, we can check that the Hessian matrix for $u$ and $v$ are positive semi-definite:
\begin{align*}
\operatorname{det}\left(\nabla^2 \frac{\|u\|^m}{v^{m-1}}\right) 
&= \operatorname{det}\left(\begin{array}{cc}(m-2)(u \otimes u) + I & -(m-1) u^{\top} \\
-(m-1) u & m-1\end{array}\right) \cdot \frac{\|u\|^{2 m - 3}}{v^{2 m}} \cdot m^2\\
& \ge 0.
\end{align*}
Therefore, the function $(u, v) \longmapsto \frac{\|u\|^m}{v^{m-1}}$ is still convex within the $u \in \mathbb{R}^d, v \in \mathbb{R}$.

Next, we calculate the moment of scores for Gaussian distribution $I_m(\sqrt{t}G)$. Let $\sqrt{t}G \sim \phi_t(x):= (2\pi t)^{-\frac{d}{2}}\exp\left(-\frac{\|x\|^2}{2t}\right)$ and $Z \sim \mathcal{N}(0,1)$,
\begin{align*}
    I_m(\sqrt{t}G) &= \mathbb{E}_{X\sim \phi_t}\|\nabla \log \phi_t(X)\|^m 
    \\
    &= t^{-\frac{m}{2}}\,\mathbb{E}_{G\sim \phi_1}\|G\|^m 
    \\
    &\le t^{-\frac{m}{2}} d^{\frac{m}{2}}\, \mathbb{E}|Z|^m
    \\
    &\le t^{-\frac{m}{2}} d^{\frac{m}{2}} (m-1)!!.
\end{align*}
\end{proof}
\begin{lemma}\label{renyi_bound}
    Suppose that $X \sim p_t$ is $\sigma_t^2$-Sub-Gaussian, and $s_t(x):=\frac{\nabla p_t(x)}{p_t(x)}$ denoted the score function for $p_t(x)$. Then for any $0< \varepsilon < 1$,
    \begin{equation*}
        \int_{p_t(x) \leq \rho_n\log^{-c}{n}}\| s_t(x)\|^2\,p_t(x)dx \lesssim 
        \varepsilon^{-(1+\frac{d}{2})} t^{-1} \rho_n^{1-\varepsilon} \sigma^{d(1-\varepsilon)},
    \end{equation*}
    where $\rho_n = n^{-1}t^{-d/2}$.
\end{lemma} 
\begin{proof}
By H\"older's inequality,
\begin{align*}
    \int_{p_t(x) \leq \rho_n\log^{-c}{n}}\| s_t(x)\|^2\,p_t(x)dx &= 
    \mathbb{E}[\|s_t(X)\|^2\mathbbm{1}_{\{p_t(x) \le \rho_n\log^{-c}{n}\}}(X)]
    \\
    &\le \left(\mathbb{E}\left[\|s_t(X)\|^m\right]\right)^{2/m} \left(\mathbb{P}[p(X)\le \rho_n\log^{-c}{n}]\right)^{1-\frac{2}{m}}.
\end{align*}
The first term is the moments of score and can be upper bounded by Lemma~\ref{s_x_bound}:
\begin{align*}
    \left(\mathbb{E}\left[\|s_t(X)\|^m\right]\right)^{2/m} &\le \left(\mathbb{E}\left[\|s_t(X)\|^m\right]\right)^{2/m} 
    \\
    &\le t^{-1} d ((m-1)!!)^{2/m} \\
    &\lesssim m\, t^{-1},
\end{align*}
where in the last inequality we use Stirling's approximation. By Markov's inequality and the Upper bound for R\'enyi entropies in Lemma~\ref{Renyi_bound},
\begin{align*}
\mathbb{P}[p(X)\le \rho_n\log^{-c}{n}]
&= \mathbb{P}\left[\frac1{p^{1-\alpha}(X)}\ge \frac1{\rho_n^{1-\alpha}(\log n)^{-(1-\alpha)c}}\right]
\\
&\le \rho_n^{1-\alpha}(\log n)^{-(1-\alpha)c}\,\exp((1-\alpha)h_{\alpha}(Y))
\\
&\le C_d \rho_n^{1-\alpha}  \sigma^{d(1-\alpha)} \alpha^{-\frac{d}{2}}.
\end{align*}
So combining the results we have
\begin{align*}
    \int_{p_t(x) \leq \rho_n\log^{-c}{n}}\| s_t(x)\|^2\,p_t(x)dx 
    &\lesssim mt^{-1} \rho_n^{(1-\alpha)(1-2/m)} \sigma^{d(1-\alpha)(1-2/m)} \alpha^{-\frac{d(1-2/m)}{2}}.
\end{align*}
Finally, we take $\alpha = \varepsilon \downarrow 0$, $m = \varepsilon^{-1} \to\infty$ such that $(1-\alpha)(1-\frac{2}{m}) \asymp 1-\varepsilon$, then
\begin{align*}
    \int_{p_t(x) \leq \rho_n\log^{-c}{n}}\| s_t(x)\|^2\,p_t(x)dx 
    &\lesssim \varepsilon^{-(1+\frac{d}{2})} t^{-1} \rho_n^{1-\varepsilon} \sigma^{d(1-\varepsilon)}.
\end{align*}
\end{proof}
\subsection{Bounds on the Maximal Function}\label{appendix_maximal}
In the MSE analysis for $\hat p_t$ and $\nabla \hat p_t$ in Proposition~\ref{MSEphat} and Proposition~\ref{MSEphatprime}, we introduce the local maximal function $p_t^*(x):=\sup_{\|\lambda\|_{\infty} < h}p_t(x+\lambda)$.
The following lemma shows that $p^*_t(x)$ and $p_t(x)$ are uniformly close to each other for all $x$. Specifically, we demonstrate that if we take $h = \frac{\sqrt{t}}{D_n}$, then for any small $\varepsilon$, the ratio $\frac{p_t^*(x)}{p_t^{1-\varepsilon}(x)}$ is bounded by a constant that only depends on $\varepsilon$ for all $x$ and $t$.
\begin{lemma}\label{pstar_property}
    Define
    \begin{equation*}
        p_t^*(x) = \sup_{\|\lambda\|_{\infty} \le h} p_t(x + \lambda).
    \end{equation*}
    If we take $h = \frac{\sqrt{t}}{D_n}$, for some $D_n$ that depends on sample size $n$, then for any $t>0$, $x\in\mathbb{R}^d$, 
    and $\varepsilon\in(0,1)$,
    \begin{equation*}
        p_t(x) \leq p^*_t(x)\le (1-\varepsilon)^{-d/2} \exp\left\{\frac{d(1-\varepsilon)}{2D_n^{2}\varepsilon}\right\} \,p_t^{1-\varepsilon}(x).
    \end{equation*}
    In particular, by taking $D_n = C \sqrt{\log n}$ for some constant $C$ we have
    \begin{equation*}
        p_t(x) \leq p^*_t(x)\le (1-\varepsilon)^{-d/2} \exp\left\{\frac{d(1-\varepsilon)}{2C^2 (\log n)\varepsilon}\right\} \,p_t^{1-\varepsilon}(x).
    \end{equation*}  
\end{lemma}
\begin{remark}
    If we take $\varepsilon \rightarrow 0$ such that $\epsilon \log n \rightarrow +\infty$, we have $p_t(x) \le p^*_t(x)\lesssim p_t^{1-\varepsilon}(x)$.
\end{remark}
\begin{proof}[Proof of Lemma~\ref{pstar_property}]
The first inequality follows from the definition of $p^*_t(x)$. Next we prove the second inequality.
By Jensen's inequality,
\begin{equation*}
   p_t^{1-\varepsilon}(x) = (p_0 * \phi_t)^{1-\varepsilon}(x) \ge p_0 * \phi_t^{1-\varepsilon}(x) \ge p_0 * \phi_\frac{t}{1-\varepsilon}(x),
\end{equation*}
where $\phi_t(x)= \exp \left(-\frac{\|x\|_2^2}{2 t}\right)$. 
Then
\begin{equation*}
    \frac{p_t^*(x)}{p_t^{1-\varepsilon}(x)} \le \sup _{\substack{x \in \mathbb{R}^d \\ \|\lambda\|_{\infty} \le \frac{\sqrt{t}}{D_n}}} \frac{ p_0 * \phi_t(x+\lambda)}{p_0 * \phi_\frac{t}{1-\varepsilon}(x)}
\end{equation*}
By changing of variable and Lemma~\ref{lemma_supfg},
\begin{align*}
\frac{p_t^*(x)}{p_t^{1-\varepsilon}(x)}
&\le \sup _{\substack{x \in \mathbb{R}^d 
\\ \|\lambda\|_{\infty} \le D_n^{-1}}} \frac{\int p_0(x+\sqrt{t}\lambda - \sqrt{t}y)\phi_1(y)dy}{\int p_0(x - \sqrt{t} y)\phi_{\frac{1}{1-\varepsilon}}(y)dy}
\\ 
&= \sup _{\substack{x \in \mathbb{R}^d 
\\ \|\lambda\|_{\infty} \le D_n^{-1}}} \frac{\int p_0(y)\phi_1\left(\lambda + \frac{x-y}{\sqrt{t}}\right)dy}{\int p_0(y)\phi_{\frac{1}{1-\varepsilon}}\left(\frac{x-y}{\sqrt{t}}\right)dy} 
\\
&
\le \sup _{\substack{x \in \mathbb{R}^d 
\\ \|\lambda\|_{\infty}\le D_n^{-1}}}\frac{\phi_1\left(\lambda + x\right)}{\phi_{\frac{1}{1-\varepsilon}}\left(x\right)} \\
&= (1-\varepsilon)^{-d/2} \sup _{\substack{x \in \mathbb{R}^d \\ \|\lambda\|_{\infty} \le D_n^{-1}}} \exp\left\{\frac{1}{2}\left[(1-\varepsilon) \|x\|^2 -\|x+\lambda\|^2\right]\right\}.
\end{align*}
The exponent above is
\begin{align*}
\sup_{\substack{x \in \mathbb{R}^d \\ \|\lambda\|_{\infty} \le D_n^{-1}}} \,\left((1-\varepsilon) \|x\|^2 -\|x+\lambda\|^2\right) 
&= \sup_{\substack{x \in \mathbb{R}^d \\ \|\lambda\|_{\infty} \le D_n^{-1}}} \, \left(-\varepsilon\|x+\lambda/\varepsilon\|^2 + (1/\varepsilon - 1)\|\lambda\|^2\right) \\
&= \sup_{\|\lambda\|_{\infty} \le D_n^{-1}} \, (1/\varepsilon - 1) \|\lambda\|^2 \\
&= d(1/\varepsilon - 1)D_n^{-2}.
\end{align*}
Therefore,
denote $D_\varepsilon := \frac{p_t^*(x)}{p_t^{1-\varepsilon}(x)}$, then for all $x\in \mathbb{R}^d$
\begin{equation*}
    D_{\varepsilon} \le (1-\varepsilon)^{-d/2} \exp\left\{\frac{d(1-\varepsilon)}{2D_n^{2}\varepsilon}\right\}.
\end{equation*}
\end{proof}
\begin{lemma}\label{lemma_supfg}
    Suppose that $f(x)$, $g(x)$ and $p(x)$ are probability density functions. 
    Suppose that the density ratio of $f$ and $g$ is bounded, i.e., 
    \begin{equation*}
        \sup_x \frac{f(x)}{g(x)} < \infty.
    \end{equation*}
    Then
    \begin{equation*}
        \frac{\int_x f(x) p(x)\,dx}{\int_x g(x) p(x) \,dx} \leq \sup_x \frac{f(x)}{g(x)}.
    \end{equation*}
\end{lemma}
\begin{proof}
    Denote 
    \begin{equation*}
        \lambda = \sup_x \frac{f(x)}{g(x)} < \infty.
    \end{equation*}
    Then
    \begin{equation*}
        f(x) \le \lambda g(x).
    \end{equation*}
    Integrate both sizes we have
    \begin{align*}
        \int f(x) p(x)\,dx \le \lambda \int g(x) p(x)\,dx.
    \end{align*}
\end{proof}


\end{document}